\documentclass[12pt]{amsart}
\usepackage{amssymb, a4wide, mathdots, url,hyperref,graphicx}
\usepackage[usenames]{color}
\usepackage{enumerate}

\newcommand{\abs}[1]{\left|{#1}\right|}

\newcommand{\GL}{\operatorname{GL}}
\newcommand{\PGL}{\operatorname{PGL}}
\newcommand{\SL}{\operatorname{SL}}
\newcommand{\Sp}{\operatorname{Sp}}
\newcommand{\SO}{\operatorname{SO}}
\newcommand{\Ad}{\operatorname{Ad}}
\newcommand{\swrz}{\mathcal{S}}

\newcommand{\sprod}[2]{\left\langle{#1},{#2}\right\rangle}

\newcommand{\gen}{\operatorname{gen}}

\newcommand{\A}{\mathbb{A}}
\newcommand{\Z}{\mathbb{Z}}
\newcommand{\C}{\mathbb{C}}

\newcommand{\N}{\mathbb{N}}

\renewcommand{\H}{\mathcal{H}}

\newcommand{\bs}{\backslash}

\newcommand{\diag}{\operatorname{diag}}

\newcommand{\p}{\mathfrak{p}}

\newcommand{\vol}{\operatorname{vol}}
\newcommand{\proj}{\operatorname{p}}
\newcommand{\triv}{{\bf{1}}}

\newcommand{\J}{\mathcal{J}}
\renewcommand{\O}{\mathcal{O}}
\newcommand{\G}{\mathcal{G}}
\newcommand{\liesl}{{\mathfrak{sl}}}
\newcommand{\sm}[4]{\left(\begin{smallmatrix}{#1}&{#2}\\{#3}&{#4}\end{smallmatrix}\right)}

\newcommand{\rel}{{\operatorname{rel}}}
\renewcommand{\c}{\mathfrak{c}}

\newcommand\scalemath[2]{\scalebox{#1}{\mbox{\ensuremath{\displaystyle #2}}}}

\newcommand{\sq}{\mathfrak{s}}

\newcommand{\val}{{\operatorname{val}}}
\newcommand{\Span}{\operatorname{Span}}
\newcommand{\Kl}{\mathcal{K}}
\newcommand{\Cu}{\mathcal{C}}

\newtheorem{theorem}{Theorem}[section]
\newtheorem{lemma}[theorem]{Lemma}

\newtheorem{proposition}[theorem]{Proposition}%[subsection]
\newtheorem{corollary}[theorem]{Corollary}%[subsection]
\newtheorem{conjecture}[theorem]{Conjecture}%[subsection]
\theoremstyle{remark}
\newtheorem{remark}[theorem]{Remark}%[subsection]
\makeatletter
\@namedef{subjclassname@2020}{%
  \textup{2020} Mathematics Subject Classification}
\makeatother

\begin{document}

\title{On the Cubic Shimura lift to $\PGL(3)$:  The Fundamental Lemma}

\author{Solomon Friedberg}
\address{Department of Mathematics, Boston College, Chestnut Hill, MA 02467}
\email{solomon.friedberg@bc.edu}

\author{Omer Offen}
\address{Department of Mathematics, Brandeis University, Waltham MA 02453}
\email{offen@brandeis.edu}

\thanks{This work was supported by the NSF, grant numbers DMS-1801497 and DMS-2100206, by a Simons Fellowship in Mathematics, award number 816862 (Friedberg) and by a Simons Collaboration Grant, award number 709678 (Offen).}

\date{\today}

\begin{abstract} The classical Shimura correspondence lifts automorphic representations on the double cover of $\SL_2$ to 
automorphic representations on $\PGL_2$.   Here we take key steps towards establishing a relative trace formula
that would give a new global Shimura lift, from the triple cover of $\SL_3$ to $\PGL_3$, and also
 characterize the image of the lift.  The characterization would be through the nonvanishing of a certain global
 period involving a function in the space of the automorphic minimal representation $\Theta_{\SO_8}$ for split $\SO_8(\A)$,
 consistent with a 2001 conjecture of Bump, Friedberg and Ginzburg.
 In this paper, we first 
analyze a global distribution on $\PGL_3(\A)$ involving this period and show that it is a sum of factorizable orbital integrals.
The same is true for the Kuznetsov distribution attached to the triple cover of $\SL_3(\A)$.
We then match the corresponding local orbital integrals for the unit
 elements of the spherical Hecke algebras; that is, we establish the Fundamental Lemma.  

\end{abstract}

\maketitle

\goodbreak 

\tableofcontents

\goodbreak

\section{Introduction} \label{intro}

The classical Shimura correspondence \cite{MR0332663} lifts automorphic representations on the double cover of $\SL_2$ to automorphic representations on $\PGL_2$.  Niwa \cite{MR0364106} (working classically) and 
Waldspurger \cite{MR577010} (working adelically) showed that the map may be obtained as a theta lifting. Waldspurger 
also showed that the image of the lift is characterized by the nonvanishing of a certain period (the integral of an automorphic form over a cycle). Jacquet \cite{MR983610} then used his relative trace formula, 
which compares distributions on two different groups, one involving the relevant period, to give another proof of these facts.

Let  $n\geq2$ and let $G$ be a split connected reductive algebraic group over a global field $F$ with a full
set of $n$-th roots of unity $\mu_n\subset F$.  Let  $\A=\A_F$ be the ring of adeles of $F$. 
Then one may define an $n$-fold cover $\widetilde{G}^{(n)}(\mathbb{A})$ of the adelic points of $G$, and it is natural to ask if there is an analogue of the Shimura map for $\widetilde{G}^{(n)}(\mathbb{A})$. 
The local Shimura correspondence was investigated by Savin \cite{MR929534}, who proved (in generality) an isomorphism of Iwahori Hecke algebras.
Based on this work, for Cartan type A one expects a global functorial lift from genuine cuspidal automorphic representations on $\widetilde{\SL}_r^{(n)}(\A)$ to automorphic
representations on $\PGL_r(\A)$ if $n$ divides $r$, and a lift to automorphic representations on $\SL_r(\A)$ if $(n,r)=1$. (See also
 Bump, Friedberg and Ginzburg \cite{MR1882041}, Section 1, for a discussion of this.)   Such lifts have been studied since the 1980s,  but progress in establishing a global Shimura
 correspondence has been obtained only in the cases $n=2$ or $r=2$.
In this paper we establish the Fundamental Lemma for a relative trace formula
that will give the global Shimura lift in the case $n=r=3$; moreover this project will
characterize the image of the lift by means of a period involving a function in the space of the automorphic minimal representation for $\SO_8(\A)$, confirming a 2001 conjecture of Bump, Friedberg and Ginzburg.

To put this work on context, recall that 
Ginzburg, Rallis and Soudry \cite{MR1439552} observed that $\SL_2$ and its 3-fold cover $\widetilde{\SL}_2^{(3)}$ form a dual pair in the 3-fold cover of the exceptional group $G_2$, and used this to establish a 
lifting of genuine cuspidal automorphic representations on $\widetilde{\SL}_2^{(3)}(\A)$ to automorphic representations on $\SL_2(\A)$ and to determine its image.
Let $\text{Sym}^3:\SL_2\to \Sp_4$ be the symmetric cube map, and $\Theta_{\Sp_4}$ be the theta representation on the metaplectic double cover of $\Sp_4(\mathbb{A})$.  This cover
splits on the image of $\text{Sym}^3$. Then they showed that an irreducible
cuspidal automorphic representation $\tau$ of $\SL_2(\A)$ is in the image of the rank-one cubic Shimura map if and only if the period integral
\begin{equation}\label{SL2-period}
\int_{\SL_2(F)\backslash \SL_2(\A)} \varphi(g)\theta(\text{Sym}^3(g))\,dg
\end{equation}
is nonzero for some $\varphi$ in the space of $\tau$ and some $\theta$ in the space of $\Theta_{\Sp_4}$. Another proof of these results was given by Mao and Rallis \cite{MR1729444} via a relative trace formula.
However, for higher degree covers or higher rank special linear groups there is no similar dual pair.

By comparing orbital integrals, Flicker \cite{MR567194} succeeded in establishing
a Shimura-type correspondence for the $n$-fold cover of $\GL_2(\A)$ (there is more than one cover, and he treated a specific one); in this generality there is no known general characterization of the image by periods.
Kazhdan and Patterson \cite{MR840303} and Flicker and Kazhdan \cite{MR876160} studied these orbital integrals for higher rank general linear groups, but were not able to obtain sufficient control to establish a correspondence.
As for double covers,  Mao \cite{MR1605813} established
 the fundamental lemma for a relative trace formula for the double cover of $\GL_3$, and the case of the double cover of $\GL_n$ has recently been treated by Do 
by first establishing the equal characteristic case  \cite{MR3323346} and then using model theoretic methods to move to characteristic zero \cite{MR4176834}.
  
Here we are concerned with the conjectured Shimura map from the cubic cover of $\SL_3$ to $\PGL_3$.  
The question of characterizing its image was considered by
 Bump, Friedberg and Ginzburg \cite{MR1882041}.   
 Let $\Theta_{\SO_8}$ be the 
 automorphic minimal representation on the split special orthogonal group $\SO_8(\A)$. This representation
was constructed by Ginzburg, Rallis and Soudry \cite{MR1469105} as a multi-residue of a Borel Eisenstein series on $\SO_8$.  Let $\Ad$ denote the adjoint representation
 $\Ad:\PGL_3\to \SO_8$ (see Subsection~\ref{The Ad embedding} below). 
 Supposing that a Shimura lift exists in this case, they then conjectured
 \begin{conjecture}[Bump, Friedberg, Ginzburg]\label{BFG-conjecture} Let $\pi$ be an irreducible cuspidal automorphic representation of $\PGL_3(\A)$.  Then $\pi$ is in the image of the cubic Shimura correspondence from 
 $\widetilde{\SL}_3^{(3)}(\A)$ if and only if the period
\begin{equation}\label{PGL3-period}
\int_{\PGL_3(F)\backslash \PGL_3(\A)} \varphi(g)\,\theta(\Ad(g))\,dg
\end{equation}
 is nonzero for some $\varphi$ in the space of $\pi$ and some $\theta$ in $\Theta_{\SO_8}$.
 \end{conjecture}
Those authors presented two pieces of evidence for this conjecture, the first from finite fields, and the second by supposing 
that $\pi$ was not cuspidal but rather an Eisenstein series induced from cuspidal data $\tau$ on $\GL_2(\A)$, formally unfolding
the period \eqref{PGL3-period} in this case, and showing that the resulting integral is nonvanishing for some choice of data if and only if the period \eqref{SL2-period} for $\tau$ is nonvanishing for some choice of data.  

The study of periods in the context of Langlands type correspondences is a main theme of contemporary research; see for example
Sakellaridis and Venkatesh \cite{MR3764130}.  However, the extension to metaplectic covers is not well developed.  In particular, 
we do not know of an 
 extension of Conjecture~\ref{BFG-conjecture} to higher rank special linear groups or higher degree covers.
 
 Here we take key steps towards establishing the existence of the global Shimura 
 correspondence with $n=r=3$ and towards characterizing its image as in Conjecture~\ref{BFG-conjecture} above.  
 Our approach, following Jacquet, Mao and Rallis, and Mao, is to establish 
a comparison of relative trace formulas.
For an algebraic group $H$ defined over $F$ we denote its automorphic quotient space by $[H]=H(F)\bs H(\A)$. 
For an affine variety $X$ defined over $F$ denote by $\swrz(X(\A))$ the space of Schwartz-Bruhat functions on $X(\A)$.
The action of $f\in \swrz(H(\A))$ on $L^2([H])$
obtained by averaging the right regular representation $$R(f)\phi(h)=\int_{H(\A)} \phi(hy) f(y)\,dy,$$ is realized by the kernel function
$$\sum_{\gamma\in G(F)} f(h_1^{-1}\gamma h_2),\quad h_1,h_2\in [H].$$
The relative trace formula compares two distributions obtained as 
integrals of these kernel functions for different groups, one integral involving the period and a Whittaker character and the other involving solely
Whittaker characters.

Let $G=\PGL(3)$ considered as an algebraic group defined over $F$ and $N$ the standard maximal unipotent subgroup of $G$ realized as the 
group of upper triangular $3\times 3$ unipotent matrices.  Fix a non-trivial character $\psi$ of $F\bs \A$. 
By abuse of notation, we further denote by $\psi$ the generic character $\psi(n)=\psi(n_{1,2}+n_{2,3})$ of $[N]$.
For $\Theta\in \Theta_{\SO_8}$ we consider the distribution $I(\Theta)$ on $G(\A)$ defined by
\begin{equation}\label{PGL3-distribution}
I(f,\Theta)= \int_{{[N]}} \int_{[G]} \Big\{\sum_{\gamma\in G(F)} f(g^{-1}\gamma n) \Big\}\Theta(\Ad(g)) \psi(n) \ dg\ dn, \ \ \ f\in \swrz(G(\A)).
\end{equation}

Let $G'(\A)$ be the adelic $3$-fold metaplectic cover of $\SL_3(\A)$. It is a central extension of $\SL_3(\A)$
satisfying the exact sequence
\[
1\rightarrow \mu_3 \rightarrow G'(\A)\rightarrow \SL_3(\A)\rightarrow 1.
\] 
(Since $G'(\A)$ is not the adelic points of an algebraic group this is an abuse of notation.)
Denote by $\swrz(G'(\A))$ the space of genuine (that is, $\mu_3$-equivariant) Schwartz-Bruhat functions on $G'(\A)$.
There is a splitting of $\SL_3(F)$ in $G'(\A)$ and we denote by $G'(F)$ its image. 
The group $N(\A)$ also splits in $G'(\A)$ and we continue to denote by $N(\A)$ the image of this splitting. 
The Kuznetsov trace formula is realized by the distribution on $G'(\A)$ defined by
\begin{equation}\label{metaplectic-distribution}
J(f')=\int_{[N]}\int_{[N]} \Big\{\sum_{\gamma \in G'(F)} f'(n_1\gamma n_2)\Big\}\psi(n_1n_2)\ dn_1\ dn_2, \ \ \ f'\in \swrz(G'(\A)).
\end{equation}
Our goal is to study a comparison between the two distributions $I(f,\Theta)$ and $J(f')$. 

The first step is to \emph{geometerize} the distribution $I(\Theta)$, that is, express it as a sum (over some geometric orbits) of distributions that are factorizable. 
In most relative trace formulas this step is straightforward: one writes the sums using the Bruhat decomposition or some variant and unfolds to get a sum over double cosets of adelic 
orbital integrals, with only
the relevant double cosets contributing. For factorizable test functions, each side is then factorizable.

The obstacle in our case is the automorphic minimal representation that appears in the distribution \eqref{PGL3-distribution}. In contrast to the rank one case, this
representation is not directly obtained from the Weil representation, so there is no easy way to express the functions in $\Theta_{\SO_8}$ as sums that may then be unfolded.
To address this, we make use of the dual pair $(\SL_2,\SO_8)$ inside $\Sp_{16}$, and realize the automorphic minimal representation $\Theta_{\SO_8}$ as the
theta lift of the trivial representation from $\SL_2(\A)$
to $\SO_8(\A)$.  This realization is due to Ginzburg, Rallis and Soudry \cite[Theorem 6.9]{MR1469105}; their proof
relies on the work of Kudla and Rallis \cite{MR1289491} (or in a classical language over $\mathbb{Q}$, the work of Deitmar and Krieg \cite{MR1128710}).
When applied to the groups at hand, these works establish that the suitably regularized theta lift of an Eisenstein series on $\SL_2$ is an Eisenstein series on $\SO_8$;
then taking a residue, one obtains the desired realization of $\Theta_{\SO_8}$. Since we are aiming for a factorizable integral, it is perhaps surprising that
it is helpful to introduce an additional integration over an automorphic quotient $[\SL_2]$.  Nonetheless, it turns out that this
realization allows us to do an unfolding and to express the distribution
$I(f,\Theta)$ as a sum of factorizable orbital integrals.

 An analogous geometric expansion for $J(f')$ follows directly from the Bruhat decomposition on $G'$ 
(and is well-known).  
This allows us to establish a 
one-to-one correspondence for the relevant orbits of the two distributions.  To effect a global comparison, 
we must then achieve a comparison of the corresponding local orbital integrals for each family of relevant orbits up to a transfer factor whose product
over all places is $1$.  

The main local result of this paper is a matching of these orbital integrals for the unit element of the spherical Hecke algebra, 
that is, the Fundamental Lemma for this relative trace formula.
The calculations are rather elaborate, but as in prior work on relative trace formulas that give Shimura correspondences, there is an algebraic fact that is key.
In Jacquet's work on the Shimura correspondence, this was the computation of a certain Sali\'e sum; this fact was also 
key in Iwaniec's work \cite{MR870736} on the same topic
with an eye towards analytic number theoretic applications.  For the work of Mao and Rallis, key use was made of an identity relating a 
cubic exponential sum to a Kloosterman
sum with a cubic character that is due to Duke and Iwaniec \cite{MR1210520} (as these authors remark, the analogous fact at a real 
archimedean place is Nicholson's formula
for the Airy integral!).  The proof of that result in turn is based on the Davenport-Hasse relation.  Their result applies when the additive 
character is of conductor 1, that is, to exponential
sums of the form
$$\sum_{x\bmod p} \exp\left(2 \pi i \frac{ax^3+bx}{p}\right)$$
 where $p$ in our computation is the cardinality of the residue field. These sums are shown to be equal to certain Kloosterman sums modulo $p$ with a cubic character.
 Remarkably, we observe that a similar relation is true for exponential sums that involve additive characters of higher conductor (Corollary~\ref{cor di}). 
 For higher conductor, the proof
 is based on the method of stationary phase. 
 
 In our work the orbital integrals for the big cell frequently reduce to integrals of pairs of Kloosterman integrals with cubic characters.  
 These integrals appear intractable. However, 
 using this identity twice to make
 them into integrals of pairs of cubic exponential integrals, we are able to effect the desired comparison.
 It also sometimes happens that we encounter Kloosterman integrals without a cubic character.  In this case for conductor 1 the result of
 Duke-Iwaniec is not applicable, but another identity, which
expresses the characteristic function of the local condition 
 $(ab^{-1})^3\equiv -1 \bmod \p$ in terms of a sum
of three cubic exponential integrals, is used instead (Lemma \ref{lem i=1 3val}).
  
We describe the contents of this paper. The first part of this paper gives the geometric expansions of the distributions we have described above.
Our first task is to express the distribution $I(f,\Theta)$ as a sum of factorizable orbital integrals. 
 In Section~\ref{sec unfold}, we formulate a general unfolding principle and recall the definition of relevant orbits. Then,  in Section~\ref{sec-unfolding-I},
 we use this principle and the theta lift to $\SO_8$ to unfold the distribution $I(f,\Theta)$ 
 and to express it
 as a sum of factorizable orbital integrals over relevant orbits  (Proposition~\ref{eq unfold}). In the following Section~\ref{relevant-I-orbits}, we 
 work with flag varieties to
determine these relevant orbits explicitly (Theorem~\ref{I-relevant-orbits}) and also write the orbital integrals for each relevant orbit in a form amenable for computation.  
  We also need the geometric expansion for the Kuznetsov distribution  $J(f')$ and this is presented in Section~\ref{J-distribution-cals}.
  This completes the reduction of the comparison of spectral sides of the global distributions to a comparison of a collection of explicitly given geometric local distributions.

The second part of this paper establishes the Fundamental Lemma.  We begin with the big cell comparison,  which is stated in Theorem~\ref{thm main},
and whose proof occupies the following four sections.  In Section~\ref{yummy-ingredients}, we develop the necessary local ingredients, including the comparison
of cubic exponential and Kloosterman integrals, which is given (as a consequence of prior computations that are closely related to the method
of stationary phase) in Subsection~\ref{cubic-comparison-123}.  The companion computation of cubic integrals when Duke-Iwaniec does not apply 
is in Subsection~\ref{cubic-additional-123}.
The orbital integral $J(a,b)$ attached to the big cell is then evaluated in Section~\ref{The-integral-J(a,b)}.  This is the piece of the calculation that
requires computations in the metaplectic group. Specifically, if $\O$ denotes the local ring of integers then
the embedding of $\SL_3(\O)$ in the local covering group is given by means of a section
that must be computed in order to work with the unit element in the corresponding Hecke algebra. 
In doing so we make use of an algorithm for computing the local two-cocycle (defined by Matsumoto \cite{MR240214})
that is given in
Bump and Hoffstein \cite{MR904946}.  This leads to the complicated orbital integrals involving pairs of Kloosterman
integrals that we then evaluate.  The
orbital integral $I(a,b)$ is evaluated in Section~\ref{The-integral-I(a,b)}.  The comparison is then completed in the brief Section~\ref{The-big-cell-comparison-is-done}.
The final section, Section~\ref{Comparison for smaller cells}, matches the remaining orbital integrals with an explicit transfer factor. For the two one-parameter families
of relevant orbits, the
comparison once again requires replacing a Kloosterman integral by a cubic exponential integral, but this time only once.

\section*{\it{Part I. Global Theory: The Geometric Expansions}}\label{Global theory}
\section{An unfolding principle}\label{sec unfold}

Let $F$ be a global field, $\A=\A_F$ the ring of adeles of $F$,  
$X$ an affine $F$-variety and $G$ an affine algebraic $F$-group acting on $X$ on the right. 
We denote the action by $(x,g)\mapsto x^g$, $x\in X$, $g\in G$. Let $\delta_G$ be the modular quasicharacter of $G(\A)$ and $[G]=G(F)\bs G(\A)$. 

For $\xi\in X(F)$ we denote by $G_\xi$ the stabilizer of $\xi$ in $G$, an affine algebraic group defined over $F$.
We say that a function $\c:X(\A)\times G(\A)\rightarrow \C$ is an \emph{automorphic cocycle} if it satisfies the following two conditions:
\begin{itemize}
\item $\c(\xi,\gamma)=1$, $\xi\in X(F)$, $\gamma\in G(F)$;
\item $\c(\xi, gh)=\c(\xi,g)\c(\xi^g,h)$, $\xi\in X(F)$, $g,\,h\in G(\A)$. 
\end{itemize}
This implies that $g\mapsto \c(\xi,g)$ restricts to an automorphic character of $G_\xi(\A)$ for every $\xi\in X(F)$. 

Assume that $\c$ is an automorphic cocycle such that for every $\xi\in X(F)$ 
\begin{equation}\label{eq int cond}
\int_{[G_\xi]} \abs{\c(\xi,h)}\delta_G(h)\delta_{G_\xi}(h)^{-1}\ dh<\infty.
\end{equation}
Consequently,  
\begin{equation}\label{eq char int}
\int_{[G_\xi]} \c(\xi,h)\delta_G(h)\delta_{G_\xi}(h)^{-1}\ dh=\begin{cases} \vol([G_\xi]) & (\c(\xi,\cdot)\delta_G)|_{G_\xi(\A)}\equiv \delta_{G_\xi} \\ 0 & \text{otherwise.} \end{cases}
\end{equation}
\begin{remark}\label{rmk nec cond} Note that if $[G_\xi]$ is compact for every $\xi\in X(F)$ then the assumption \eqref{eq int cond} clearly holds. If $G(\A)$ and $G_\xi(\A)$ are unimodular and $\c(\xi,\cdot)|_{G_\xi(\A)}$ is a unitary character then for \eqref{eq int cond} to hold it is enough to assume that $[G_\xi]$ is of finite volume.
\end{remark}
\begin{remark}\label{rmk twist cocyc}
Note that if $\c:X(\A)\times G(\A)\rightarrow \C$ is an automorphic cocycle and $\chi$ is an automorphic character of $G(\A)$ then the function $\c_\chi$ defined by $\c_\chi(x,g)=\chi(g)\c(x,g)$ is also an automorphic cocycle. 
\end{remark}

We say that $\xi\in X(F)$ is \emph{relevant} if $\c(\xi,\cdot)\delta_G\delta_{G_\xi}^{-1}$ is trivial on $G_\xi(\A)$. 
By our assumptions, for $\xi\in X(F)$, $\gamma\in G(F)$ and $h\in G_\xi(\A)$ we have $\c(\xi^\gamma,\gamma^{-1}h\gamma)=\c(\xi, h)$. It is also clear that $\delta_G(\gamma^{-1}h\gamma)=\delta_G(h)$ and $\delta_{G_{\xi^\gamma}}(\gamma^{-1}h\gamma)=\delta_{G_\xi}(h)$. Therefore $\xi$ is relevant if and only if $\xi^\gamma$ is relevant. We denote by $(X/G)$ a set of representatives for the $G(F)$-orbits in $X(F)$ and by $(X/G)_\rel$ the subset of relevant elements in $(X/G)$. 
We note the following formal unfolding principle. At this generality we ignore convergence issues, hence the term principle. For $\phi\in \swrz(X(\A))$, formally, we have
\begin{equation}\label{eq unfolding}
\int_{[G]} \sum_{\xi\in X(F)} \phi(\xi^g) \c(\xi,g)\ dg =\sum_{\xi\in (X/G)_\rel} \vol([G_\xi]) \int_{G_\xi(\A)\bs G(\A)} \phi(\xi^g)\c(\xi,g)\ dg.
\end{equation}
Indeed, let $(X/G)$ be a set of representatives for the $G(F)$-orbits in $X(F)$. Then 
\[
\sum_{\xi\in X(F)} \phi(\xi^g) \c(\xi,g)=\sum_{\xi\in (X/G)}\sum_{\gamma\in G_\xi(F)\bs G(F)} \phi(\xi^{\gamma g}) \c(\xi^\gamma,g)=\sum_{\xi\in (X/G)}\sum_{\gamma\in G_\xi(F)\bs G(F)} \phi(\xi^{\gamma g}) \c(\xi,\gamma g).
\]
On the left hand side of \eqref{eq unfolding} we exchange the order of integration over $[G]$ and summation over $(X/G)$ and then unfold $\gamma$ with $g$ to obtain
\[
\sum_{\xi\in (X/G)} \int_{G_\xi(F)\bs G(\A)} \phi(\xi^g)\c(\xi,g)\ dg.
\]
Integrating in stages this equals
\begin{multline*}
\sum_{\xi\in (X/G)} \int_{G_\xi(\A)\bs G(\A)} \phi(\xi^g)   \int_{[G_\xi]}\c(\xi,hg)\delta_G(h) \delta_{G_\xi}(h)^{-1}\ dh\ dg=\\ \sum_{\xi\in (X/G)} \int_{G_\xi(\A)\bs G(\A)} \phi(\xi^g)\c(\xi,g) \ dg \int_{[G_\xi]}\c(\xi,h)\delta_G(h) \delta_{G_\xi}(h)^{-1} \ dh.
\end{multline*}
It follows from \eqref{eq char int} that this equals the right hand side of \eqref{eq unfolding}.

\section{The geometrization of the distribution \texorpdfstring{$I(f,\Theta)$}{I(f,theta)}}\label{sec-unfolding-I}

Throughout the rest of Part I we suppose that the global field $F$ contains a 
primitive cube root of unity $\rho$ (that is, there exits $\rho\in F$ such that $\rho^2+\rho+1=0$), and we set
$\mu_3=\langle\rho\rangle$.  Our goal is to geometrize the distribution $I(f,\Theta)$ given by \eqref{PGL3-distribution}.

\subsection{The Adjoint embedding}\label{The Ad embedding}
For $n\geq2$ let $w_n$ denote the $n\times n$ antidiagonal matrix $w_n=(\delta_{i,n+1-j})\in \GL_n$.
We realize the Adjoint representation of $\PGL(3)$ as the map $\PGL(3)\to \GL(\liesl(3))$ defined by conjugation.  Considering $\liesl(3)$ as a quadratic
space with respect to the bilinear form $<x,y>=\text{Tr}(xy)$, the image lies in the special orthogonal group $\SO(\liesl(3))$.
With respect to the basis
$$\{e_{1,3},e_{2,3},e_{1,2},e_{1,1}-e_{2,2},e_{2,2}-e_{3,3},e_{2,1},e_{3,2},e_{3,1}\}$$
of $\liesl(3)$ where the $e_{i,j}$ are the standard elementary matrices, $\SO(\liesl(3))$ is isomorphic to $\SO(J)$ where
$$J=\begin{pmatrix}&&&w_3\\&2&-1&\\&-1&2&\\w_3&&&\end{pmatrix}.
$$

Since $F$ contains the cube roots of unity, the group $\SO(J)$ is split over $F$ and may be conjugated to $\SO_8:=\SO(w_8)$; indeed 
if $g_0=-\tfrac13 \left(\begin{smallmatrix}1-\rho&2+\rho\\-(1+2\rho)&1+2\rho\end{smallmatrix}\right)$ then
$^t g_0\left( \begin{smallmatrix}2&-1\\-1&2\end{smallmatrix} \right)g_0=\left(\begin{smallmatrix}&1\\1&\end{smallmatrix}\right)$.
In this way we realize the adjoint map as a homomorphism $\Ad: \PGL(3)\rightarrow \SO_8$. This is the map that is used in the period \eqref{PGL3-period}.
Restricted to $N(\A)$ this leads to
the expression
\begin{multline*}\text{Ad}\begin{pmatrix}1&x&z\\&1&y\\&&1\end{pmatrix}=\\
\begin{pmatrix}1&x&-y&-(xy+\rho z)&-(xy+\rho^2z)&x(xy-z)&-zy&z(xy-z)\\&1&0&\rho^2 y&\rho y&xy-z&-y^2&y(xy-z)\\
&&1&x&x&-x^2&z&-zx\\&&&1&0&-x&-\rho y&\rho xy+\rho^2 z\\&&&&1&-x&-\rho^2y&\rho^2xy+\rho z\\&&&&&1&0&y\\&&&&&&1&-x\\&&&&&&&1
\end{pmatrix}.
\end{multline*}
It was also explicated in \cite{MR1882041}.

\subsection{First steps}
Since $\Theta$ is automorphic and $\Ad$ is defined over $F$ we may combine the integral over $[G]$ with the sum over $G(F)$ to obtain
\[
I(f,\Theta)=\int_{[N]}\int_{G(\A)} f(g^{-1}n) \Theta(\Ad(g)) \psi(n)\ dg\ dn
\]
and after the change of variables $g\mapsto ng$ we obtain that 
\begin{equation}\label{eq crude form I}
I(f,\Theta)=\int_{[N]} \{\int_{G(\A)} f(g^{-1})\Theta(\Ad(n)\Ad(g))\ dg\}\psi(n)\ dn=\int_{[N]} \Theta_f(\Ad(n))\psi(n)\ dn
\end{equation}
where $\Theta_f\in \Theta_{\SO_8}$ is defined by $\Theta_f(x)=\int_{G(\A)} f(g^{-1}) \Theta(x\Ad(g))\ dg$.

This means, that in order to obtain a geometric expansion, we are reduced to showing that
\[
\int_{[N]}\Theta(\Ad(n))\psi(n)\ dn
\] 
is a sum of factorizable integrals.

As noted in the Introduction, by \cite[Theorem 6.9]{MR1469105}, the minimal representation $\Theta_{\SO_8}$ can be expressed as a theta lift using the Weil representation of the 
metaplectic double cover $\widetilde{\Sp}_{16}(\A)$ of $\Sp_{16}(\A)$.  We now make this explicit, and use this to show the desired expansion.
To use the theta lift, we follow Deitmar and Krieg (\cite{MR1128710}, Section 2) and Kudla-Rallis (\cite{MR1289491}, Section 5) by imposing a condition on $f$ (see Subsection~\ref{theta-lift-done-here} below) at one archimedean place.
With this condition, the theta integral of the Eisenstein series is convergent, and so the computation applies.

\subsection{On the Weil representation}
We realize the group $\Sp_{16}$ as the group of automorphisms preserving the alternating matrix $\sm{0}{w_8}{-w_8}{0}$. For $g\in \GL_n$ let $g^*=w_n {}^t g^{-1} w_n$.

For a matrix $Y\in M_n(\A)$ let ${}_tY=w_n {}^tY w_n$. The equality $Y={}_tY$ means that $Y$ is symmetric with respect to the second diagonal, i.e., $Y=(y_{i,j})$ where $y_{i,j}=y_{n+1-j,n+1-i}$ (or equivalently, that $w_n Y$ is symmetric).
It follows from the standard explicit formulas for the Weil representation $\omega_\psi$ of $\widetilde{\Sp}_{16}(\A)$ (see e.g. \cite{MR2848523}) that for $x\in \A^8$ and $\phi\in \swrz(\A^8)$ we have
\[
\omega_\psi\sm{g}{}{}{g^*} \phi(x)=\phi(xg), \ \ \ g\in \SL_8(\A)\]
and
\[
\omega_\psi\sm{I_8}{Y}{}{I_8}\phi(x)=\psi(xw_8 {}^tY\ {}^t x)\phi(x)=\psi(xYw_8\ {}^t x)\phi(x),\ \ \ Y={}_t Y\in M_8(\A). 
\]
Combined this gives
\begin{equation}\label{eq weil}
\omega_\psi(\sm{I_8}{Y}{}{I_8} \sm{g}{}{}{g^*})\phi(x)=\psi(x  Yw_8\,  {}^t x)\phi(xg).
\end{equation}
Let $P=M\ltimes U$ be the Siegel parabolic subgroup of $\Sp_{16}$ with unipotent radical $U=\{\sm{I_8}{Y}{}{I_8}: {}_tY=Y\in M_8\}$ isomorphic to the vector space of symmetric matrices in $M_8$ and $M=\{\diag(g,g^*):g\in \GL_8\}$. Let $P^\circ=M^\circ \ltimes U$ where $M^\circ= \{\diag(g,g^*):g\in \SL_8\}$ and let $\proj_M:P\rightarrow M$ be the projection map to the Levi subgroup. 
We consider the right action of $P^\circ$ on affine $8$-space $\mathbb{G}_a^8$ defined by
\[
x^p=xg, \ \ \ x\in \mathbb{G}_a^8,\ p\in P^\circ,\ \proj_M(p)=\diag(g,g^*)
\]
and the function $\c:\A^8 \times P^\circ(\A)\rightarrow \C$ defined by
\[
\c(x,p)=\psi(xYw_8 \,{}^t x), \ \ \ x\in \A^8,\ p=\sm{I_8}{Y}{}{I_8}\proj_M(p)\in P^\circ(\A).
\]
We can rewrite \eqref{eq weil} as
\begin{equation}\label{eq weil cocycle form}
\omega_\psi(p)\phi(x)=\c(x,p)\phi(x^p), \ \ \ x\in \A^8,\ p\in P^\circ(\A).
\end{equation}
We denote by $P^\circ_\xi$ the stabilizer in $P^\circ$ of $\xi\in F^8$ (considered as an algebraic group over $F$) and by $P^\circ(\A)_x$ the stabilizer in $P^\circ(\A)$ of $x\in \A^8$. Note that
\[
\c(\xi,\gamma)=1, \ \ \ \xi\in F^8, \gamma\in P^\circ(F).
\]

\begin{lemma}\label{lem cocyc}
The function $\c$ satisfies the cocycle condition
\[
\c(x,p_1p_2)=\c(x,p_1)\c(x^{p_1},p_2),\ \ \ x\in \A^8,\,p_1,\,p_2\in P^\circ(\A).
\]
In particular, it is an automorphic cocycle. 
\end{lemma}
\begin{proof}
Write $p_i=\sm{I_8}{Y_i}{}{I_8} \sm{g_i}{}{}{g_i^*}$, $i=1,2$. Then $p_1p_2=\sm{I_8}{Y_1+g_1Y_2(g_1^*)^{-1}}{}{I_8} \sm{g_1g_2}{}{}{(g_1g_2)^*}$. Since $g_1Y_2(g_1^*)^{-1}w_8=g_1Y_2 w_8\,{}^t g_1$ it follows that 
\[
\c(x,p_1p_2)=\psi(x Y_1 w_8 \,{}^t x+ xg_1 Y_2w_8 \, {}^t (x g_1))=\c(x,p_1)\c(x g_1,p_2)=\c(x,p_1)\c(x^{p_1},p_2)
\] 
as required. 
\end{proof}

\subsection{On the theta lift}\label{theta-lift-done-here}
We fix an embedding of $\SL_2\times\SO_8$ in $\Sp_{16}$ such that $\SL_2\times \Ad(N)$ embeds in $P^\circ$ (we later explicate such an embedding)  
and denote by $(g,h)\mapsto [g,h]$ the associated embedding of $\SL_2(\A)\times \SO_8(\A)$ in $\widetilde\Sp_{16}(\A)$. Note that the action of the Weil representation directly shows that $P^\circ(\A)$ splits in $\widetilde\Sp_{16}(\A)$.
Let 
\[
\tilde\theta_\phi(g)=\sum_{x\in F^8}\omega_\psi(g)\phi(x),\ \ \ g\in \widetilde\Sp_{16}(\A).
\]

There is a subspace $\swrz^\circ(\A^8)$ of $\swrz(\A^8)$ (the image of $\swrz(\A^8)$ under the Casimir operator in the center of the universal enveloping algebra of the Lie algebra of $\SL_2$ at one fixed archimedean place) such that for $\phi\in \swrz^\circ(\A^8)$ we have that $g\mapsto \tilde\theta_\phi([g,h])$
is a rapidly decreasing automorphic function on $\SL_2(\A)$. For $\phi\in \swrz^\circ(\A^8)$ let
\[
\Theta_\phi(g)=\int_{[\SL_2]} \tilde\theta_\phi([h,g])\ dh, \ \ \ g\in \SO_8(\A).
\]
We then have
\[
\Theta_{\SO_8}=\{\Theta_\phi:\phi\in \swrz^\circ(\A^8)\}.
\]
We consider the distributions on $G(\A)$ defined by
\[
I(f,\phi)=\int_{[N]}\int_{G(\A)} f(g^{-1}) \int_{[\SL_2]}\tilde \theta_\phi([h,\Ad(ng)])\ dh\ dg\ \psi(n)\ dn, \ \  f\in \swrz(G(\A)), \ \  \phi\in \swrz^\circ(\A^8).
\]

Note that $I(f,\phi)=I(f,\Theta_\phi)$.
Let 
\[
\phi[f](x)=\int_{G(\A)} f(g^{-1})\omega_\psi([I_2,\Ad(g)])\phi(x)\ dg.
\]
Then
\begin{multline*}
I(f,\phi)=\int_{[N]}\int_{[\SL_2]} \tilde \theta_{\phi[f]}([h,\Ad(n)])\ dh\ \psi(n)\ dn=\\ \int_{[N]}\int_{[\SL_2]} \sum_{\xi\in F^8} \omega_\psi([h,\Ad(n)])\phi[f](\xi)\ dh\ \psi(n)\ dn.
\end{multline*}
We therefore need to unfold integrals of the form 
\begin{equation}\label{eye-of-phi}
I(\phi)=\int_{[N]}\int_{[\SL_2]} \sum_{\xi\in F^8} \omega_\psi(h,\Ad(n))\phi(\xi)\ dh\ \psi(n)\ dn.
\end{equation}
\subsection{The unfolding of $I$}
Let $H=\SL_2\times N$. For $h=(g,n)\in H(\A)$, by abuse of notation we write $[h]=[g,\Ad(n)]$ and we extend $\psi$ to an automorphic character of $H(\A)$ (still denoted by $\psi$) that is trivial on $\SL_2(\A)$.
With our choice of embedding of $H(\A)$ and by \eqref{eq weil cocycle form} we have
\[
I(\phi)=\int_{[H]} \sum_{\xi\in F^8} \phi(\xi^{[h]})\ \c_{\psi}(\xi,[h])\ dh
\]
where $\c_\psi(\xi,[g,\Ad(n)])=\psi(n)\c(\xi,[g,\Ad(n)])$.

Let $\Xi_\rel$ be a set of representatives for the $H(F)$-orbits in $F^8$ for which $\c_\psi(\xi,\cdot)|_{H_\xi(\A)}\equiv 1$. 
We will later see that $H_\xi(\A)$ is unimodular and $[H_\xi]$ has finite volume. Indeed, $H_\xi$ is unipotent unless $\xi=0$ in which case $H_\xi=H$. As a consequence of Lemma \ref{lem cocyc} and Remarks \ref{rmk nec cond} and \ref{rmk twist cocyc} we can apply the unfolding principle of Section~\ref{sec unfold} to obtain
\begin{proposition}\label{eq unfold} The distribution $I$ given in \eqref{eye-of-phi} is a sum of factorizable orbital integrals
\begin{equation*}
I(\phi)=\sum_{\xi\in \Xi_\rel} \O(\xi,\phi)
\end{equation*}
where for $\xi\in F^8$ relevant, the associated orbital integral is given by
\[
\O(\xi,\phi)= \int_{H_\xi(\A)\bs H(\A)} \phi(\xi^{[h]}) \c_\psi(\xi,[h])\ dh.
\]
Here the measures are normalized so that the volume of each compact quotient $[H_\xi]$ is one. \qed
\end{proposition}

In order to compare the distributions $I$ and $J$, we must describe the quantities in this proposition precisely.  This is accomplished in Theorem~\ref{I-relevant-orbits} below.

\goodbreak

\section{Determination of the relevant orbits and orbital integrals for \texorpdfstring{$I$}~}{}\label{relevant-I-orbits}

\subsection{Explication of the set up}
Let $V=\Span(e,f)$ be the two dimensional symplectic space with $\sprod{e}{f}=1=-\sprod{f}{e}$ and $W=\Span (e_1,\dots, e_8)$ with $\sprod{e_i}{e_j}=\delta_{i,9-j}$. Fixing the basis $\{e,f\}$  identifies $\SL_2$ with $\Sp(V)$ and the basis $\{e_1,\dots,e_8\}$ identifies $\SO_8$ with $\SO(W)$. 
We also consider $V\otimes W$ as a symplectic space with the product form. Fixing the basis 
\[
\{e\otimes e_1 ,f\otimes e_1,\dots, e\otimes e_4,f\otimes e_4, -e\otimes e_5,f\otimes e_5,\dots,  -e\otimes e_8, f\otimes e_8\}
\] 
identifies $\Sp_{16}$ with $\Sp(V\otimes W)$. Via the tensor product embedding $\Sp(V) \times \SO(W) \hookrightarrow \Sp(V\otimes W)$ this defines an embedding of 
\[
\SL_2 \times \SO_8 \hookrightarrow \Sp_{16}.
\]  
We explicate this embedding.

Let $\lambda:\GL_2\rightarrow \GL_8$ be the embedding $\lambda(h)=\diag(h,h,h,h)$. Note that $\lambda({}^t h)={}^t\lambda(h)$ and $\lambda(h)^*=\lambda(h^*)$. 
The embedding of $\SL_2$ in $\Sp_{16}$ is then $h\mapsto \diag(\lambda(h),\lambda(h^*))$. 

The embedding of $\SO_8$ in $\Sp_{16}$ can be described as follows. Let $\imath:M_4\rightarrow M_8$ be the standard `tensor with $I_2$ map', that is $\imath(A)=(a_{i,j}I_2)$, 
$A=(a_{i,j})\in M_4$ and let $\jmath:M_4\rightarrow M_8$ be the twisted map $\jmath(A)=\lambda(\epsilon_2)\imath(A)=\imath(A)\lambda(\epsilon_2)$ where $\epsilon_2=\diag(1,-1)$. 
Then, the embedding of $\SO_8$ to $\Sp_{16}$ in terms of the $4\times 4$-blocks is given by
\[
\begin{pmatrix} a & b \\ c & d\end{pmatrix} \mapsto \begin{pmatrix} \imath(a) & \jmath(b) \\ \jmath(c) & \imath(d)\end{pmatrix}=\diag(\lambda(\epsilon_2),I_8)    \begin{pmatrix} \imath(a) & \imath(b) \\ \imath(c) & \imath(d)\end{pmatrix} \diag(\lambda(\epsilon_2),I_8).
\]
Furthermore, we remark that the group homomorphism $\lambda$ and the algebra homomorphism $\imath$ are such that $\lambda(\GL_2)$ and $\imath(M_4)$ commute with each other. 

Recall that $\Ad(N)$ consists of upper-triangular unipotent matrices in $\SO_8$.
For
\begin{equation}\label{eq n coord}
n=\begin{pmatrix} 1 & x & z \\  & 1 & y \\ & & 1 \end{pmatrix}\in N(\A)
\end{equation}
we write
\[
\Ad(n)=\begin{pmatrix} I_4 & s \\ & I_4 \end{pmatrix}\begin{pmatrix} u &  \\ & u^* \end{pmatrix}
\]
with $u=u(n)\in \GL_4$ upper-triangular unipotent and $s=s(n)\in M_4$ such that $sw_4$ is alternating.
We have
\[
[h,\Ad(n)]=\begin{pmatrix} I_8 & \jmath(s) \\ & I_8 \end{pmatrix}\begin{pmatrix} \imath(u) \lambda(h) &  \\ & \imath(u^*)\lambda(h^*) \end{pmatrix}, \ \ \ h\in \SL_2(\A),\,
n\in N(\A).
\]
Since $w_8=\lambda(w_2)\imath(w_4)$ we see that $\jmath(s)w_8=\imath(sw_4)\lambda(\epsilon_2w_2)$.
We also note that
\[
\xi^{[h,\Ad(n)]}=\xi \imath(u)\lambda(h).
\]
Since $h\epsilon_2w_2 \,{}^t h=\epsilon_2w_2 $ it follows that $x\jmath(s)w_8{}^t x=x\lambda(h)\jmath(s)w_8{}^t \lambda(h){}^t x$, $x\in \A^8$, $h\in \SL_2(\A)$ and therefore
\begin{equation}\label{eq hind}
\c(\xi,[h,\Ad(n)])=\c(\xi,[I_2,\Ad(n)]), \ \ \ n\in N(\A)
\end{equation}
is independent of $h\in \SL_2(\A)$. 

We notice from the explicit formula for $\Ad(n)$ in Section~\ref{The Ad embedding} that
\[
u=\begin{pmatrix} 1 &x &-y & -yx-z\rho\\ & 1 & &  y\rho^2 \\ & & 1 & x \\ & & &  1 \end{pmatrix}
\]
and 
\[
sw_4=\begin{pmatrix}  & (z-yx)y\rho^2 & zx\rho & -(yx+z\rho^2)\\ (yx-z)y\rho^2 &  & -(z+xy\rho^2) & y\rho\\ -zx\rho & z+xy\rho^2&  & x \\ yx+z\rho^2 & -y\rho & -x  &  \end{pmatrix}. 
\]
Writing $\xi=(\xi_1,\xi_2,\xi_3,\xi_4)\in F^8$ with $\xi_i\in F^2$, $i=1,2,3,4$ we therefore have
\begin{equation}\label{eq action}
\xi^{[h,\Ad(n)]}=(\xi_1 h,(\xi_2+x\xi_1)h, (\xi_3-y\xi_1)h,(\xi_4+x\xi_3+y\rho^2 \xi_2-(yx+z\rho)\xi_1)h).
\end{equation}
We also conclude that $\jmath(s)w_8$ equals
\[ \scalemath{0.83}{
\begin{pmatrix}  0 & 0 & 0 & (z-xy)y\rho^2  & 0 & xz\rho & 0 & -(xy+z\rho^2)\\ 																					                                                                                          0 & 0 & (xy-z)y\rho^2  & 0 & -xz\rho & 0 & xy+z\rho^2 & 0 
\\ 0 & (xy-z)y\rho^2 & 0 & 0 & 0 & -(xy\rho^2 +z) & 0 & y\rho 
 \\ (z-xy)y\rho^2 & 0 & 0 & 0&  xy\rho^2 +z & 0 & -y\rho & 0 
 \\ 0 & -xz\rho & 0 & xy\rho^2+z & 0 & 0 & 0 & x 
 \\  xz\rho & 0 & -(xy\rho^2+z) & 0 & 0 & 0 & -x & 0 
  \\ 0 & xy+z\rho^2  & 0 &-y\rho & 0 & -x & 0& 0 
  \\-(xy+z\rho^2) & 0 & y\rho & 0 & x & 0 & 0 & 0  \end{pmatrix}}
\]
and this gives us an explicit formula for $\c(\xi,[h,\Ad(n)])=\psi(\xi \jmath(s)w_8 \,{}^t\xi)$.

\subsection{Relevant orbits and their associated orbital integrals} 

We use the above notation and formulas in order to find an explicit set $\Xi_\rel$ of relevant orbit representatives, compute their stabilizer and explicate the associated factorizable orbital integrals. Let $e_1=(1,0)$, $e_2=(0,1)$ be the standard basis of $F^2$ and  $N_2=\{g\in \SL_2:e_2 g=e_2\}$. Whenever we write $x=(x_1,x_2,x_3,x_4)\in \A^8$ we mean that each $x_i\in \A^2$. In particular, we write $x_i=0$ if $x_i=(0,0)$.

The following elementary observation will be used repeatedly in the orbit analysis.
\begin{lemma}
Let $\{v_1,v_2\}$ be a basis of $F^2$. Then there exists $h\in \SL_2(F)$ and $a\in F^*$ such that $v_1 h=e_2$ and $v_2h=a e_1$.
\end{lemma}
\begin{proof}
Applying $h=\sm{}{1}{-1}{}$ if necessary, we may assume without loss of generality that $v_1=(a,b)$ with $b\ne 0$. Then by applying $\sm{b}{}{-a}{b^{-1}}$ we may assume without loss of generality that $v_1=e_2$. Now write $v_2= (a,b)$.  Then $a\ne 0$ and applying $\sm{1}{-a^{-1}b}{}{1}$ the lemma follows.  
\end{proof}

For $\phi\in \swrz(\A^8)$ we let $\phi_{K_2}\in \swrz(\A^8)$ be defined by
\[
\phi_{K_2}(v)=\int_{K_2}\phi[v\lambda(k)]\ dk
\]
where $K_2$ is the standard maximal compact subgroup of $\SL_2(\A)$.

Let $\xi=(\xi_1,\xi_2,\xi_3,\xi_4)\in F^8$. It follows from \eqref{eq action} that
the following are invariants of $H(F)$-orbits. 
\begin{itemize}
\item $d_1(\xi)=\dim \langle \xi_1 \rangle$;
\item $d_{1,2}(\xi)=\dim \langle \xi_1,\xi_2 \rangle$;
\item $d_{1,3}(\xi)=\dim \langle \xi_1,\xi_3 \rangle$;
\item $d_{1,2,3}(\xi)=\dim \langle \xi_1,\xi_2,\xi_3 \rangle$;
\item $d_{1,2,3,4}(\xi)=\dim \langle \xi_1,\xi_2,\xi_3,\xi_4 \rangle$;
\end{itemize}

In the remainder of this Section, we will analyze the orbits case by case using these invariants.  The following result summarizes this analysis.

\begin{theorem}\label{I-relevant-orbits}
 A complete set of representatives for the $H(F)$-orbits that are relevant is given in Lemma~\ref{lem sl2} (orbits with $d_1=0$ and $d_{1,2}=d_{1,2,3,4}=1$);
Lemma~\ref{relevant-2} (orbits with $d_1=0$, $d_{1,2}=d_{1,2,3}=1$ and $d_{1,2,3,4}=2$); Lemma~\ref{relevant-47} (orbits with $d_1=1=d_{1,2,3,4}$); 
Lemma~\ref{relevant-3} (orbits with $d_1=1=d_{1,2}$ and $d_{1,2,3}=2$); Lemma~\ref{relevant-4} (orbits with
$d_1=1=d_{1,3}$ and $d_{1,2}=2$); and Lemma~\ref{relevant-5} (orbits with $d_1=1$ and $d_{1,2}=2=d_{1,3}$; these are the generic relevant orbits).  
\end{theorem}

When the orbit is relevant, 
in these Lemmas we also write the associated orbital integral explicitly.  There are other $H(F)$-orbits, and as part of the proof we will show that they are not relevant.
Throughout this analysis we implicitly use the coordinates \eqref{eq n coord} for $n\in N(\A)$.

\subsubsection{Orbits with $d_{1,2}=0$ are irrelevant}

Suppose $d_{1,2}(\xi)=0$ and write $\xi=(0,0,\xi_3,\xi_4)$.
Note that 
\[
\c(\xi,[h,\Ad(n)])=2x \xi_3 \epsilon_2 w_2\,{}^t \xi_4
\]
is independent of $y$ (and $z$) and by \eqref{eq action} we have  
\[
\xi^{[h,\Ad(n)]}=(0,0, \xi_3 h,(\xi_4+x\xi_3)h).
\]
In particular, 
\[
\left\{(I_2,\begin{pmatrix} 1 & &  \\ & 1 & y \\ & & 1 \end{pmatrix}): y\in \A\right\}\subseteq H_\xi(\A)
\]
and therefore $\xi$ is not relevant.

\subsubsection{The $H(F)$-orbits with invariants $d_1=0$ and $d_{1,2}=d_{1,2,3,4}=1$}
\begin{lemma}\label{lem sl2}
A complete set of representatives for $H(F)$-orbits with invariants $d_1=0$ and $d_{1,2}=d_{1,2,3,4}=1$ is given by 
\[
\{(0,e_2,\alpha e_2,0):\alpha\in F\}.
\]
The vector $(0,e_2,\alpha e_2,0)$ is relevant if and only if $\alpha=\rho^2$.
The associated orbital integral is
\[
I((0,0,0,1,0,\rho^2,0,0),\phi)=\int_\A\int_{\A^*} \phi_{K_2}[(0,0,0,t,0,\rho^2t,0,xt)]\abs{t}^2\ d^*t\ \psi(x) \ dx.
\]
\end{lemma}
\begin{proof}
Let $\xi\in F^8$ be such that $d_1(\xi)=0$ and $d_{1,2}(\xi)=d_{1,2,3,4}(\xi)=1$. Then there exists $0\ne v\in F^2$ and $\alpha,\,\alpha'\in F$ such that $\xi=(0,v,\alpha v,\alpha' v)$. Since $\SL_2(F)$ acts transitively on non-zero vectors in $F^2$ there exists $h_0\in \SL_2(F)$ such that $vh_0=e_2$. Let 
\[
n_0=\begin{pmatrix}
1 &  & \\ & 1 & -\alpha' \rho \\ & & 1
\end{pmatrix};
\] 
then $\xi^{[h_0,\Ad(n_0)]}=(0,e_2,\alpha e_2,0)$.
Now for a general $g=(h,n)\in H(F)$ we have
\[
(0,e_2,\alpha e_2,0)^{[g]}=(0,e_2 h,\alpha e_2 h,(x\alpha+y\rho^2)e_2 h).
\]
This is of the form $(0,e_2,\alpha'e_2,0)$ for some $\alpha'\in F$ if and only if $h\in N_2(F)$ and $x\alpha+y\rho^2=0$. We conclude that $(0,e_2,\alpha e_2,0)$ and $(0,e_2,\alpha'e_2,0)$ are in the same $H(F)$-orbit if and only if $\alpha=\alpha'$ and furthermore 
\[
H_{(0,e_2,\alpha e_2,0)}=N_2 \times \{n\in N: x\alpha+y\rho^2=0\}.
\] 
We further observe that $\c((0,e_2,\alpha e_2,0),[g])=1$ for all $g\in H(\A)$. It follows that the orbit containing $(0,e_2,\alpha e_2,0)$ is relevant if and only if $\alpha=\rho^2$. 
Thus, for $\xi=(0,e_2,\rho^2 e_2,0)$ we have
\[
\O(\xi,\phi)=\int_\A \int_{N_2(\A)\bs \SL_2(\A)} \phi[(0,0,0,1,0,\rho^2,0,x)\lambda(h)]\ dh\ \psi(x) \ dx.
\]
The lemma now follows using the Iwasawa decomposition.
\end{proof}
\subsubsection{The $H(F)$-orbits with invariants $d_1=0$, $d_{1,2}=d_{1,2,3}=1$ and $d_{1,2,3,4}=2$}

\begin{lemma}\label{relevant-2}
A complete set of representatives for $H(F)$-orbits with invariants $d_1=0$, $d_{1,2}=d_{1,2,3}=1$ and $d_{1,2,3,4}=2$ is given by 
\[
\{(0,e_2,\alpha e_2,\beta  e_1):\alpha\in F,\beta\in F^*\}.
\]
The vector $(0,e_2,\alpha e_2,\beta e_1)$ is relevant if and only if $\alpha=\rho$ and $\beta=\frac12\rho^2$. The associated orbital integral is
\[
\O((0,0,0,1,0,\rho,\frac12\rho^2,0),\phi)=\int_{\A^*}\int_{\A}\phi_{K_2}[(0,0,0,t,0,\rho t,\frac12\rho^2 t^{-1}, a)]\ da\abs{t}\ d^*t.
\]
\end{lemma}
\begin{proof}
Let $\xi\in F^8$ be such that $d_1(\xi)=0$, $d_{1,2}(\xi)=d_{1,2,3}(\xi)=1$ and $d_{1,2,3,4}(\xi)=2$. Then, there exists a basis $\{v,w\}$ of $F^2$ and $\alpha
\in F$ such that $\xi=(0,v,\alpha v,w)$. Applying Lemma \ref{lem sl2} there exist $h_0\in \SL_2(F)$ and $\beta\in F^*$ such that 
\[
\xi^{[h_0,I_8]}=(0,e_2,\alpha e_2,\beta e_1).
\]
 Now for a general $g=(h,n)\in H(F)$ we have
\[
(0,e_2,\alpha e_2,\beta e_1)^{[g]}=(0,e_2 h,\alpha e_2 h,[\beta e_1+(x\alpha+y\rho^2)e_2] h).
\]
This is of the form $(0,e_2,\alpha'e_2,\beta' e_1)$ for some $\alpha'\in F$ and $\beta'\in F^*$ if and only if $h=\sm{1}{a}{}{1}$where $\beta a+x\alpha+y\rho^2=0$ and in this case $(\alpha',\beta')=(\alpha,\beta)$. We conclude that $(0,e_2,\alpha e_2,\beta e_1)$ and $(0,e_2,\alpha'e_2,\beta' e_1)$ are in the same $H(F)$-orbit if and only if $(\alpha',\beta')=(\alpha,\beta)$ and furthermore 
\[
H_{(0,e_2,\alpha e_2,\beta e_1)}=\{(\begin{pmatrix} 1 & -\beta^{-1}(x\alpha+y\rho^2) \\  & 1 \end{pmatrix}, n): n\in N \}.
\] 
We further observe that 
\[
\c_{\psi}((0,e_2,\alpha e_2,\beta  e_1),[h,\Ad(n)])=\psi[(1-2\beta \rho)y+(1-2\beta\alpha)x].
\]
It follows that $(0,e_2, \alpha e_2,\beta  e_1)$ is relevant if and only if $1-2\beta\rho=1-2\beta\alpha=0$, i.e.,  if and only if $\alpha=\rho$ and $\beta=\frac12\rho^2$.
Thus, for $\xi=(0,e_2,\rho e_2,\frac12\rho^2 e_1)$ we have
\[
\O(\xi,\phi)= \int_{\SL_2(\A)} \phi[(0,0,0,1,0,\rho,\frac12\rho^2,0)\lambda(h)]\ dh.
\]
Using the Iwasawa decomposition we can write this as
\[
\int_{\A^*}\int_{\A}\phi_{K_2}[(0,0,0,t,0,\rho t,\frac12\rho^2 t^{-1},\frac12\rho^2 a t)\ da\abs{t}^2\ d^*t.
\]
Making the change of variables $a\mapsto 2\rho t^{-1} a$ the lemma follows. 

\end{proof}
\subsubsection{The $H(F)$-orbits with invariants $d_1=0$, $d_{1,2}=1$ and $d_{1,2,3}=2$ are not relevant}

\begin{lemma}
A complete set of representatives for $H(F)$-orbits with invariants $d_1=0$, $d_{1,2}=1$ and $d_{1,2,3}=2$ is 
\[
\{(0,e_2,\beta e_1,0):\beta\in F^*\}.
\]
None of them is relevant.
\end{lemma}
\begin{proof}
Let $\xi\in F^8$ be such that $d_1(\xi)=0$, $d_{1,2}(\xi)=1$ and $d_{1,2,3}(\xi)=2$. Then, there exists a basis $\{v,w\}$ of $F^2$ and $u\in F^2$ such that $\xi=(0,v,w,u)$. Applying Lemma \ref{lem sl2} let $h_0\in \SL_2(F)$ and $\beta\in F^*$ be such that $\{vh_0,wh_0\}=\{e_2,\beta e_1\}$. Let $x_0,
,y_0\in F$ be such that $-u=x_0 \beta v+y_0\rho^2 w$ and let
\[
n_0=\begin{pmatrix}
1 & x_0 &  \\ & 1 & y_0 \\ & & 1
\end{pmatrix}.
\] 
Then
\[
\xi^{[h_0,n_0]}=(0,e_2,\beta e_1,0).
\]
 Now for a general $g=(h,n)\in H(F)$ we have
\[
(0,e_2,\beta e_1,0)^{[g]}=(0,e_2 h,\beta e_1 h,[x\beta e_1+y\rho^2 e_2] h).
\]
This is of the form $(0,e_2,\beta' e_1,0)$ for some  $\beta'\in F^*$ if and only if $h=I_2$ and $x=y=0$ and in this case $\beta'=\beta$. We conclude that $(0,e_2,\beta e_1,0)$ and $(0,e_2,\beta' e_1,0)$ are in the same $H(F)$-orbit if and only if $\beta'=\beta$ and furthermore 
\[
H_{(0,e_2,\beta e_1,0)}=I_2 \times \{n\in N: x=y=0\}.
\] 
We further observe that 
\[
\c((0,e_2,\beta e_1,0),[h,\Ad(n)])=\psi[2\beta(xy\rho^2+z)], \ \ \ (h,n)\in H(\A).
\] 
The inclusion
\[
\{I_2\}\times \{\begin{pmatrix}
1 &  & z \\ & 1 &  \\ & & 1
\end{pmatrix}:z\in \A\}\subseteq H_{(0,e_2,\beta e_1,0)}(\A)
\]
therefore shows that ${(0,e_2,\beta e_1,0)}$ is not relevant for any $\beta\ne 0$. 
\end{proof}

This completes the analysis of orbits with $d_1=0$. 

\subsubsection{The $H(F)$-orbits with invariants $d_1=1=d_{1,2,3,4}$}

\begin{lemma}\label{relevant-47}
The set of all $\xi\in F^8$ with invariants $d_1=1=d_{1,2,3,4}$ consists of a unique $H(F)$-orbit with representative $(e_2,0,0,0)$ and it is relevant. The associated orbital integral is
\begin{multline*}
\O((0,1,0,0,0,0,0,0),\phi)=\\
\int_\A \int_\A\int_\A \int_{\A^*} \phi_{K_2}[(0,t,0,xt,0,-yt,0,zt)]\abs{t}^2\ d^*t \ dz\ \psi(x+y)\ dx\ dy.
\end{multline*}
\end{lemma}
\begin{proof}
Assume that $d_1(\xi)=1=d_{1,2,3,4}(\xi)$. Then there exists $0\ne v\in F^2$ and $\alpha_i\in F$, $i=1,2,3$ such that $\xi=(v,\alpha_1 v,\alpha_2 v,\alpha_3 v)$. 
Let $h_0\in \SL_2(F)$ be such that $vh_0=e_2$ and let
\[
n_0=\begin{pmatrix}
1 & -\alpha_1 & \rho^2\alpha_3+\rho\alpha_1\alpha_2\\  & 1 & \alpha_2 \\ & & 1
\end{pmatrix}.
\]
Then $\xi^{[h_0,n_0]}=(e_2,0,0,0)$.
For $g=(h,n)\in H(\A)$ we have
\[
(e_2,0,0,0)^{[g]}=(e_2 h,xe_2h,-ye_2h,-(xy+z\rho)e_2h)
\]
and therefore $H_{(e_2,0,0,0)}(\A)=N_2(\A)\times \{I_3\}$. Furthermore, $\c((e_2,0,0,0),[h])=1$, $h\in H(\A)$. It follows that $\xi=(e_2,0,0,0)$ is relevant and we have
\[
\O(\xi,\phi)=\int_\A \int_\A\int_\A\int_{N_2(\A)\bs \SL_2(\A)} \phi[(e_2,xe_2,-ye_2,-(xy+z\rho)e_2)\lambda(h)]\ dh\ \psi(x+y) \ dx\ dy\ dz.
\] 
The change of variables $z\mapsto -\rho^2(z+xy)$ gives
\[
\O(\xi,\phi)=\int_\A \int_\A\int_\A\int_{N_2(\A)\bs \SL_2(\A)} \phi[(e_2,xe_2,-ye_2,ze_2)\lambda(h)]\ dh\ dz\ \psi(x+y) \ dx\ dy.
\] 
The lemma follows by applying the Iwasawa decomposition. 
\end{proof}

\subsubsection{The $H(F)$-orbits with invariants $d_1=1=d_{1,2,3}$ and $d_{1,2,3,4}=2$ are not relevant}
\begin{lemma}
A complete set of representatives of $H(F)$-orbits with invariants $d_1=1=d_{1,2,3}$ and $d_{1,2,3,4}=2$ is 
\[
\{(e_2,0,0,\beta e_1):\beta\in F^*\}.
\]
None of them is relevant.
\end{lemma}
\begin{proof}
Assume that $d_1(\xi)=1=d_{1,2,3}(\xi)$ and $d_{1,2,3,4}(\xi)=2$. Then there exists a basis $\{v,w\}$ of $F^2$ and $\alpha_i\in F$, $i=1,2$ such that $\xi=(v,\alpha_1 v,\alpha_2 v,w)$. 
Applying Lemma \ref{lem sl2} let $h_0\in \SL_2(F)$ and $\beta\in F^*$ be such that $vh_0=e_2$ and $wh_0=\beta e_1$. Let 
\[
n_0=\begin{pmatrix}
1 & -\alpha_1 & \rho\alpha_1\alpha_2\\  & 1 & \alpha_2 \\ & & 1
\end{pmatrix}.
\]
Then $\xi^{[h_0,n_0]}=(e_2,0,0,\beta e_1)$.
For $g=(h,n)\in H(\A)$ we have
\[
(e_2,0,0,\beta e_1)^{[g]}=(e_2 h,xe_2h,-ye_2h,\beta e_1h-(xy+z\rho)e_2 h).
\]
This is of the form $(e_2,0,0,\beta' e_1)$ for some $\beta'\in F^*$ if and only if $x=y=0$ and $h=\sm{1}{a}{}{1}$ with $\beta a-z\rho=0$.
It follows that $(e_2,0,0,\beta e_1)$ and $(e_2,0,0,\beta' e_1)$ are in the same $H(F)$-orbit if and only if $\beta=\beta'$ and
\[
H_{(e_2,0,0,\beta e_1)}(\A)=\{  (\begin{pmatrix} 1 & \beta^{-1}z\rho \\  & 1 \end{pmatrix},    \begin{pmatrix}
1 &  & z\\  & 1 &  \\ & & 1
\end{pmatrix}):z\in \A  \}.
\]
We further note that
\[
\c((e_2,0,0,\beta e_1),[g])=\psi[2\beta(xy+z\rho^2)]
\]
and therefore we can deduce that $(e_2,0,0,\beta e_1)$ is not relevant.

\end{proof}

\subsubsection{The $H(F)$-orbits with invariants $d_1=1=d_{1,2}$ and $d_{1,2,3}=2$}

\begin{lemma}\label{relevant-3}
A complete set of representatives of $H(F)$-orbits with invariants $d_1=1=d_{1,2}$ and $d_{1,2,3}=2$ is 
\[
\{(e_2,\alpha e_2,\beta e_1,0):\alpha\in F,\,\beta\in F^*\}.
\]
The vector $(e_2,\alpha e_2,\beta e_1,0)$ is relevant if and only if $1+2\beta\rho\alpha^2=0$. For $\alpha\in F^*$ and $\xi[\alpha]=(e_2,\alpha e_2,-\frac12 \rho^2\alpha^{-2} e_1,0)$ the associated orbital integral is
\begin{multline*}
\O(\xi[\alpha],\phi)=\\
\int_{\A^*}\int_\A\int_\A\int_\A \phi_{K_2}[(0,t,0,t(\alpha+x),-\frac12 \rho^2\alpha^{-2} t^{-1},-ty,-\frac12 \rho^2\alpha^{-2}t^{-1} x,t(y\rho^2\alpha-xy-z\rho))]\\ \psi[-\rho\alpha^{-1} xy-\rho^2\alpha^{-1}z+ \alpha^{-2}xz+x+y] \ dx\ dy\ dz\abs{t}^2\ d^*t.
\end{multline*}

\end{lemma}
\begin{proof}
Assume that $d_1(\xi)=1=d_{1,2}(\xi)$ and $d_{1,2,3}(\xi)=2$. Then there exists a basis $\{v,w\}$ of $F^2$, $\alpha\in F$ and $u\in F^2$ such that $\xi=(v,\alpha v,w,u)$. 
Let $x_0,\,z_0\in F$ be such that $u=z_0\rho v-x_0 w$ and let
\[
n_0=\begin{pmatrix}
1 & x_0 & z_0 \\  & 1 &  \\ & & 1
\end{pmatrix}.
\]
Applying Lemma \ref{lem sl2}, let $h_0\in \SL_2(F)$ and $\beta\in F^*$ be such that $vh_0=e_2$ and $wh_0=\beta e_1$.
Then $\xi^{[h_0,n_0]}=(e_2,(\alpha+x_0)e_2,\beta e_1,0)$.
For $g=(h,n)\in H(\A)$ we have
\[
(e_2,\alpha e_2,\beta e_1,0)^{[g]}=(e_2 h,(\alpha+x)e_2h,(\beta e_1-ye_2)h,(\beta x e_1+(y\rho^2\alpha-xy-z\rho) e_2)h).
\]
This is of the form $(e_2,\alpha' e_2,\beta' e_1,0)$ for some $\alpha'\in F$ and $\beta'\in F^*$ if and only if $x=0$, $z\rho=y\rho^2\alpha$ and $h=\sm{1}{a}{}{1}$ with $\beta a=y$.
It follows that $(e_2,\alpha e_2,\beta e_1,0)$ and $(e_2,\alpha' e_2,\beta' e_1,0)$ are in the same $H(F)$-orbit if and only if $(\alpha,\beta)=(\alpha',\beta')$ and
\[
H_{(e_2,\alpha e_2,\beta e_1,0)}(\A)=\{  (\begin{pmatrix} 1 & \beta^{-1}y \\  & 1 \end{pmatrix},    \begin{pmatrix}
1 &  & y\rho\alpha\\  & 1 & y \\ & & 1
\end{pmatrix}):y\in \A  \}.
\]
We further note that
\[
\c((e_2,\alpha e_2,\beta e_1,0),[g])=\psi[2\beta\alpha (xy\rho^2+z)-2\beta xz\rho].
\]
We therefore have
\begin{multline*}
\c_\psi((e_2,\alpha e_2,\beta e_1,0),[g])=\psi[(1+2\beta\alpha^2\rho)y],\\ g=(\begin{pmatrix} 1 & \beta^{-1}y \\  & 1 \end{pmatrix},    \begin{pmatrix}
1 &  & y\rho\alpha\\  & 1 & y \\ & & 1
\end{pmatrix})\in H_{(e_2,\alpha e_2,\beta e_1,0)}(\A).
\end{multline*}
We conclude that $(e_2,\alpha e_2,\beta e_1,0)$ is relevant if and only if $1+2\beta\alpha^2\rho=0$ and for $\beta=-\frac12 \rho^2\alpha^{-2}$ we have
\begin{multline*}
\O(0,1,0,\alpha,\beta,0,0,0),\phi)=\\ \int_\A\int_\A\int_{\SL_2(\A)} \phi[(0,1,0, \alpha+x,\beta,0,\beta x,-z\rho)\lambda(h)]\ dh\ \psi(x-\rho^2 z-2\beta\rho xz)\ dx\ dz.
\end{multline*}
Alternatively, we can also express the orbital integral as the iterated integral
\begin{multline*}
\int_{N_2(\A)\bs \SL_2(\A)} \int_\A\int_\A\int_\A \phi[(0,1,0,\alpha+x,\beta,-y,\beta x,y\rho^2\alpha-xy-z\rho) \lambda(h)]\\ \psi[2\beta\alpha (xy\rho^2+z)-2\beta xz\rho+x+y] \ dx\ dy\ dz\ dh
\end{multline*}
where the integral over ${N_2(\A)\bs \SL_2(\A)}$ makes sense only after integration over $y$ and $z$. 
Applying the Iwasawa decomposition, this equals
\begin{multline*}
\int_{\A^*}\int_\A\int_\A\int_\A \phi_{K_2}[(0,t,0,t(\alpha+x),t^{-1}\beta,-ty,t^{-1}\beta x,t(y\rho^2\alpha-xy-z\rho))]\\ \psi[2\beta\alpha (xy\rho^2+z)-2\beta xz\rho+x+y] \ dx\ dy\ dz\abs{t}^2\ d^*t.
\end{multline*}
The lemma follows.

\end{proof}

\subsubsection{The $H(F)$-orbits with invariants $d_1=1=d_{1,3}$ and $d_{1,2}=2$}

\begin{lemma}\label{relevant-4}
A complete set of representatives of $H(F)$-orbits with invariants $d_1=1=d_{1,3}$ and $d_{1,2}=2$ is 
\[
\{(e_2,\beta e_1,\alpha e_2,0):\alpha\in F,\,\beta\in F^*\}.
\]
The vector $(e_2,\beta e_1,\alpha e_2,0)$ is relevant if and only if $\beta=\frac12 \rho\alpha^{-2}$. For $\alpha\in F^*$ and $\xi[\alpha]=(e_2,\frac12 \rho\alpha^{-2} e_1,\alpha e_2,0)$ the associated orbital integral is
\begin{multline*}
\O(\xi[\alpha],\phi)=\int_{\A^*}\int_\A\int_\A\int_\A \phi_{K_2}[(0,t,\frac12 \rho\alpha^{-2}t^{-1},tx,0,t(\alpha-y),\frac12\alpha^{-2}t^{-1}y,t(x\alpha-xy-z\rho))] \\ \psi[\alpha^{-2}(xy-z)y-\rho(xy\rho^2+z)\alpha^{-1}+x+y] \ dx\ dy\ dz\abs{t}^2\ d^*t.
\end{multline*}
\end{lemma}
\begin{proof}
Assume that $d_1(\xi)=1=d_{1,3}(\xi)$ and $d_{1,2}(\xi)=2$. Then there exists a basis $\{v,w\}$ of $F^2$, $\alpha\in F$ and $u\in F^2$ such that $\xi=(v,w,\alpha v,u)$. 
Let $y_0,\,z_0\in F$ be such that $u=z_0\rho v-y_0\rho^2 w$ and let
\[
n_0=\begin{pmatrix}
1 &  & z_0 \\  & 1 & y_0 \\ & & 1
\end{pmatrix}.
\]
Applying Lemma \ref{lem sl2} let $h_0\in \SL_2(F)$ and $\beta\in F^*$ be such that $vh_0=e_2$ and $wh_0=\beta e_1$.
Then $\xi^{[h_0,n_0]}=(e_2,\beta e_1,(\alpha-y_0) e_2,0)$. 
For $g=(h,n)\in H(\A)$ we have
\[
(e_2,\beta e_1,\alpha e_2,0)^{[g]}=(e_2 h,(\beta e_1+xe_2)h,(\alpha-y) e_2 h,(x\alpha e_2+y\rho^2\beta  e_1-(xy+z\rho) e_2)h).
\]
This is of the form $(e_2,\beta' e_1,\alpha' e_2,0)$ for some $\alpha'\in F$ and $\beta'\in F^*$ if and only if
$h=\sm{1}{a}{}{1}$ with $\beta a+x=0$, $y=0$ and $z\rho=x\alpha$ and in this case $(\alpha,\beta)=(\alpha',\beta')$. 
It follows that $(e_2,\beta e_1,\alpha e_2,0)$ and $(e_2,\beta' e_1,\alpha' e_2,0)$ are in the same $H(F)$-orbit if and only if $(\alpha,\beta)=(\alpha',\beta')$ and
\[
H_{(e_2,\beta e_1,\alpha e_2,0)}(\A)=\{  (\begin{pmatrix} 1 & -\beta^{-1}x \\  & 1 \end{pmatrix},    \begin{pmatrix}
1 & x & \rho^2 x\alpha\\  & 1 &  \\ & & 1
\end{pmatrix}):x\in \A  \}.
\]
We further note that
\[
\c((e_2,\beta e_1,\alpha e_2,0),[g])=\psi[2\beta(xy-z)\rho^2y-2(xy\rho^2+z)\alpha \beta].
\]
We therefore have
\[
\c_\psi((e_2,\beta e_1,\alpha e_2,0),[g])=\psi[(1-2\rho^2\alpha^2\beta)x],\ \ \ g=(\begin{pmatrix} 1 & -\beta^{-1}x \\  & 1 \end{pmatrix},    \begin{pmatrix}
1 & x & \rho^2 x\alpha\\  & 1 &  \\ & & 1
\end{pmatrix}).
\]
We conclude that $(e_2,\beta e_1,\alpha e_2,0)$ is relevant if and only if $2\rho^2\alpha^2\beta=1$ and in this case, the associated (iterated, as in the previous lemma) orbital integral is
\begin{multline*}
\int_{N_2(\A)\bs \SL_2(\A)}\int_\A\int_\A\int_\A \phi[(0,1,\beta,x,0,\alpha-y,y\rho^2\beta,x\alpha-xy-z\rho)\lambda(h)] \\ \psi[2\rho^2\beta(xy-z)y-2(xy\rho^2+z)\alpha \beta+x+y] \ dx\ dy\ dz\ dh.
\end{multline*}
Again applying the Iwasawa decomposition, the lemma follows.
\end{proof}

\subsubsection{The generic $H(F)$-orbits}
\begin{lemma}\label{relevant-5}
A complete set of representatives of $H(F)$-orbits with invariants $d_1=1$ and $d_{1,2}=2=d_{1,3}$ is 
\[
\{(e_2,\beta e_1,\beta' e_1,0):\beta,\,\beta'\in F^*\}.
\]
They are all relevant. For $a,\,b\in F^*$ and $\xi[a,b]=(e_2,a e_1,be_1,0)$ the associated orbital integral is
\begin{multline*}
\O(\xi[a,b],\phi)= \int_{\A}\int_{\A}\int_{\A}\int_\A \int_{\A^*} \\ \phi_{K_2}[(0,t,t^{-1}a,t(x+as),t^{-1}b,t(bs-y),t^{-1}(xb+y\rho^2a),t[(xb+y\rho^2a)s-(xy+z\rho)])]\abs{t}^2\ d^*t \ ds\\ \psi[x+y+2a(xy-z)y\rho^2-2bxz\rho] \ dx\ dy\ dz.
\end{multline*}
\end{lemma}
\begin{proof}
Let $\xi\in F^8$ be such that $d_1(\xi)=1$ and $d_{1,2}(\xi)=2=d_{1,3}(\xi)$. Then there exist $v,\,w,\,w',\,u\in F^2$ such that $\{v,w\}$ and $\{v,w'\}$ are bases of $F^2$ and 
$\xi=(v,w,w',u)$. Let $a,\,b\in F$ be such that $w'=a v+b w$ and note that $b\ne 0$. Let $x_0,\,y_0,\,z_0\in F$ be the unique common solutions to the equations
\[
\begin{cases} w'-y_0 v=b(w+x_0 v) \\ u+x_0 w'+y_0\rho^2 w-(x_0y_0+z_0\rho)v=0.\end{cases}
\]
Applying Lemma \ref{lem sl2} let $h_0\in \SL_2(F)$ and $\beta\in F^*$ be such that $(vh_0, (w+x_0 v)h_0)=(e_2,\beta e_1)$ and let
\[
n_0=\begin{pmatrix} 1 & x_0 & z_0 \\ & 1 & y_0 \\ & & 1 \end{pmatrix}.
\] 
Then
\[
\xi^{[h_0,\Ad(n_0)]}=(e_2,\beta e_1, b\beta e_1,0).
\]
For $g=(h,n)\in H(F)$ we have
\[
(e_2,\beta e_1,\beta' e_1,0)^{[g]}=(e_2 h, (\beta e_1+x e_2)h, (\beta' e_1-y e_2)h, [(x\beta'+y\rho^2 \beta)e_1-(xy+z\rho)e_2]h)
\]
and we can deduce that this is of the form $(e_2,\beta_1 e_1,\beta_1' e_1,0)$ for some $\beta_1,\,\beta_1'\in F^*$ if and only if $g$ is the identity. 
Thus, $H_{(e_2,\beta e_1,\beta' e_1,0)}$ is trivial and therefore $(e_2,\beta e_1,\beta' e_1,0)$ is relevant. 
We also note that
\[
\c((e_2,\beta e_1,\beta' e_1,0),[g])=\psi[2\beta(xy-z)y\rho^2-2\beta'xz\rho].
\]
We therefore have
\begin{multline*}
\O((e_2,\beta e_1,\beta' e_1,0),\phi)=\\ \int_{\A}\int_{\A}\int_{\A}\int_{\SL_2(\A)} \phi[(0,1,\beta,x,\beta',-y,x\beta'+y\rho^2\beta,-(xy+z\rho))\lambda(h)]\ dh\\ \psi[x+y+2\beta(xy-z)y\rho^2-2\beta'xz\rho] \ dx\ dy\ dz.
\end{multline*}
The lemma follows by the Iwasawa decomposition. 
\end{proof}

This completes the proof of Theorem~\ref{I-relevant-orbits}. \qed

\section{Unfolding the distribution \texorpdfstring{$J$}~}{}\label{J-distribution-cals}
Recall that $G'(\A)$ denotes the adelic $3$-fold cover of $\SL_3(\A)$ and $G'(F)$ denotes the embedded image of $\SL_3(F)$ in $G'(\A)$.  For notational
convenience we suppress this embedding and write elements of $G'(F)$ by matrices.
We consider the action $g \cdot (n_1,n_2) =n_1^{-1} g n_2$ of $N(F)\times N(F)$ on $G'(F)$.  
The element $g$ or its orbit is called relevant if $(n_1,n_2)\mapsto \psi(n_1^{-1}n_2)$ is trivial on the stabilizer $(N\times N)_g(\A)$ of $g$. 
For $a,b\in F^*$ let $t(a,b)=\diag(a,a^{-1}b,b^{-1})$, and let
\[
w_1=\begin{pmatrix} & -1 & \\ 1 &  & \\ & & 1 \end{pmatrix}, \ \ \ w_2=\begin{pmatrix} 1 & & \\ & & -1 \\ & 1 & \end{pmatrix}.
\]
Then using the Bruhat decomposition, one sees that a complete set ${\Xi}'_\rel$ of representatives for the relevant orbits 
is given as follows:
\begin{enumerate}
\item $t(a,a^2)=aI_3$, $a\in \mu_3\subseteq F^*$;
\item $t(a^{-2},a^{-1})w_1w_2=\begin{pmatrix} &  & a^{-2}  \\ a &  & \\   & a & \end{pmatrix}$, $a\in F^*$;
\item $t(-a,a^2)w_2w_1=\begin{pmatrix}  & a &  \\  &  & a \\ a^{-2} &  & \end{pmatrix}$, $a\in F^*$;
\item $t(b^{-1},a^{-1})w_1w_2w_1=\begin{pmatrix}  &  & b^{-1} \\  & -a^{-1}b &  \\ a &  & \end{pmatrix}$, $a,\,b\in F^*$.
\end{enumerate}
(See for example \cite{MR945275}, where the determination of relevant orbits is carried out in the language of Poincar\'e series; the adelic ingredients are the same.)
Note that the notion of relevance does not depend on the cover, though the orbital integral that we consider does.

To give the needed orbital integrals, 
we fix once and for all an embedding of $\mu_3$ in $\C^*$ and consider the genuine Hecke algebra $\H_{\gen}(G'(\A))=\{f'\in C_c^\infty(G'(\A)): f'(zg)=zf'(g)\}$. 
For $f'\in \H_{\gen}(G'(\A))$ and $g\in G'(F)$ relevant we consider the orbital integral
\[
\O'(g,f')=\int_{N(\A)\times N(\A)/(N\times N)_g(\A)} f'(n_1^{-1}gn_2)\psi(n_1^{-1}n_2)\ dn_1\ dn_2.
\]
Here $N(\A)$ is embedded in $G'(\A)$ by the trivial section $n\mapsto (n,1)$.
Explicitly, according to the relevant orbits we have
\begin{enumerate}
\item $\O'(t(a,a^2),f')=\int_{N(\A)} f'(an)\psi(n)\ dn$, $a\in \mu_3$;
\item $\O'(t(a^{-2},a^{-1})w_1w_2,f')=\int_{N(\A)\times U_{2,1}(\A)} f'(n^{-1}t(a^{-2},a^{-1})w_1w_2u)\psi(n^{-1}u)\ dn\ du$;
\item $\O'(t(-a,a^2)w_2w_1,f')=\int_{U_{2,1}(\A)\times N(\A)} f'(u^{-1}t(-a,a^2)w_2w_1n)\psi(u^{-1}n)\ du\ dn$;
\item $\O'(t(b^{-1},a^{-1})w_1w_2w_1,f')=\\ \int_{N(\A)\times N(\A)} f'(n_1^{-1}t(b^{-1},a^{-1})w_1w_2w_1n_2))\psi(n_1^{-1}n_2)\ dn_1\ dn_2$
\end{enumerate}
where 
\[
U_{2,1}(\A)=\{\begin{pmatrix} I_2 & v \\  & 1 \end{pmatrix} : v\in \A^2\}.
\]

\begin{proposition}\label{eq unfold2} The distribution $J$ given in \eqref{metaplectic-distribution} is a sum of factorizable orbital integrals
\begin{equation*}
J(f')=\sum_{\xi'\in {\Xi}'_\rel} \O(\xi',f')
\end{equation*}
where the relevant orbits and the associated orbital integrals are given above.
\end{proposition}

\goodbreak

\section*{\it{Part II. Local Theory: The Fundamental Lemma}}\label{Local theory}

\section{Notation and main result}

Throughout Part II we set the following notation.  Let $F$ denote a non-archimedean local field, $\O$ the ring of integers of $F$,
$\p$ its maximal ideal, $\varpi$ a uniformizer in $F$ (a generator of $\p$), $k_F=\O/\p$ the residual field of $F$,  $p$ the characteristic of $k_F$,
$q$ the number of elements in $k_F$, $\val:F^*\rightarrow \Z$ the standard valuation so that $\val(u\varpi^n)=n$ for $u\in \O^*$, $n\in \Z$, and
$S^n=\{s^n:s\in S\}$ for any subgroup $S$ of $F^*$ and $n\in\N$.
Assume that $F$ contains a primitive cube root of unity $\rho$ and set $\mu_3=\langle\rho\rangle$.

Fix a character $\psi$ of $F$ with conductor $\O$ (that is, $\psi$ restricted to $\O$ is trivial and restricted to $\p^{-1}$ is non-trivial). 
We denote by $\triv_A$ the characteristic function of a set $A$. 
We normalize integration so that $\int_\O\,dx=1$, $\int_{\O^*} \,d^*x=1$.

The main work is to compare the two local orbital integrals corresponding to the big cells.
These are families of integrals parameterized by $F^* \times F^*$. 
The first is defined by the formula
\begin{multline}\label{eq I def}
I(a,b)=\int_{F}\int_{F}\int_{F}\int_F \int_{F^*} \\ \triv_{\O^8}[(0,t,t^{-1}a,t(x+as),t^{-1}b,t(bs-y),t^{-1}(xb+y\rho^2a),t[(xb+y\rho^2a)s-(xy+z\rho)])]\abs{t}^2\ d^*t \ ds\\ \psi[x+y+2a(xy-z)y\rho^2-2bxz\rho] \ dx\ dy\ dz.
\end{multline}
This is the local integral obtained from Lemma~\ref{relevant-5}.

In order to define the second we need to set some further notation. Consider the local metaplectic three fold cover $\widetilde\SL_3(F)$ of $\SL_3(F)$.
It is a central extension
\[
1\rightarrow \mu_3 \rightarrow \widetilde\SL_3(F) \rightarrow \SL_3(F) \rightarrow 1.
\]
As a set, we identify $\widetilde\SL_3(F)$ with $\SL_3(F) \times \mu_3$ and the group operation is given by $(g_1,z_1)(g_2,z_2)=(g_1g_2,z_1z_2\sigma(g_1,g_2))$ where $\sigma$ is a certain $2$-cocycle of $\SL_3(F)$.
This cocycle is described, following Matsumoto \cite{MR240214}, in Bump and Hoffstein \cite[\S 2]{MR904946}.

Let $N$ be the group of upper triangular unipotent matrices in $\SL_3(F)$ and $K=\SL_3(\O)$. The map $n\mapsto (n,1):N \rightarrow \widetilde\SL_3(F)$ is an imbedding of $N$ in $\widetilde\SL_3(F)$ and this way we view $N$ as a subgroup of $\widetilde\SL_3(F)$.
The group $K$ also admits a splitting in $\widetilde\SL_3(F)$. There is a map $\kappa:K\rightarrow \mu_3$ such that the map $g\mapsto (g,\kappa(g))$ imbeds $K$ into $\widetilde\SL_3(F)$. 
By abuse of notation we also denote by $\psi$ the character of $N$ that satisfies $\psi(u)=\psi(u_{1,2}+u_{2,3})$, $u\in N$.

Fix an embedding of $\mu_3$ in $\C^*$ and let $f_0:\widetilde\SL_3(F)\rightarrow \C$ be defined by 
\[
f_0(g,z)=\begin{cases} z\kappa(g)^{-1} & g\in K \\ 0 & \text{otherwise}\end{cases} \ \ \ g\in\SL_3(F), \ z\in\mu_3.
\]
The second family of integrals is defined by 
\begin{equation}\label{eq J def}
J(a,b)=\int_N \int_N f_0( (u_1g_{a,b}u_2,1) )\psi(u_1 u_2)\ du_1\ du_2
\end{equation}
where
\[
g_{a,b}=\begin{pmatrix} & & b^{-1} \\ & -a^{-1}b & \\ a & & \end{pmatrix}\in \SL_3(F).
\]

Let $(\cdot,\cdot)_3:F^* \times F^* \rightarrow \mu_3$ be the cubic Hilbert symbol.  Our main result is the following comparison.
\begin{theorem}[The fundamental lemma for the big cell orbital integrals-unit element]\label{thm main}
Assume $p>3$. For any $a,\,b\in F^*$ we have
\[
I(a,b)=(c,d)_3\ J(c,d), \ \ \ \text{where} \ \ \ c=-54 a,\ d=54 b.
\]
\end{theorem}
In the final section, we also compare the orbital integrals for the other relevant orbits (Theorem~\ref{degenerate equality}).

\section{Ingredients needed for the comparison}\label{yummy-ingredients}

We develop the ingredients that are needed to establish Theorem~\ref{thm main}.  From now on we assume $p>3$. 

\subsection{The cubic Hilbert symbol}

We begin by recalling the basic properties of the cubic Hilbert symbol. They will be used throughout our computation without further mention. 
For $x,\,y,\,z\in F^*$ we have
\begin{itemize}
\item $(y,x)_3=\overline{(x,y)_3}$
\item $(xy,z)_3=(x,z)_3(y,z)_3$
\item $(x,1-x)_3=1$, $x\ne 1$
\item $(x,y)_3=1$ for all $y\in F^*$ if and only if $x\in F^{*3}$ 
\item $(x,u)_3=1$ for all $u\in \O^*$ if and only if $3| \val(x)$.
\end{itemize}
The last property is a consequence of our assumption that the residual characteristic $p>3$.
We further observe that by this assumption and Hensel's Lemma $1+\p\subseteq F^{*3}$ and consequently
\[
(x+y,z)_3=(x,z)_3\ \ \ \text{whenever} \ \ \  \abs{y}<\abs{x}.
\]

\subsection{Cubic Gauss sums} For future reference we recall here a formula that is essentially the calculation of the absolute value squared of a cubic Gauss sum. This result is standard and the proof is omitted.
\begin{lemma}\label{lem gauss}
Let $a\in F$ be such that $\abs{a}=q$. Then
\[
\int_{\O^*}\int_{\O^*} (a,uv^{-1})_3 \,\psi(a(u+v))\ du\ dv=q^{-1}.
\]
\qed
\end{lemma}

\subsection{Kloosterman integrals}
For $y\in F^*$ and $a,b\in F$ consider the integral
\[
\Kl(y;a,b)=\int_{\O^*} (y,u)_3\,\psi(au+bu^{-1})\ du.
\]
We provide a formula for $\Kl(y;a,b)$ in terms of Kloosterman sums over the residual field, the function
\[
\delta_{3\mid \val(x)}=\begin{cases} 1 & 3\mid \val(x) \\ 0 & \text{otherwise} \end{cases}, \ \ \ x\in F^*
\] 
and the function $\sq$ on $F^*$ defined by
\[
\sq(x)=\sum_{z\in \p^{\ell-e}/\p^{\ell}} \psi(xz^2)
\]
where $\ell\in\Z$ and $e\in\{0,1\}$ are such that $\abs{x}=q^{2\ell-e}$. Note that $\ell=\lfloor \frac{1-\val(x)}2 \rfloor$ and the summand $\psi(xz^2)$ is well defined (independent of the class of $z$ mod $\p^\ell$). In particular, $\sq(x)=1$ if $\val(x)$ is even.
It is easy to verify the basic property of $\sq$:
\begin{equation}\label{eq sq pr}
\sq(xy^2)=\sq(x),\ \ \ x,y\in F^*.
\end{equation}
In the evaluation below and subsequent formulas, if $x\in F^*$ is a square then we write $\sqrt{x}$ for a square root of $x$ in $F$; all formulas will be independent of the choice of square root.

\begin{lemma}\label{lem kl int}
We have
\[
\Kl(y;a,b)=\begin{cases} (1-q^{-1})\delta_{3\mid \val(y)} & \max(\abs{a},\abs{b})\le 1 \\
q^{-1}\sum_{u\in k_F^*}(y,u)_3\,\psi(au+bu^{-1}) &  \max(\abs{a},\abs{b})=q \\
q^{-\lfloor\frac{1-\val(a)}2\rfloor}(y,ab^{-1})_3\,\sum_{\epsilon=\pm 1} \psi(2\epsilon\sqrt{ab})\sq(\epsilon\sqrt{ab})& \max(\abs{a},\abs{b})>q \text{ and }a^{-1}b\in \O^{*2} \\
0 & \max(\abs{a},\abs{b})>q \text{ and }a^{-1}b\not\in \O^{*2}.
 \end{cases}
\]
\end{lemma}
\begin{proof}
If $\max(\abs{a},\abs{b})\le q$ then the integrand $(y,u)_3\psi(au+bu^{-1})$ depends only on $u+\p$ and therefore $\Kl(y;a,b)=q^{-1}\sum_{u\in k_F^*}(y,u)_3\psi(au+bu^{-1}) $. If in addition $a,b\in\O$ then $\psi(au+bu^{-1})=1$ for all $u\in k_F^*$. This explains the formula when $\max(\abs{a},\abs{b})\le q$.

Assume now that $\max(\abs{a},\abs{b})\ge q^2$ and let $\ell\ge 1$ and $e\in\{0,1\}$ be such that $\max(\abs{a},\abs{b}) = q^{2\ell -e}$. We decompose the integral by writing $u=v(1+z)$ with $v\in \O^*/(1+\p^\ell)$ and $z\in \p^\ell$. Since $(1+z)^{-1}=\sum_{i=0}^\infty (-z)^i\in 1-z+\p^{2\ell}$ and $bv^{-1}\p^{2\ell}\subseteq \O$ we conclude that
\[
\Kl(y;a,b)=\sum_{v\in \O^*/(1+\p^\ell)} (y,v)_3\psi(av+bv^{-1})\int_{\p^{\ell}} \psi((av-bv^{-1})z)\ dz. 
\]
We have 
\[
\int_{\p^{\ell}} \psi((av-bv^{-1})z)\ dz=\begin{cases}
q^{-\ell} & av-bv^{-1}\in \p^{-\ell} \\ 
0 & \text{otherwise.}
\end{cases}
\]
Note that $av-bv^{-1}\in \p^{-\ell}$ if and only if $\varpi^{2\ell-e}(av^2-b)\in\p^{\ell-e}$ and by assumption $\ell-e\ge 1$. 
By Hensel's lemma, there exists $v\in \O^*$ such that $av-bv^{-1}\in \p^{-\ell}$ if and only if $a^{-1}b \in\O^{*2}$ and in this case $v$ is of this form if and only if $v\in \pm v_0+ \p^{\ell-e}$ where $v_0^2=a^{-1}b$.
Note that for $v=\pm v_0(1+z)$ with $z\in\p^{\ell-e}$ we have
\[
(y,v)_3=(y,\pm v_0)_3=(y,v_0^{-2})_3=(y,ab^{-1})_3,
\]
\[
av+bv^{-1}\in \pm(av_0+bv_0^{-1}+(av_0-bv_0^{-1})z+bv_0^{-1}z^2)+b\p^{3(\ell-e)}
\]
and by assumption $av_0-bv_0^{-1}=0$ and $b\p^{3(\ell-e)}\subseteq \O$ so that
\[
\psi(av+bv^{-1})=\psi(\pm(av_0+bv_0^{-1})+bv_0^{-1}z^2).
\]
Note further that $av_0+bv_0^{-1}=2av_0=2bv_0^{-1}$ and that $bv_0^{-1}$ is a square root of $ab$. The lemma readily follows.
\end{proof}

We remark that the analysis above of the case $\max(\abs{a},\abs{b})\ge q^2$ is tantamount to an evaluation of the sum using the method of stationary phase.

\subsection{Cubic exponential integrals}
For $a,b\in F$ consider the integrals
\[
\Cu(a,b)=\int_\O \psi(ax+bx^3)\ dx \ \ \ \text{and}\ \ \ \Cu_0(a,b)=\int_{\O^*} \psi(au+bu^3)\ du.
\]
\begin{lemma}\label{lem cobic exp}
We have
\[
\Cu(a,b)=\begin{cases} 
1 & \max(\abs{a},\abs{b}) \le 1 \\ 
0 & \abs{b}<\abs{a}=q \\
q^{-1}\sum_{x\in k_F} \psi(ax+bx^3) &  \abs{a}\le \abs{b}=q \\
 0 &\max(\abs{b},q)<\abs{a} \\
q^{-\lfloor\frac{1-\val(a)}2\rfloor}\sum_{\epsilon=\pm 1} \psi(\epsilon \frac29 a\sqrt{-3ab^{-1}})\sq(\epsilon \frac19 a\sqrt{-3ab^{-1}}) & \abs{a}=\abs{b}>q \text{ and }ab^{-1}\in \O^{*2} \\ 
0 & \abs{a}=\abs{b}>q\text{ and }ab^{-1}\not\in \O^{*2}\end{cases}
\]
and 
\[
\Cu_0(a,b)=\begin{cases} 1-q^{-1} & \max(\abs{a},\abs{b}) \le 1 \\ 
-q^{-1} & \abs{b}<\abs{a}=q\\
q^{-1}\sum_{u\in k_F^*} \psi(au+bu^3) &  \abs{a}\le \abs{b}=q 
\\ q^{-\lfloor\frac{1-\val(a)}2\rfloor} \sum_{\epsilon=\pm 1} \psi(\epsilon \frac29 a\sqrt{-3ab^{-1}})\sq(\epsilon \frac19 a\sqrt{-3ab^{-1}}) & \max(\abs{a},\abs{b})>q \text{ and }ab^{-1}\in \O^{*2} \\ 0 & \text{otherwise.}\end{cases}
\]
In particular, 
\[
\Cu(a,b)=\begin{cases}
q^{-1}+\Cu_0(a,b) & \max(\abs{a},\abs{b})\le q \\
\Cu_0(a,b) & \max(\abs{b},q^2)\le \abs{a}.
\end{cases}
\]
\end{lemma}
\begin{remark}
The lemma excludes the computation of $\Cu(a,b)$ when $\max(\abs{a},q)<\abs{b}$. In fact, in this case we have
\[
\Cu(a,b)=\begin{cases} 
q^{-\ell}\sum_{x\in \p^{\lfloor \frac{\ell+1-e}2\rfloor}/\p^\ell} \psi(ax+bx^3) & \abs{a}\le q^\ell \\
0 & \text{otherwise}
\end{cases}
\] 
where $\ell\ge 1$ and $e\in\{0,1\}$ are such that $\abs{b}=q^{2\ell-e}$.
However, we omit the proof for this case since it is never used in our work. 
\end{remark}

\begin{proof}
The two formulas are straightforward when $\max(\abs{a},\abs{b})\le q$. Assume now that $\max(\abs{b},q^2)\le \abs{a}$ and let $\ell\ge 1$ and $e\in\{0,1\}$ be such that $\abs{a}=q^{2\ell-e}$. 
Since $b\p^{2\ell}\subseteq \O$, for $y\in \O$ we have
\[
\int_{y+\p^\ell}\psi(ax+bx^3)\ dx=\psi(ay+by^3)\int_{\p^\ell}\psi((a+3by^2)z)\ dz=\begin{cases} 
q^{-\ell}\psi(ay+by^3) & a+3by^2\in\p^{-\ell} \\
0 & \text{otherwise.}
\end{cases}
\]
It follows that
\[
\Cu(a,b)=q^{-\ell}\sum_{\substack{x\in \O/\p^\ell\\ a+3bx^2\in \p^{-\ell}}}\psi(ax+bx^3)\ \ \  \text{and}\ \ \  
\Cu_0(a,b)=q^{-\ell}\sum_{\substack{u\in \O^*/(1+\p^\ell)\\a+3bu^2\in \p^{-\ell}}}\psi(au+bu^3). 
\]

Note that for $y\in \O$ we have $a+3by^2\in\p^{-\ell}$ if and only if $\varpi^{2\ell-e}(a+3by^2)\in\p^{\ell-e}$ and that by assumption $\p^{\ell-e}\subseteq \p$ while $\varpi^{2\ell-e}a\in \O^*$. Consequently, if $y\in \O$ is such that $a+3by^2\in\p^{-\ell}$ then $|a|=|b|$ and $y\in \O^*$. It follows that indeed in this case $\Cu(a,b)=\Cu_0(a,b)$.
Note that $-3=(1+2\rho)^2\in\O^*$. It further follows from Hensel's lemma that there exists $y\in\O^*$ such that $a+3by^2\in\p^{-\ell}$ if and only if $ab^{-1}\in\O^{*2}$. 
The vanishing of $\Cu(a,b)$ when $\max(\abs{b},q)<\abs{a}$ and when $\abs{a}=\abs{b}>q$ and $ab^{-1}\not\in\O^{*2}$ follows. 

Assume now that $\abs{a}=\abs{b}>q$ and $ab^{-1}\in\O^{*2}$ and let
$v_0\in\O^*$ be such that $v_0^2=-3 ab^{-1}$. Then for $y\in\O$ we have $a+3by^2\in\p^{-\ell}$ if and only if $y\in \pm \frac13v_0+\p^{\ell-e}$.

Since $b\p^{3(\ell-e)}\subseteq \O$, for $z\in \p^{\ell-e}$ and $x=\pm (\frac13v_0+z)$ we have
\[
ax+bx^3\in \pm (\frac29av_0+bv_0z^2)+\O.
\]
Consequently,
\[
\Cu(a,b)=q^{-\ell}\sum_{\epsilon=\pm 1}\psi(\epsilon\frac29av_0) \sq(\epsilon bv_0).
\]
Since by assumption $\frac19 ab^{-1}\in F^{*2}$ it follows from \eqref{eq sq pr} that $\sq(\epsilon bv_0)=\sq(\epsilon \frac19 av_0)$ and the lemma follows. 
\end{proof}

Once again, the proof above may be regarded as a use of the method of stationary phase.

For future reference, for $\ell\in \Z$ let
\[
\Cu_\ell(a,b)=\int_{\val(x)=\ell} \psi(ax+bx^3)\ dx.
\]
The change of variables $x\mapsto t x$ shows that 
\begin{equation}\label{eq cu l}
\Cu_\ell(a,b)=\abs{t}\Cu_{\ell-\val(t)}(t a,t^3b),\ \ \ t\in F^*.
\end{equation}
As an immediate consequence of Lemma \ref{lem cobic exp} we record here that
\begin{equation}\label{eq cu l is zero}
\Cu_\ell(a,b)=0\ \ \ \text{whenever} \ \ \ \max(q^{-\ell}\abs{a},q^{-3\ell}\abs{b})>q\ \ \ \text{and} \ \ \ \abs{a^{-1}b}\ne q^{2\ell}.
\end{equation}

\subsection{The comparison of cubic exponential integrals and Kloosterman integrals}\label{cubic-comparison-123}
Duke and Iwaniec \cite{MR1210520} established the following identity between cubic exponential and Kloosterman sums.
Let $\chi$ be an order three character of $k_F^*$ ($\chi$ and $\chi^{-1}$ are the only such characters) and $\xi$ a non-trivial character of $k_F$. Then
\begin{equation}\label{eq di}
\sum_{x\in k_F} \xi(x+ax^3)= \chi(a)^{-1}\sum_{u\in k_F^*}\chi(u)\xi(u-3^{-3}a^{-1}u^{-1}).
\end{equation}
This identity, together with the work above which may be regarded as an application of the method of stationary phase, gives
the following key comparison of cubic exponential and Kloosterman sums.
\begin{corollary}\label{cor di}
We have
\[
\Cu(a,-3^{-3}c^{-1}d^{-1}a^3)=(t,c^{-1}d)_3\Kl(t;c,d)
\]
whenever
\begin{itemize}
\item either $\abs{a}=\abs{c}=\abs{d}=q$ and $3\nmid \val(t)$
\item or $\abs{a}=\abs{c}=\abs{d}>q$.
\end{itemize}
\begin{proof}
If $\abs{a}=\abs{c}=\abs{d}=q$ then
\[
\Cu(a,-3^{-3}c^{-1}d^{-1}a^3)=q^{-1}\sum_{x\in k_F}\psi(ax-3^{-3}c^{-1}d^{-1}a^3x^3) 
\]
while
\[
\Kl(t;c,d)=q^{-1}\sum_{u\in k_F^*}(t,u)_3\psi(cu+du^{-1}).
\]
If in addition $3\nmid \val(t)$ then the character $(t,\cdot)_3$ is not trivial on $\O^*$ and we apply \eqref{eq di} with 
\[
\xi(x+\p)=\psi(ax)\ \ \ x\in \O \ \ \  \text{and} \ \ \ \chi(u+\p)=(t,u)_3, \ \ \ u\in\O^*
\]
to obtain
\[
\Cu(a,-3^{-3}c^{-1}d^{-1}a^3)=q^{-1}(t,-3^3cda^{-2})_3\sum_{u\in k_F^*}(t,u)_3\xi(u+cda^{-2}u^{-1}). 
\]
Note that 
\[
(t,-3^3cda^{-2})_3(t,a^{-1}c)_3=(t,c^{-1}d)_3
\]
and therefore the change of variables $u\mapsto a^{-1}cu$ proves the desired identity. 

If $\abs{a}=\abs{c}=\abs{d}>q$ the lemma follows by comparing the formulas in Lemmas \ref{lem kl int} and \ref{lem cobic exp}.
\end{proof}
\end{corollary}

\subsection{An additional computation involving cubic exponential integrals}\label{cubic-additional-123}
We present one additional computation involving a cubic exponential integral.
We will use this evaluation when an orbital integral involves Kloosterman integrals without a cubic character -- that is, integrals $\Kl(t;c,d)$ with 
$\abs{c}=\abs{d}=q$ but with $3\mid \val(t)$ -- 
and so the result of Duke-Iwaniec is not applicable. 

\begin{lemma}\label{lem i=1 3val}
For $\abs{a}=\abs{b}\le q^{-3}$ we have
\[
3+\abs{a}^{-1}\sum_{k=0}^2\Cu_{\val(a)-1}(b^{-1}+a^{-1}\rho^k,-3^{-3}a^{-1}b^{-1})=\begin{cases} q& -(ab^{-1})^3\in1+\p \\ 0 & -(ab^{-1})^3\not\in1+\p.\end{cases}
\]
\end{lemma}
\begin{proof}
Note that by \eqref{eq cu l} with $t=\varpi^{-1}a$ we have
\[
\abs{a}^{-1}\Cu_{\val(a)-1}(b^{-1}+a^{-1}\rho^k,-3^{-3}a^{-1}b^{-1})=q\,\Cu_0(\varpi^{-1}(ab^{-1}+\rho^k),-3^{-3}\varpi^{-3}a^2b^{-1})
\]
and by assumption $-3^{-3}\varpi^{-3}a^2b^{-1}\in\O$ so that by Lemma \ref{lem cobic exp} we have 
\[
q\,\Cu_0(\varpi^{-1}(ab^{-1}+\rho^k),-3^{-3}\varpi^{-3}a^2b^{-1})=\begin{cases} q-1 & -ab^{-1}\in\rho^k+\p \\ -1 & -ab^{-1}\not\in\rho^k+\p.\end{cases}
\]
Since $ -ab^{-1}\in\rho^k+\p$ for at most one $k\in\{0,1,2\}$ and this is the case if and only if $(-ab^{-1})^3\in 1+\p$ the lemma follows.
\end{proof}

\section{The integral \texorpdfstring{$J(a,b)$}~}{}\label{The-integral-J(a,b)}

In this section we establish the following formula for $J(a,b)$.
\begin{proposition}\label{prop J formula}
For $\abs{b}\le \abs{a}$ we have
\[
(a,b)_3\, J(a,b)=
\]
\[
\begin{cases}
1 & \abs{a}=\abs{b}=1 \\
2q & \abs{a}=\abs{b}=q^{-1} \\
3\abs{a}^{-1}+\abs{ab}^{-1}\sum_{\ell=\lfloor\frac{\val(a)+1}2\rfloor}^{\val(a)-1} \sum_{k=0}^2 \Cu_\ell(b^{-1}+\rho^k a^{-1},-3^{-3}a^{-1}b^{-1}) & \abs{a}=\abs{b} \le q^{-2} \\
\abs{b}^{-1}\Cu(b^{-1},-3^{-3}a^{-1}b^{-1}) & \abs{b}< \abs{a}=1 \\
\abs{ab}^{-1}\sum_{k=0}^2 \Cu_{\frac{\val(a)}2}(b^{-1}+\rho^k a^{-1},-3^{-3}a^{-1}b^{-1}) & \abs{b}<\abs{a}<1 \text{ and }2|\val(a) \\
0 &  \abs{b}<\abs{a}<1 \text{ and }2\nmid \val(a) \\
0 & \abs{a}>1.
\end{cases}
\]
Furthermore,
\begin{equation}\label{eq J bar fe}
J(b,a)=\overline{J(a,b)}, \ \ \ a,\, b\in F^*.
\end{equation}
\end{proposition}
\subsection{Fixing coordinates}
By definition
\[
J(a,b)=\int_{A[a,b]} \kappa(u_1 g_{a,b} u_2)^{-1}\psi(u_1 u_2)\ du_1\ du_2
\]
where
\[
A[a,b]=\{(u_1,u_2)\in N\times N: u_1 g_{a,b} u_2\in K\}.
\]
We fix coordinates on $N\times N$ as follows.
In our computation we set 
\[
u_i=\begin{pmatrix} 1 & x_i & z_i \\ & 1 & y_i \\ & & 1 \end{pmatrix}\in N, \ \ \ i=1,2
\]
so that
\begin{equation}\label{eq open cell}
u_1g_{a,b}u_2=\begin{pmatrix} az_1 & az_1x_2-a^{-1}bx_1 & az_1z_2 -a^{-1}bx_1y_2+b^{-1}  \\ ay_1 & ay_1x_2 -a^{-1}b & ay_1z_2-a^{-1}by_2 \\ a & ax_2 & az_2\end{pmatrix}.
\end{equation}

We also recall that an element $g=(g_{i,j})\in \SL_3(F)$ satisfies $g\in K$ if and only if $g_{i,j}\in \O$ for all $1\le i,\,j\le 3$.

\subsection{A functional equation}
We observe the following properties on the big Bruhat cell.
\begin{lemma}\label{lem: kappa inv}
Let $u_1,u_2\in N$, $a,b\in F^*$ and  $g=u_1g_{a,b} u_2$. Then
\begin{itemize}
\item $\sigma(g,g^{-1})=1$.
\item Consequently, if in addition $g\in K$ then $\kappa(g^{-1})=\kappa(g)^{-1}=\overline{\kappa(g)}$.
\end{itemize}
\end{lemma}
\begin{proof}
The second part is immediate from the first since 
\[
1=\kappa(I_3)=\kappa(gg^{-1})=\kappa(g)\kappa(g^{-1})\sigma(g,g^{-1}). 
\]
For the first part note that $g_{a,b}^{-1}=g_{b,a}$ and
\[
\sigma(g,g^{-1})=\sigma(g_{a,b},g_{b,a})=1
\]
by the explicit formulas for $\sigma$ on monomial matrices in \cite[\S 2]{MR904946}. 
\end{proof}
The following consequence of the lemma reduces the computation of $J$ to the case $\abs{b}\le \abs{a}$.

\begin{corollary}\label{cor: Jfeq}
We have
\[
J(b,a)=\overline{J(a,b)},\ \ \ a,\, b\in F^*
\]
\end{corollary}
\begin{proof}
Since $g_{a,b}^{-1}=g_{b,a}$ the map $(u_1,u_2)\mapsto (u_2^{-1},u_1^{-1})$ maps $A[a,b]$ bijectively to $A[b,a]$. Since $\psi(u_2^{-1}u_1^{-1})=\overline{\psi(u_1u_2)}$ the Corollary is now immediate from Lemma \ref{lem: kappa inv}.
\end{proof}

\subsection{Computation of $\kappa$}
We provide a formula for $\kappa(g)$ for almost all $g$ in the intersection of $K$ with the big Bruhat cell (outside a measure zero set). 
Throughout the computation we freely use coordinates of $g$ as in \eqref{eq open cell}. 

The fact that $g\mapsto (g,\kappa(g))$ is a splitting of $K$ in $\widetilde\SL_3(F)$ is expressed by the equality
\[
\kappa(g_1 g_2)=\kappa(g_1)\kappa(g_2)\sigma(g_1,g_2),\ \ \ g_1,\,g_2\in K.
\]
Our computation uses properties of $\kappa$ which we recall below and an algorithm of Bump and Hoffstein for computing $\sigma(g_1,g_2)$. The algorithm is based on an explicit Bruhat decomposition of $g_1$ and it will be more convenient to write such a decomposition with factors in $\GL_3(F)$ rather than in $\SL_3(F)$.
For this reason we consider $\widetilde\SL_3(F)$ as a subgroup of the cubic cover $\widetilde\GL_3(F)$ of $\GL_3(F)$ constructed in Bump and Hoffstein \cite{MR904946} (following Matsumoto). We then freely use the algorithm for the computation of the two-cocycle $\sigma$ in \cite[\S 2]{MR904946}. In particular, in terms of the Bruhat decomposition, for $h\in \GL_3(F)$ we denote by $R(h)$ the unique 
monomial (i.e., scaled permutation) matrix such that $h\in NR(h)N$. Furthermore, we continue to denote by $\kappa$ an extension to a splitting of $\GL_3(\O)$ in $\widetilde\GL_3(F)$ and freely use the following properties of $\kappa$ (see \cite[\S1]{MR743816}):
\begin{itemize}
\item $\kappa$ is trivial on signed permutation matrices and on upper-triangular matrices in $\GL_3(\O)$,
\item $\kappa(\diag(g,1))=\kappa(\diag(1,g))=\kappa_2(g)$, $g\in \GL_2(\O)$ where
\[
\kappa_2(g)=\begin{cases} (c,d\det g^{-1})_3 & 0<\abs{c}<1 \\ 1 & \text{otherwise}\end{cases}, \ \ \ g=\begin{pmatrix} a & b \\ c & d \end{pmatrix}\in \GL_2(\O). 
\]
 
\end{itemize}
\begin{lemma}\label{lem k comp}
Let 
\[
g=\begin{pmatrix} 0 & a & b \\  u & 0 & 0 \\ x & c & d \end{pmatrix}\in K
\]
with $x,\,c\ne 0$. Then 
\[
\kappa(g)=(u,x)_3\kappa_2\sm{a}{b}{c}{d}=\begin{cases} (u,x)_3 & \abs{c}=1 \\ (c,d)_3(u,xc^{-1})_3 & \text{otherwise.}\end{cases}
\]
\end{lemma}
\begin{proof}
Note that $g=ksh$ with
\[
k=\begin{pmatrix}  1 &  & \\  & 1 &  \\  & u^{-1}x & 1 \end{pmatrix}, \ \ \ s=\begin{pmatrix} & 1 & \\ 1 & & \\  & & 1 \end{pmatrix}, \ \ \ h=\begin{pmatrix} u & 0 & 0 \\  0 & a & b \\ 0 & c & d \end{pmatrix} \in K
\]
and therefore $\kappa(g)=\kappa(ks)\kappa(h)\sigma(ks,h)$.
Clearly, $\kappa(h)=\kappa_2\sm{a}{b}{c}{d}$ and $\kappa(ks)=\sigma(k,s)$. 
We compute that $\sigma(k,s)=1$ and $\sigma(ks,h)=(u,x)_3$ using the algorithm in \cite[\S 2]{MR904946}.
\end{proof}

\begin{lemma}\label{lem k comp 32}
Let 
\[
g=\begin{pmatrix}  a &0 & b \\ c & 0 &  d\\  x & u & 0  \end{pmatrix}\in K
\]
with $x,\,c\ne 0$. Then 
\[
\kappa(g)=\kappa_2\sm{a}{b}{c}{d}(c,u^{-1}x)_3=\begin{cases} (c,dx)_3 & \abs{c}<1 \\ (c,x)_3 & \abs{c}=1. \end{cases}
\]
\end{lemma}
\begin{proof}
Note that $g=hsk$ with
\[
h=\begin{pmatrix}  a & b & 0 \\ c & d &  0\\  0&0  & u \end{pmatrix}, \ \ \ s=\begin{pmatrix}  1 &0 &0 \\ 0 & 0& 1\\  0& 1 &0  \end{pmatrix}, \ \ \ k=\begin{pmatrix} 1 & 0 & 0 \\  u^{-1}x & 1 & 0 \\ 0 & 0 & 1 \end{pmatrix} \in K
\]
and therefore $\kappa(g)=\kappa(h)\kappa(sk)\sigma(h,sk)$.
Clearly, $\kappa(h)=\kappa_2\sm{a}{b}{c}{d}$ and $\kappa(sk)=\sigma(s,k)$. 
We compute that $\sigma(s,k)=1$ and $\sigma(h,sk)=(c,u^{-1}x)_3$ using the algorithm in \cite[\S 2]{MR904946}.
\end{proof}

\begin{lemma}\label{lem k comp 22}
Let 
\[
g=\begin{pmatrix}  a &0 & b \\ x & u &  0\\  c & y & d  \end{pmatrix}\in K
\]
with $u\in \O^*$ and  $x,\,c,\,xy-cu\ne 0$. Then 
\[
\kappa(g)=(x,c)_3(c-u^{-1}xy,u^{-1}x)_3  \kappa_2\sm{a}{b}{c-u^{-1}xy}{d} =\begin{cases} (x,c)_3(c-u^{-1}xy,xd)_3 & \abs{cu-xy}<1 \\  (x,c)_3(c-u^{-1}xy,x)_3 & \abs{cu-xy}=1. \end{cases}
\]
\end{lemma}
\begin{proof} 
Note first that $g=sk$ where
\[
s=\begin{pmatrix}  1 &0 &0 \\ 0 & 0& 1\\  0& 1 &0  \end{pmatrix}, \ \ \ k=\begin{pmatrix} a & 0 & b \\  c & y &d \\ x & u & 0 \end{pmatrix} \in K
\]
so that $\kappa(g)=\kappa(s)\kappa(k)\sigma(s,k)$. Note that $\kappa(s)=1$ and we compute that
\[
\sigma(s,k)=(x,c)_3 
\]
using the algorithm in \cite[\S 2]{MR904946}.
In order to compute $\kappa(k)$ we note that $k=nhk'$ where
\[
n=\begin{pmatrix}  1 &0  & 0\\ 0 & 1 & u^{-1}y\\   0 &0 & 1 \end{pmatrix}\in N\cap K,\ \ \ h=\begin{pmatrix}   a & b & 0  \\ c-u^{-1}xy & d & 0 \\ 0 & 0 & u \end{pmatrix},  \ \ \ k'=\begin{pmatrix}  1 & 0 & 0  \\  0 & 0 & 1 \\ u^{-1}x  & 1 & 0 \end{pmatrix} \in K
\]
and therefore $\kappa(k)=\kappa(hk')=\kappa(h)\kappa(k')\sigma(h,k')$.
Clearly, $\kappa(h)=\kappa_2\sm{a}{b}{c-u^{-1}xy}{d}$ (note that $\det h=\det g=1$). Furthermore, it is easy to see that $\kappa(sk')=1$ and therefore also that $\kappa(k')=\sigma(s,sk')=1$. 
We further compute that 
\[
\sigma(h,k')=(c-u^{-1}xy,u^{-1}x)_3   
\] 
using the algorithm in \cite[\S 2]{MR904946}.
The lemma follows.
\end{proof}

\begin{lemma}\label{lem kappa}
Let $g=u_1g_{a,b}u_2\in K$ with $a,\,b\in F^*$ and $u_1,\,u_2\in N$. 
\begin{enumerate}
\item if $\abs{a}=1=\abs{b}$ then $\kappa(g)=1$.
\item\label{case 1=a>b} if $\abs{a}=1>\abs{b}$ then $\kappa(g)=(b,a)_3(a^{-1}b,y_2)_3$.
\item if $\abs{a},\,\abs{b}<1$ then at most one of $ay_1$ and $ax_2$ is in $\O^*$ and 
\begin{enumerate}
\item if $\abs{ay_1}=1$ then $\kappa(g)=(b,a)_3(y_1y_2,ab^{-1})_3(y_2,y_1)_3$
\item if $\abs{ax_2}=1$ then $\kappa(g)=(b,ax_2)_3(z_2x_2^{-1}-y_2,b^{-1}ax_2)_3$
\item if $\abs{ay_1},\,\abs{ax_2}<1$ then $ay_1x_2-a^{-1}b\in\O^*$ and if in addition $y_1\ne 0$ then 
\[
\kappa(g)=(b,a)_3(y_1,ab^{-1})_3(ay_1(x_2y_2-z_2),ay_1x_2-a^{-1}b)_3(b,x_2y_2-z_2)_3.
\]
\end{enumerate}
\end{enumerate}
\end{lemma}
\begin{proof}
If $\abs{a}=1=\abs{b}$ then $u_1,\,u_2\in K$ and therefore $\kappa(g)=\kappa(g_{a,b})=1$. If $\abs{a}=1>\abs{b}$ then $x_2,\,y_1,\,z_1,\,z_2\in \O$ and therefore 
$\kappa(g)=\kappa\sm{}{g_0}{a}{}$ where $g_0=\sm{-a^{-1}bx_1}{b^{-1}-x_1y_2a^{-1}b}{-a^{-1}b}{-a^{-1}by_2}$. Note that $\sm{}{g_0}{a}{}=\sm{g_0}{}{}{1}\sm{}{I_2}{a}{}$ and that $\sigma(\sm{g_0}{}{}{1},\sm{}{I_2}{a}{})=1$. Therefore 
\[
\kappa\sm{}{g_0}{a}{}=\kappa_2(g_0)=(-a^{-1}b,-by_2)_3=(b,a)_3(a^{-1}b,y_2)_3.
\]
For the rest of the proof assume that $\abs{a},\,\abs{b}<1$. Since the (2,2)-entry of $g$ is $ay_1x_2-a^{-1}b\in \O$ we cannot have both $ay_1$ and $ax_2$ in $\O^*$. 
If $\abs{ay_1}=1$ then
\[
k=\begin{pmatrix} 1 & \frac{a^{-1}b-y_1a x_2}{y_1a} & \frac{a^{-1}by_2-y_1z_2 a}{y_1a} \\ & 1 & \\ & & 1 \end{pmatrix} 
, \ \ \ k'=\begin{pmatrix}  1 & -\frac{z_1a}{y_1 a} & \\ & 1 & \\ & & 1\end{pmatrix}\in N\cap K,
\]
so that $\kappa(g)=\kappa(k'gk)$. Note that
\[
k'gk=\begin{pmatrix}    0 & \frac{(z_1-x_1y_1)b}{y_1a} & \frac{(z_1-x_1y_1)y_2b}{y_1a}+b^{-1} \\ y_1a & 0 &0 \\ a & \frac{b}{y_1a} & \frac{y_2b}{y_1a}   \end{pmatrix}
\]
and from Lemma \ref{lem k comp} it follows that
\[
\kappa(g)=(b(ay_1)^{-1},b(ay_1)^{-1}y_2)_3(ay_1,ab^{-1}ay_1)_3=(b,a)_3(y_1y_2,ab^{-1})_3(y_2,y_1)_3.
\]

If $\abs{ax_2}=1$ then
\[
k=\begin{pmatrix} 1 &  &  \\ & 1 & \frac{-z_2a}{x_2a} \\ & & 1 \end{pmatrix} 
, \ \ \ k'=\begin{pmatrix}  1 & & \frac{x_1a^{-1}b-z_1x_2a}{x_2a} \\ & 1 & \frac{a^{-1}b-y_1x_2a}{x_2a}\\ & & 1\end{pmatrix}\in N\cap K,
\]
so that $\kappa(g)=\kappa(k'gk)$. Note that
\[
k'gk=\begin{pmatrix}    x_1x_2^{-1}a^{-1} b & 0 & (z_2x_2^{-1}-y_2)x_1a^{-1}b+b^{-1} \\ x_2^{-1}a^{-1} b & 0 & (z_2x_2^{-1}-y_2)a^{-1}b \\ a & x_2 a & 0   \end{pmatrix}
\]
and from Lemma \ref{lem k comp 32} it follows that
\[
\kappa(g)=(x_2^{-1}a^{-1}b,(z_2x_2^{-1}-y_2)b)_3=(b,ax_2)_3(z_2x_2^{-1}-y_2,b^{-1}ax_2)_3.
\]
Finally, assume that $\abs{ax_2},\,\abs{ay_1}<1$. Any two by two minor of an element of $K$ must have an entry in $\O^*$. Since the (1,3)-minor of $g$ is $\sm{ay_1}{ay_1x_2-a^{-1}b}{a}{ax_2}$, the condition $g\in K$ implies that $ay_1x_2-a^{-1}b\in\O^*$. 
Therefore
\[
k=\begin{pmatrix} 1 &  &  \\ & 1 & \frac{a^{-1}by_2-y_1z_2 a}{y_1a x_2-a^{-1}b} \\ & & 1 \end{pmatrix} 
, \ \ \ k'=\begin{pmatrix}  1 & \frac{x_1a^{-1}b-z_1x_2a}{ y_1a x_2-a^{-1}b} & \\ & 1 & \\ & & 1\end{pmatrix}\in N\cap K
\]
and $\kappa(g)=\kappa(k'gk)$. Note that $k'gk$ has the form
\[
k'gk=\begin{pmatrix}    
* & 0 & *
\\ y_1a & y_1a x_2-a^{-1}b & 0 \\ a & x_2 a & \frac{(x_2y_2-z_2)b}{y_1a x_2-a^{-1}b} \end{pmatrix}.
\]
Assuming furthermore that $y_1\ne 0$ the rest of the computation follows from Lemma \ref{lem k comp 22}.
\end{proof}

\subsection{Evaluation of $J(a,b)$: some simple cases}
\subsubsection{The case $\max(\abs{a},\abs{b})>1$}
\begin{lemma}\label{lem J not in O}
If $\max(\abs{a},\abs{b})>1$ then $J(a,b)=0$.
\end{lemma}
\begin{proof}
For $g=(g_{i,j})\in \SL_3(F)$ let $\Delta_1(g)=g_{3,1}$ and $\Delta_2(g)=\det \sm{g_{2,1}}{g_{2,2}}{g_{3,1}}{g_{3,2}}$. Note that $\Delta_1(g)=a$ and $\Delta_2(g)=b$ for every $g\in Ng_{a,b}N$. Consequently, $A[a,b]$ is empty and therefore $J(a,b)=0$ unless $a,b \in \O$.  
\end{proof}
\subsubsection{The case $\abs{a}=\abs{b}=1$}
\begin{lemma}\label{lem J ab units}
If $a,b\in \O^*$ then
\[
(a,b)_3\,J(a,b)=1.
\]
\end{lemma}
\begin{proof}
If $a,b\in \O^*$ then $(a,b)_3=1$ and it easily follows from \eqref{eq open cell} that $A[a,b]=(N\cap K)\times (N\cap K)$. Since $\kappa$ is bi-$(N\cap K)$-invariant, $\kappa(g_{a,b})=1$ and $\psi$ is trivial on $N\cap K$ the lemma follows.
\end{proof}

\subsubsection{The case $1=\abs{a}>\abs{b}$}
\begin{lemma}\label{lem J a unit}
If $1=\abs{a}>\abs{b}$ then 
\[
(a,b)_3\,J(a,b)=\abs{b}^{-1}\Cu(b^{-1},-3^{-3} a^{-1}b^{-1}).
\]

\end{lemma}
\begin{proof}
It follows from \eqref{eq open cell} and Lemma \ref{lem kappa} part \eqref{case 1=a>b} that
\[
J(a,b)=(a,b)_3\int (ab^{-1},y_2)_3\psi(x_1+y_2) \ dx_1\ dy_2 
\]
where the integral is over $x_1$ and $y_2$ such that
\[
x_1b,y_2b,b^{-1}-x_1y_2a^{-1}b\in \O
\]
or equivalently such that 
\[
y_2 b\in \O^*,\ x_1\in (y_2b)^{-1}ab^{-1}+\O.
\]
Integrating over $x_1$ we obtain that
\[
J(a,b)=(a,b)_3\int_{\abs{y_2}=\abs{b}^{-1}} (ab^{-1},y_2)_3\psi(y_2+ab^{-2}y_2^{-1}) \ dy_2 
\]
and after the change of variables $y_2\mapsto b^{-1}y_2$ this becomes
\[
\abs{b}^{-1}\int_{\O^*}(ab^{-1},y_2)_3\psi(b^{-1}y_2+ab^{-1}y_2^{-1}) \ dy_2=\abs{b}^{-1}\Kl(ab^{-1};b^{-1},ab^{-1}).
\]
The lemma now follows from Corollary \ref{cor di}.
\end{proof}
\subsection{Calculation of $J(a,b)$ when $1>\abs{a}\ge\abs{b}$: subdivision and evaluations of $J_1$, $J_2$}
Assume now that $1>\abs{a}\ge \abs{b}$ and consider the conditions
\begin{enumerate}
\item $\abs{ay_1}=1$
\item $\abs{ax_2}=1$ 
\item $\abs{ay_1},\,\abs{ax_2}<1$.
\end{enumerate}
Let 
\[
J_i(a,b)=\int_{A_i[a,b]}\kappa(u_1g_{a,b}u_2)^{-1}\psi(u_1u_2)\ du_1\ du_2
\] 
where $A_i[a,b]$ is the intersection of $A[a,b]$ with the set defined by the condition $(i)$ for $i=1,2,3$. By Lemma \ref{lem kappa} we have
\begin{equation}\label{eq break j}
J(a,b)=J_1(a,b)+J_2(a,b)+J_3(a,b).
\end{equation}
\subsubsection{Computation of $J_1$} 
\begin{lemma}\label{lem J1}
Let $1>\abs{a}\ge\abs{b}$ then 
\[
(a,b)_3\,J_1(a,b)=\begin{cases} q & \abs{a}=\abs{b}=q^{-1} \\ 0 & \text{otherwise.}\end{cases}
\]
\end{lemma}
\begin{proof}

Note that according to Lemma \ref{lem kappa} we have
\[
J_1(a,b)=(a,b)_3\int (y_1y_2,a^{-1}b)_3(y_1,y_2)_3 \psi(x_1+x_2+y_1+y_2)\ dx_1\ dx_2\ dy_1\ dy_2\ dz_1\ dz_2
\]
where the integral is over
\[
ay_1\in \O^*,\, az_1, az_2, ax_2, y_1x_2 a-a^{-1}b, y_1z_2a-y_2a^{-1}b ,\,z_1x_2a-x_1a^{-1}b,  z_1z_2 a-x_1y_2a^{-1}b+b^{-1}\in \O
\]
or equivalently
\begin{equation}\label{eq j1 dom}
ay_1\in \O^*,\ az_1,\,az_2,\, x_2,\,y_1z_2a-y_2a^{-1}b ,\,x_1a^{-1}b,\,  z_1z_2 a-x_1y_2a^{-1}b+b^{-1}\in \O.
\end{equation}
We first show that $J_1(a,b)=0$ unless $\abs{a}=\abs{b}=q^{-1}$. Indeed, for $u,\,v\in \O^*$ the variable change
\[
(x_1,y_1,z_1,y_2,z_2)\mapsto (ux_1,vy_1,uvz_1,u^{-1}y_2,u^{-1}v^{-1}z_2)
\]
preserves the domain of integration. Averaging over $u,\,v$ we see that the integral $J_1(a,b)$ factors through
\[
\int_{\O^*}(v,a^{-1}by_2)_3\psi(vy_1)\ dv\int_{\O^*} (u,ab^{-1} y_1)_3\psi(ux_1+u^{-1}y_2)\ du.
\]
Note that in the domain of integration $\abs{y_1}=\abs{a}^{-1}>1$ and for such $y_1$ it follows from Lemma \ref{lem kl int} that $\int_{\O^*}(v,a^{-1}by_2)_3\psi(vy_1)\ dv=0$ unless $\abs{a}=q^{-1}$. Next, assume that $\abs{b}<\abs{a}=q^{-1}$.
Note that in the domain \eqref{eq j1 dom} we have 
\begin{equation}\label{eq absy2x1}
\abs{y_2}=\abs{b}^{-1}>q\ \ \ \text{and} \ \ \ \abs{x_1}=\abs{ab^{-1}}<\abs{y_2}.
\end{equation}
Indeed, in the domain of integration 
\[
\abs{ay_1z_2},\,\abs{az_1z_2}\le \abs{a}^{-1}<\abs{b}^{-1}
\]
and consequently,
\[
\abs{a^{-1}by_2}\le \abs{a}^{-1} \ \ \ \text{and}\ \ \ \abs{a^{-1}bx_1y_2}=\abs{b}^{-1}.
\]
Combined with the condition $a^{-1}bx_1\in \O$ we obtain \eqref{eq absy2x1}.
Consequently,  it follows from Lemma \ref{lem kl int} that 
$\int_{\O^*} (u,ab^{-1} y_1)_3\psi(ux_1+u^{-1}y_2)\ du=0$.

Assume now that $\abs{a}=\abs{b}=q^{-1}$. The domain of integration \eqref{eq j1 dom} is equivalently characterized by 
\[
\abs{y_1}=\abs{y_2}=q,\,x_2,\,az_1\in \O,\,z_2\in a^{-2}by_1^{-1}y_2+\O,\,x_1\in ab^{-1}y_2^{-1}(az_1z_2+b^{-1})+\p
\]
and in particular, the integrand is independent of $x_1,\,x_2,\,z_1$ and $z_2$ in this domain. After integrating over $x_1$ and $x_2$ we further integrate over $z_1$ and $z_2$. After the variable change 
\[
(y_1,y_2)\mapsto (a^{-1}y_1,a^{-1}y_2)
\]
we have
\[
J_1(a,b)=(b,a)_3 q^2\int_{\O^*}\int_{\O^*}(a^{-1},y_2y_1^{-1})_3\psi(a^{-1}(y_1+y_2))\ dy_1\ dy_2.
\]
The lemma now follows from Lemma \ref{lem gauss}.
\end{proof}

\subsubsection{Computation of $J_2$}

\begin{lemma}\label{lem J2}
Let $1>\abs{a}\ge\abs{b}$ then 
\[
(a,b)_3\,J_2(a,b)=\begin{cases} q & \abs{a}=\abs{b}=q^{-1} \\ 0 & \text{otherwise.}\end{cases}
\]
\end{lemma}
\begin{proof}
Note that according to Lemma \ref{lem kappa} we have
\[
J_2(a,b)=(a,b)_3\int (x_2,b)_3(b^{-1}ax_2,z_2x_2^{-1}-y_2)_3  \psi(x_1+x_2+y_1+y_2)\ dx_1\ dx_2\ dy_1\ dy_2\ dz_1\ dz_2
\]
where the integral is over
\[
ax_2\in \O^*,\ az_1,\,az_2,\,ay_1,\, y_1x_2 a-a^{-1}b,\,y_1z_2a-y_2a^{-1}b ,\,z_1x_2a-x_1a^{-1}b,\,  z_1z_2 a-x_1y_2a^{-1}b+b^{-1}\in \O.
\]
Note further that
\[
z_1z_2 a-x_1y_2a^{-1}b+b^{-1}=x_2^{-1}z_2(z_1x_2a-x_1a^{-1}b)+a^{-1}bx_1(x_2^{-1}z_2-y_2)+b^{-1}
\]
and the domain of integration is equivalently characterized by
\[
ax_2\in \O^*,\ az_1,\,az_2,\,y_1,\,y_2a^{-1}b ,\,z_1x_2a-x_1a^{-1}b,\,  a^{-1}bx_1(x_2^{-1}z_2-y_2)+b^{-1}\in \O.
\]
Since the variable change 
\[
(x_1,\,x_2,\,y_2,\,z_1,\,z_2)\mapsto (u^{-1}x_1,\,vx_2,\,uy_2,\,u^{-1}v^{-1}z_1,\,uvz_2)
\]
preserves the domain of integration for every $u,\,v\in \O^*$ an argument analogous to the proof of Lemma \ref{lem J1} shows that $J_2(a,b)=0$ unless $\abs{a}=\abs{b}=q^{-1}$.

Assume now that $\abs{a}=\abs{b}=q^{-1}$. 
After the variable change $y_2\mapsto y_2+x_2^{-1}z_2$ we have
\[
J_2(a,b)=(a,b)_3\int (x_2,b)_3(b^{-1}ax_2,y_2)_3   \psi(x_1+x_2)\ dx_1\ dx_2\ dy_1\ dy_2\ dz_1\ dz_2
\]
where the integral is over
\[
\abs{x_2}=q,\ az_1,\,az_2,\,y_1,\,y_2 ,\,z_1x_2a-x_1a^{-1}b,\,  a^{-1}bx_1y_2-b^{-1}\in \O
\]
or equivalently
\[
\abs{x_1}=\abs{x_2}=q,\,az_2,\,y_1\in \O ,\,z_1\in a^{-2}bx_2^{-1}x_1+\O,\,  y_2\in ab^{-2}x_1^{-1}+\p.
\]
Note that in this domain $ax_2,\,y_2\in \O^*$ and therefore
\[
(x_2,b)_3(b^{-1}ax_2,y_2)_3=(x_2y_2,b)_3=(ax_1^{-1}x_2,b)_3.
\]
That is, the integrand is independent of $y_1,\,z_1,\,y_2,\,z_2$ in the domain of integration. Integrating over these variables and applying the variable change
\[
(x_1,x_2)\mapsto (a^{-1}x_1,a^{-1}x_2)
\]
we obtain that 
\[
J_2(a,b)=(b,a)_3q^2\int_{\O^*}\int_{\O^*} (b,x_1x_2^{-1})_3\psi(a^{-1}(x_1+x_2))\ dx_1\ dx_2.  
\]
The lemma now follows from Lemma \ref{lem gauss}.
\end{proof}

\subsection{The calculation of $J(a,b)$ when $1>\abs{a}\ge\abs{b}$:  evaluation of $J_3$}
We turn to the evaluation of the third summand that contributes to $J(a,b)$.
Recall that by definition
\[
J_3(a,b)= \int_{A_3[a,b]}\kappa(u_1g_{a,b}u_2)^{-1}\psi(u_1u_2)\ du_1\ du_2
\]
where $A_3[a,b]$ is defined by the conditions (coordinates are as in \eqref {eq open cell})
\begin{enumerate}[a)]
\item\label{c1} $ay_1\in \p$
\item\label{c2} $ax_2\in \p$
\item\label{c3} $az_1\in \O$
\item\label{c4} $az_2\in \O$
\item\label{c5} $y_1x_2 a-a^{-1}b\in \O$
\item\label{c6} $y_1z_2a-y_2a^{-1}b\in \O$
\item\label{c7} $z_1x_2a-x_1a^{-1}b\in \O$
\item\label{c8} $z_1z_2 a-x_1y_2a^{-1}b+b^{-1}\in \O$.
\end{enumerate}

In the proof of the following simple lemma we refer to this list of defining conditions.
\begin{lemma}\label{lem A3basic}
Let $a,b\in F^*$ be such that $\abs{a}<1$.
\begin{enumerate}
\item\label{part x1} Every element of $A_3[a,b]$ satisfies $bx_1\in\p$.
\item\label{part z12 units} Every element of $A_3[a,b]$ satisfies $ay_1x_2-a^{-1}b,az_1,az_2\in\O^*$.
\end{enumerate}
\end{lemma}
\begin{proof}
It follows from \ref{c2} and \ref{c3} that $\abs{az_1x_2}<\abs{z_1}\le \abs{a}^{-1}$. Part \ref{part x1} of the lemma therefore follows from \ref{c7}.
For an element of $A_3[a,b]$ the matrix \eqref {eq open cell} reduces $\bmod \p$ to upper triangular form. Part \ref{part z12 units} follows.
\end{proof}

Note that the subset of $A_3[a,b]$ with $y_1=0$ is of measure zero (in fact, it is empty if $\abs{b}<\abs{a}$). It therefore follows from Lemma \ref{lem kappa} that 
\begin{multline}\label{eq j3}
J_3(a,b)=(a,b)_3\int_{A_3'[a,b]}  (y_1,a^{-1}b)_3(ay_1x_2-a^{-1}b,ay_1(x_2y_2-z_2))_3(x_2y_2-z_2,b)_3 \\ \psi(x_1+x_2+y_1+y_2)\ dx_1\ dx_2\ dy_1\ dy_2\ dz_1\ dz_2
\end{multline}
where $A_3'[a,b]$ is the subset of elements of $A_3[a,b]$ satisfying $y_1\ne 0$. 

\subsubsection{Integration over $z_1$}
We begin the computation of $J_3$ by integrating over the single variable $z_1$. We use the coordinates $(x_1,x_2,y_1,y_2,z_2)$ in $F^5$. For a set $D$ in $F^5$ let $D'$ be the set of elements in $D$ such that $y_1\ne 0$. Let $D[a,b]$ be the domain defined by the conditions
\begin{enumerate}[a)]
\item\label{dc1} $ay_1\in \p$
\item\label{dc2} $ax_2\in \p$
\item\label{dc3} $az_2\in \O^*$
\item\label{dc4} $y_1-a^{-1}by_2\in \O$
\item\label{dc5} $x_2-a^{-1}bx_1\in \O$
\item\label{dc6} $\abs{1-bx_1y_2}=\abs{ab^{-1}}$.
\end{enumerate}
In the rest of the section we refer to this list of conditions. 
\begin{lemma}\label{lem z1int}
Let $\abs{b}\le \abs{a}<1$. We have
\begin{multline*}
J_3(a,b)=(a,b)_3\abs{a^{-1}b}\int_{D'[a,b]}(y_1,ab^{-1}(ay_1x_2-a^{-1}b))_3(ax_2y_2-1,ay_1x_2-a^{-1}b)_3\\ (b,az_2(ax_2y_2-1))_3  \psi(-(b z_2)^{-1}(1-bx_1y_2)^{-1}x_1+y_1+x_2+az_2y_2)dx_1\ dx_2\ dy_1\ dy_2\ dz_2.
\end{multline*}
\end{lemma}
\begin{proof}
Applying the variable change
\[
x_1\mapsto az_1x_1, \ \ \ y_2\mapsto az_2 y_2
\]
to \eqref{eq j3} and Lemma \ref{lem A3basic}  we have
\begin{multline*}
J_3(a,b)=(a,b)_3\int  (y_1,a^{-1}b)_3(ay_1x_2-a^{-1}b,y_1(ax_2y_2-1))_3(z_2(ax_2y_2-1),b)_3  \\ \psi(a z_1x_1+y_1+x_2+az_2y_2)\ dx_1\ dx_2\ dy_1\ dy_2\ dz_1\ dz_2
\end{multline*}
where integration is now over the domain defined by $y_1\ne 0$ and 
\[
ay_1,\,ax_2\in \p,\, az_1,\,az_2\in \O^*,\,ay_1x_2,\, y_1-a^{-1}by_2,\,x_2-a^{-1}bx_1, az_1z_2(1-bx_1y_2)+b^{-1}\in \O.
\]
Note that the conditions 
\[
az_1,\,az_2\in \O^*,\, az_1z_2(1-bx_1y_2)+b^{-1}\in \O
\] 
are equivalent to the conditions
\[
az_2\in \O^*,\abs{1-bx_1y_2}=\abs{ab^{-1}},\, z_1\in -(az_2)^{-1}(1-bx_1y_2)^{-1}b^{-1}+a^{-1}b\O.
\]
Furthermore, it follows from Lemma \ref{lem A3basic} that $bx_1\in \p\subseteq \O$ and therefore in the domain of integration
\[
\psi(a z_1x_1)=\psi(-(b z_2)^{-1}(1-bx_1y_2)^{-1}x_1)
\]
so that the integrand is independent of $z_1$ in its $a^{-1}b\O$ coset. Integrating over $z_1$ we obtain
\begin{multline*}
J_3(a,b)=(a,b)_3\abs{a^{-1}b}\int (y_1,a^{-1}b)_3(ay_1x_2-a^{-1}b,y_1(ax_2y_2-1))_3(z_2(ax_2y_2-1),b)_3  \\ \psi(-(b z_2)^{-1}(1-bx_1y_2)^{-1}x_1+y_1+x_2+az_2y_2)\ dx_1\ dx_2\ dy_1\ dy_2\ dz_2
\end{multline*}
where the integral is over the subdomain of elements of $D'[a,b]$ that satisfy $ay_1x_2\in \O$. However, this condition is satisfied by any element of $D[a,b]$. Indeed, if for an element of $D'[a,b]$ either $x_2$ or $y_1$ is in $\O$ this follows from the conditions \ref{dc1} and \ref{dc2}, otherwise it follows from \ref{dc4} and \ref{dc5} that $\abs{x_2}=\abs{a^{-1}bx_1}$ and $\abs{y_1}=\abs{a^{-1}by_2}$ and therefore $\abs{ay_1x_2}=\abs{a^{-1}b^2x_1y_2}$. The equality \ref{dc6} implies that $\abs{bx_1y_2}\le \abs{ab^{-1}}$. The statement and the lemma follow.  
\end{proof}
\subsubsection{On the domain $D[a,b]$}
Let $D_0[a,b]$ be the subset of elements in $D[a,b]$ that satisfy either $a^{-1}bx_1\in\O$ or $a^{-1}by_2\in \O$, or equivalently (by \ref{dc4} and \ref{dc5}) either $x_2\in \O$ or $y_1\in\O$ and let $D_{>0}[a,b]=D[a,b]\setminus D_0[a,b]$ be its complement. Further let $D_{0x}[a,b]$(resp. $D_{0y}[a,b]$) be the subset of elements in $D[a,b]$ that satisfy $x_2\in\O$ (resp. $y_1\in \O$).
\begin{lemma}\label{lem d}
Let $\abs{b}\le \abs{a}<1$.
\begin{enumerate}
\item\label{part dsym} The set $D[a,b]$ is preserved by the variable interchange
\begin{equation}\label{eq dsym}
(x_1,y_1)\leftrightarrow (y_2,x_2).
\end{equation}
\item\label{part dxy} Every element of $D[a,b]$ satisfies $by_2,\,bx_1\in \p$.
\item\label{part d0} The set $D_0[a,b]$ is empty unless $\abs{a}=\abs{b}$.
\item\label{part d1} If $\abs{a}=\abs{b}$ then the set $D_{0x}[a,b]$ is defined by the conditions
\begin{equation}\label{eq d0x}
x_1,\,x_2\in \O,ay_1\in \p,\, y_2\in ab^{-1}y_1+\O,\,az_2\in\O^*.
\end{equation}
\item\label{part d2} If $\abs{a}=\abs{b}$ then the set $D_{0y}[a,b]$  is defined by the conditions
\begin{equation}\label{eq d0y}
y_1,\,y_2\in \O,ax_2\in \p,\, x_1\in ab^{-1}x_2+\O,\,az_2\in\O^*.
\end{equation}
\item\label{part d0kappa} If $\abs{a}=\abs{b}$ then every element of $D'_0[a,b]$ satisfies
\[
(y_1,a^{-1}b)_3(ay_1x_2-a^{-1}b,y_1(ax_2y_2-1))_3(z_2(ax_2y_2-1),b)_3 =(z_2,b)_3.
\]
\item\label{part dkappa} For every element of $D'_{>0}[a,b]$ we have 
\[
(y_1,a^{-1}b)_3(ay_1x_2-a^{-1}b,y_1(ax_2y_2-1))_3(z_2(ax_2y_2-1),b)_3 =(x_1,bx_1y_2-1)_3(z_2,b)_3.
\]
\end{enumerate}
\end{lemma}
\begin{proof}
Part \ref{part dsym} is straightforward from the defining properties \ref{dc1}--\ref{dc6} of $D[a,b]$.
It follows from \ref{dc1} and \ref{dc4} that an element of $D[a,b]$ satisfies $a^{-1}by_2\in a^{-1}\p$ or equivalently $by_2\in\p$. Part \ref{part dxy} now follows further applying the symmetry \eqref{eq dsym}. 

By part \ref{part dxy} an element of $D_{0x}[a,b]$ satisfies $bx_1y_2\in ay_2\O\subseteq ab^{-1}\p$. If $\abs{b}<\abs{a}$ this contradicts \ref{dc6}. By the symmetry \eqref{eq dsym} if $\abs{b}<\abs{a}$ then $D_{0y}[a,b]$ is also empty. Part \ref{part d0} follows.

Suppose that $\abs{a}=\abs{b}$. It is easy to see that an element of $D_{0x}[a,b]$ satisfies the conditions \eqref{eq d0x}.  
For the reverse inclusion note that an element of $F^5$  satisfying \eqref{eq d0x} satisfies $by_2\in ay_1+b\O\subseteq \p$ and therefore also $bx_1y_2\in \p$. Consequently, $\abs{1-bx_1y_2}=1=\abs{ab^{-1}}$. Part \ref{part d1} now easily follows. Part \ref{part d2} follows from part \ref{part d1} and the  symmetry \eqref{eq dsym}.
Every element of $D_0[a,b]$ satisfies $ay_1x_2,ax_2y_2\in\p$ while, by assumption, $a^{-1}b\in\O^*$ and therefore
\[
(t,ay_1x_2-a^{-1}b)_3=(t,a^{-1}b)_3\ \ \ \text{and}\ \ \ (t,ax_2y_2-1)_3=1,\ \ \ t\in F^*.
\]
Part \ref{part d0kappa} now follows by properties of the Hilbert symbol. 

For part \ref{part dkappa} we return to the more general setting $\abs{b}\le \abs{a}<1$. For an element of $D_{>0}[a,b]$ the conditions \ref{dc4} and \ref{dc5} imply that
\[
\abs{x_2}=\abs{a^{-1}bx_1}>1\ \ \ \text{and}\ \ \ \abs{y_1}=\abs{a^{-1}by_2}>1.
\]
In particular, from conditions \ref{dc4} and \ref{dc5} it follows that
\[
(ay_1,t)_3=(by_2,t)_3,\ \ \ t\in F^*
\]
and it easily follows that
\[
ay_1x_2-a^{-1}b\in a^{-1}b(bx_1y_2-1)+\p \ \ \ \text{and} \ \ \ ax_2y_2-1\in bx_1y_2-1+ab^{-1}\p.
\]
Combined with \eqref{dc6} it follows that
\[
(t,ay_1x_2-a^{-1}b)_3=(t,a^{-1}b(bx_1y_2-1))_3\ \ \ \text{and}\ \ \ (t,ax_2y_2-1)_3=(t,bx_1y_2-1)_3,\ \ \ t\in F^*.
\]
Part  \ref{part dkappa} now follows from the properties of Hilbert symbols.
\end{proof}

Based on Lemma \ref{lem d}, if $\abs{b}\le \abs{a}<1$ we have
\begin{equation}\label{eq J3}
J_3(a,b)=\J_0(a,b)+\J_{>0}(a,b)
\end{equation}
where 
\begin{multline}\label{eq j>0}
\J_{>0}(a,b)=\abs{a^{-1}b}\int_{D'_{>0}[a,b]} (x_1,bx_1y_2-1)_3(az_2,b)_3\\ \psi(-(b z_2)^{-1}(1-bx_1y_2)^{-1}x_1+y_1+x_2+az_2y_2)\ dx_1\ dx_2\ dy_1\ dy_2\ dz_2
\end{multline}
and $\J_0(a,b)=0$ unless $\abs{a}=\abs{b}$ in which case
\[
\J_0(a,b)=\int_{D'_0[a,b]} (az_2,b)_3\psi(-(b z_2)^{-1}(1-bx_1y_2)^{-1}x_1+y_1+x_2+az_2y_2)\ dx_1\ dx_2\ dy_1\ dy_2\ dz_2.
\]
\subsubsection{The computation of $\J_0$}
\begin{lemma}\label{lem j0}
Let $\abs{a}=\abs{b}<1$. We have
\[
(a,b)_3\J_0(a,b)=\begin{cases} 2\abs{a}^{-1} & 3\nmid \val(a)>1 \\ (1+q^{-1})\abs{a}^{-1} & 3|\val(a)\\ 0 & \val(a)=1. \end{cases}
\]
\end{lemma}
\begin{proof}
Suppose that $\abs{a}=\abs{b}<1$. Write $\text{\j}_1$, resp.\ $\text{\j}_2$, resp.\ $\text{\j}_3$ for the integral of 
\[
(az_2,b)_3\psi(-(b z_2)^{-1}(1-bx_1y_2)^{-1}x_1+y_1+x_2+az_2y_2)
\]
over the domain $D'_{0x}[a,b]$, resp.\ $D'_{0y}[a,b]$, resp.\ $D'_{0x}[a,b]\cap D'_{0y}[a,b]$ so that
\[
\J_0(a,b)=\text{\j}_1+\text{\j}_2-\text{\j}_3.
\]
By Lemma \ref{lem d}, part~\ref{part d1},  an element of $D_{0x}[a,b]$ satisfies 
$-(b z_2)^{-1}(1-bx_1y_2)^{-1}x_1,\,x_2\in \O$ and $az_2y_2\in a^2b^{-1}z_2y_1+\O$. Therefore
\[
\text{\j}_1=\int_{a^{-1}\O^*}\int_{a^{-1}\p}\int_{ab^{-1}y_1+\O}\int_\O\int_\O (az_2,b)_3\psi((a^2b^{-1}z_2+1)y_1)\ dx_1\ dx_2\ dy_2\ dy_1\ dz_2.
\]
Integrating over $x_1,x_2$ and $y_2$ and applying the variable change $z_2\mapsto a^{-2}bz_2$ we have
\begin{equation}\label{eq j1}
(a,b)_3\,\text{\j}_1=\abs{a}^{-1}\int_{\O^*}(z_2,b)_3\int_{a^{-1}\p}\psi((z_2+1)y_1)\ dy_1\ dz_2.
\end{equation}

If $\abs{a}=q^{-1}$ then $a^{-1}\p=\O$ and
\[
\int_\O \psi((z_2+1)y_1)\ dy_1=1
\]
is independent of $z_2\in\O^*$. Since $\val(b)=1$, we conclude in this case that
\[
(a,b)_3\,\text{\j}_1=\abs{a}^{-1}\int_{\O^*}(z_2,b)_3\ dz_2=0.
\]
If $\abs{a}<q^{-1}$ then $a\p^{-1}\subseteq \p$ and
\[
\int_{a^{-1}\p}\psi((z_2+1)y_1)\ dy_1=\begin{cases} q^{-1}\abs{a}^{-1} & z_2\in -1+a\p^{-1} \\ 0 & z_2\in \O^* \setminus (-1+a\p^{-1}).\end{cases}
\]
Consequently, 
\[
(a,b)_3\,\text{\j}_1=\abs{a}^{-2}q^{-1} \int_{-1+a\p^{-1}} (z_2,b)\ dz_2=\abs{a}^{-1}.
\]

By part \ref{part d2} of Lemma \ref{lem d} an element of $D_{0y}[a,b]$ satisfies 
$y_1,\,az_2y_2\in \O$ and $-(b z_2)^{-1}(1-bx_1y_2)^{-1}x_1\in -ab^{-1}(b z_2)^{-1}(1-bx_1y_2)^{-1})x_2+ \O$. Therefore
\begin{multline*}
\text{\j}_2=\int_{a^{-1}\O^*}\int_{a^{-1}\p}\int_{ab^{-1}x_2+\O}\int_\O\int_\O (az_2,b)_3\psi((1-ab^{-1}(b z_2)^{-1}(1-bx_1y_2)^{-1})x_2)\\ dy_1\ dy_2\ dx_1\ dx_2\ dz_2.
\end{multline*}
Note that in the domain of integration we have $bx_1y_2\in \p$. Applying the variable change $z_2\mapsto ab^{-2}(1-bx_1y_2)^{-1}z_2$ 
the integrand becomes independent of $x_1\in ab^{-1}x_2+\O$ and $ y_1,y_2\in \O$ and after integrating over these three variables we obtain
\[
(a,b)_3\, \text{\j}_2=\abs{a}^{-1}\int_{\O^*}\int_{a^{-1}\p} (z_2,b)_3\psi((1-z_2^{-1})x_2)\ dx_2\ dz_2.
\]
If $\abs{a}=q^{-1}$ we again see that the integral factors through $\int_{\O^*} (z_2,b)\ dz_2$ and therefore $\text{\j}_2=0$.
If $\abs{a}<q^{-1}$ then the computation is similar to that of $\text{\j}_1$. We have
\[
(a,b)_3\, \text{\j}_2=\abs{a}^{-2}q^{-1}\int_{1+a\p^{-1}}(z_2,b)\ dz_2=\abs{a}^{-1}.
\]

Finally, it easily follows from parts \ref{part d1} and \ref{part d2} of Lemma \ref{lem d} that $D_{0x}[a,b]\cap D_{0y}[a,b]$ is characterized by the conditions 
\[
x_1,\,x_2,\,y_1,\,y_2\in \O,\,az_2\in\O^*.
\]
Consequently, after integrating over $x_1, x_2, y_1$ and $y_2$ and applying the variable change $z_2\mapsto a^{-2}bz_2$ we have
\[
(a,b)_3\, \text{\j}_3=\abs{a}^{-1}\int_{\O^*}(z_2,b)\ dz_2=\abs{a}^{-1}(1-q^{-1})\delta_{3|\val(b)}.
\]
The lemma follows.
\end{proof}

\subsubsection{A simplification of $\J_{>0}(a,b)$}
\begin{lemma}\label{lem j>0}
Let $\abs{b}\le \abs{a}< 1$. Then
\[
\J_{>0}(a,b)=\sum_{\ell=\lfloor \frac{\val(a)+1}2\rfloor}^{\val(a)-1}\J_\ell(a,b)
\]
where
\begin{multline*}
\J_\ell(a,b)=\abs{a}^{-2}\int_{\abs{y_2}=q^{-2\ell}\abs{b}^{-1},\,\abs{1-y_2}=\abs{ab^{-1}}} (b^{-1}\varpi^\ell,y_2-1)_3\\ \Kl(b;-ab^{-2}\varpi^\ell(1-y_2)^{-1},\varpi^{-\ell} y_2)\Kl(b^{-1}(y_2-1)^{-1};a^{-1}\varpi^\ell ,a^{-1}by_2\varpi^{-\ell} )\ dy_2.
\end{multline*}
In particular, $\J_{>0}(a,b)=0$ if $\abs{a}=q^{-1}$.
\end{lemma}
\begin{proof}
Note that the domain $D_{>0}[a,b]$ is characterized by the conditions
\[
\abs{ab^{-1}}<\abs{x_1},\abs{y_2}<\abs{b}^{-1};\, az_2\in \O^*;\, y_1-a^{-1}by_2,\,x_2-a^{-1}bx_1\in \O;\, \abs{1-bx_1y_2}=\abs{ab^{-1}}.
\]
In particular, if $\abs{a}=q^{-1}$ then $D_{>0}[a,b]$ is empty and $\J_{>0}(a,b)=0$. For the rest of this section assume that $\abs{a}<q^{-1}$.
Applying the variable changes $y_1\mapsto y_1+a^{-1}by_2$ and $x_2\mapsto x_2+a^{-1}bx_1$ to \eqref{eq j>0},
the integrand becomes independent of $y_1,\,x_2\in\O$. After integrating over these two variables we obtain the expression
\begin{multline*}
\J_{>0}(a,b)=\abs{a^{-1}b}\int (x_1,bx_1y_2-1)_3  (az_2,b)_3\\
 \psi((a^{-1}b-(b z_2)^{-1}(1-bx_1y_2)^{-1})x_1+(a^{-1}b+az_2)y_2)\ dx_1\ dy_2\ dz_2
\end{multline*}
where integration is over the domain defined by the conditions
\[
az_2\in\O^*,\ \ \ \abs{ab^{-1}}<\abs{x_1},\abs{y_2}<\abs{b}^{-1}\ \ \  \text{and} \ \ \ \abs{1-bx_1y_2}=\abs{ab^{-1}}. 
\]

Next, applying the variable change 
\[
x_1\mapsto az_2x_1, \ y_2\mapsto (az_2)^{-1}y_2
\]
we have
\begin{multline*}
\J_{>0}(a,b)=\abs{a^{-1}b}\int (x_1,bx_1y_2-1)_3  (az_2,b(bx_1y_2-1))_3 \\ \psi(bz_2x_1-ab^{-1}(1-bx_1y_2)^{-1}x_1+a^{-2}by_2z_2^{-1}+y_2)\ dx_1\ dy_2\ dz_2
\end{multline*}
where integration is over the same domain.
Now applying the variable change 
\[
z_2\mapsto z_2x_1^{-1},\,y_2\mapsto y_2(bx_1)^{-1}
\]
we have
\begin{multline*}
\J_{>0}(a,b)=\abs{a}^{-1}\int \abs{x_1}^{-2}(b,x_1)_3  (az_2,b(y_2-1))_3 \\ \psi(bz_2-ab^{-1}(1-y_2)^{-1}x_1+a^{-2}y_2z_2^{-1}+b^{-1}y_2x_1^{-1})\ dx_1\ dy_2\ dz_2
\end{multline*}
where integration is over the domain defined by
\[
\abs{ab^{-1}}<\abs{x_1}=\abs{az_2}<\abs{b}^{-1},\ \ \ \abs{ax_1}<\abs{y_2}<\abs{x_1}\ \ \  \text{and} \ \ \ \abs{1-y_2}=\abs{ab^{-1}}. 
\]
Note further that the condition $\abs{1-y_2}=\abs{ab^{-1}}$ implies that $\abs{y_2}\le \abs{ab^{-1}}$ and combined with the condition $\abs{ab^{-1}}<\abs{x_1}$ it 
implies that $\abs{y_2}<\abs{x_1}$. Thus the condition $\abs{y_2}<\abs{x_1}$ may be omitted.

We express the integral $\J_{>0}(a,b)$ as a sum over $\ell=1,\dots,\val(a)-1$ of integrals over the subdomain 
where $\abs{x_1}=\abs{b}^{-1}q^{-\ell}$, that is, over the domain defined by the conditions
\[
\abs{x_1}=\abs{b}^{-1}q^{-\ell},\ \abs{z_2}=\abs{ab}^{-1}q^{-\ell},\ q^{-\ell}\abs{b^{-1}a}<\abs{y_2}\ \ \ \text{and}\ \ \ \abs{1-y_2}=\abs{ab^{-1}}.
\]
Applying the variable change
\[
x_1\mapsto b^{-1}\varpi^\ell x_1,\ z_2\mapsto a^{-1}b^{-1}\varpi^\ell z_2
\]
to the $\ell$-th summand we have
\begin{multline*}
\J_{>0}(a,b)=\abs{a}^{-2}\sum_{\ell=1}^{\val(a)-1}\int (b^{-1}\varpi^\ell,y_2-1)_3\int_{\O^*}\int_{\O^*}(b,x_1)_3  (z_2,b(y_2-1))_3  \\ \psi(a^{-1}\varpi^\ell z_2-ab^{-2}\varpi^\ell(1-y_2)^{-1} x_1+a^{-1}b \varpi^{-\ell} y_2 z_2^{-1}+\varpi^{-\ell} y_2x_1^{-1})\ dx_1\ dz_2 \ dy_2\\=
 \abs{a}^{-2}\sum_{\ell=1}^{\val(a)-1}\int  (b^{-1}\varpi^\ell,y_2-1)_3\\ \Kl(b;-ab^{-2}\varpi^\ell(1-y_2)^{-1},\varpi^{-\ell} y_2)\Kl(b^{-1}(y_2-1)^{-1};a^{-1}\varpi^\ell ,a^{-1}by_2\varpi^{-\ell} )\ dy_2
\end{multline*}
where in the $\ell$-th integral $y_2$ is integrated over the domain defined by the conditions
\[
q^{-\ell}\abs{b^{-1}a}<\abs{y_2}\ \ \ \text{and}\ \ \ \abs{1-y_2}=\abs{ab^{-1}}.
\]

It follows from Lemma \ref{lem kl int} that in the domain of integration above we have 
\[
\Kl(b^{-1}(y_2-1)^{-1};a^{-1}\varpi^\ell ,a^{-1}by_2\varpi^{-\ell} )=0
\] 
unless 
\[
\abs{y_2}=q^{-2\ell}\abs{b}^{-1}.
\] 
Indeed, if $\ell<\val(a)-1$ then $\abs{a^{-1}\varpi^\ell}>q$ and for $\ell=\val(a)-1$ we have $\abs{a^{-1}\varpi^\ell}=q$ while in the domain of integration $\abs{a^{-1}by_2\varpi^{-\ell}}\ge q$. Since in the domain of the $\ell$-th integral we have $\abs{y_2}\le \abs{ab^{-1}}$ and since $q^{-2\ell}\le \abs{a}$ if and only if $\ell\ge \lfloor \frac{\val(a)+1}2\rfloor$ the lemma follows.
\end{proof}

In order to proceed with the computation of $\J_\ell(a,b)$ we separate into the two cases:
\begin{itemize}
\item Case 1: $\ell<\val(b)-1$ or $\ell=\val(b)-1$ and $3\nmid \val(b)$.
\item Case 2: $\ell=\val(b)-1$ and $3|\val(b)$.
\end{itemize}
The computation in Case 1 requires some preparation that we carry out first.
\subsubsection{}\label{ss G int} Let $a,\,b\in F^*$ and let $\ell$  be an integer such that
\begin{itemize}
\item either $\abs{a}=\abs{b}\le q^{-2}$ and $\frac{\val(a)}2\le \ell\le \val(a)-1$ 
\item or $\abs{b}<\abs{a}<1$ and $\ell=\frac{\val(a)}2$.
\end{itemize}
For $j\ge 1$ let 
\[
\Gamma_j=\{(x,y)\in\O\times \O: x^3-y^3 \in a\p^{-j}\}
\]
and let
\[
G_j^\ell(a,b)=\int_{\Gamma_j}\psi(\varpi^\ell b^{-1}x-3^{-3}\varpi^{3\ell} a^{-1}b^{-1}x^3+\varpi^\ell a^{-1} y)\ dx\ dy.
\]
Note that the dependence of $\Gamma_j$ on $a$ is only via $\val(a)$. 
Our goal in this subsection is to compute $G_\ell^\ell(a,b)$ and $G_{\ell+1}^\ell(a,b)$. Let $j\in\{\ell,\ell+1\}$ and set
\[
m=\val(a)-j. 
\]
Note that with the above assumptions $m\ge 0$. We begin with the explication of $\Gamma_j$.
\begin{lemma}\label{lem gamma}
We have the disjoint union
\[
\Gamma_j=\left(\p^{-\lfloor -\frac{m}3\rfloor} \times \p^{-\lfloor -\frac m3\rfloor}\right)\sqcup \mathop{\sqcup}\limits_{n=0}^{\lfloor\frac{m-1}3\rfloor} \mathop{\sqcup}\limits_{k=0}^2 \{(x,y):\abs{x}=q^{-n},\ y\in \rho^k x+\p^{m-2n}\}.
\]
(In particular, if $m=0$ then $\Gamma_j=\O\times \O$.)
\end{lemma}
\begin{proof}
The case $m=0$ is straightforward. Assume that $m>0$. By definition we have
\[
\Gamma_j=\{(x,y)\in\O\times \O: x^3-y^3\in\p^m\}.
\]
Note that for $(x,y)\in \Gamma_j$ we have $x^3\in \p^m$ if and only if $y^3\in \p^m$ and that for $x\in F$ we have $x^3\in \p^m$ if and only if $x\in \p^{-\lfloor -\frac{m}3\rfloor}$. Suppose that $(x,y)\in \Gamma_j$ with $x^3\not\in \p^m$ and set $\abs{x}=q^{-n}$. By assumption $0\le n\le  \lfloor\frac{m-1}3\rfloor$ and thus $\abs{x}=\abs{y}$. Note that in the decomposition
\[
y^3-x^3=(y-x)(y-\rho x)(y-\rho^2 x)\in \p^m
\]
two of the three factors must also be of absolute value $q^{-n}$ and the third must therefore be in $\p^{m-2n}$. This gives the inclusion of $\Gamma_j$ in the disjoint union. The above decomposition also explains the other inclusion. 
\end{proof}

\begin{lemma}\label{lem G comp}
With the above assumptions we have
\[
G_{\ell+1}^\ell(a,b)=0
\]
and
\[
G_\ell^\ell(a,b)=q^{-2}\delta_{\ell,\val(a)-1}\delta_{\val(a),\val(b)}+q^{2\ell}\abs{a}\sum_{k=0}^2 \Cu_\ell(b^{-1}+a^{-1}\rho^k,-3^{-3}a^{-1}b^{-1}).
\]
In particular, if $\abs{a}=\abs{b}$ then
\[
G_\ell^\ell(a,b)=G_\ell^\ell(b,a).
\]
\end{lemma}
\begin{proof}
It follows from Lemma \ref{lem gamma} that
\[
G_j^i(a,b)=\G_0(a,b)+\sum_{n=0}^{\lfloor\frac{m-1}3\rfloor}\sum_{k=0}^2 \G_{n,k}(a,b)
\]
where
\[
\G_0(a,b)=\int_{\p^{-\lfloor -\frac{m}3\rfloor}}\psi(\varpi^\ell  a^{-1} y)\ dy
\int_{\p^{-\lfloor -\frac{m}3\rfloor}}\psi(\varpi^\ell b^{-1} x-3^{-3}\varpi^{3\ell}a^{-1}b^{-1}x^3)\ dx
\]
and
\[
\G_{n,k}(a,b)=\int_{\p^{m-2n}}\psi(\varpi^\ell a^{-1}z)\ dz\int_{\abs{x}=q^{-n}} \psi(\varpi^\ell (b^{-1}+\rho^k a^{-1}) x-3^{-3}\varpi^{3\ell}a^{-1}b^{-1}x^3)\ dx.
\]
Note that writing $j=\ell+e$ with $e\in \{0,1\}$, since $\abs{\varpi^\ell a^{-1}}=q^{m+e}$ we have
\[
\int_{\p^{-\lfloor -\frac{m}3\rfloor}}\psi(\varpi^\ell  a^{-1}  y)\ dy=\begin{cases} q^{\lfloor -\frac{m}3\rfloor} & e=0\text{ and }m\le 1 \\ 0 &  e=1\text{ or }m>1. \end{cases}
\]
In our set up, the condition $e=0$ and $m\le 1$ is met if and only if one of these conditions holds:
\begin{enumerate}
\item\label{case b=a} $\val(b)=\val(a)$, $\ell=\val(a)-1$ and $e=0$, so that $m=1$;
\item\label{case b<a} $\val(b)>\val(a)=2$ and $e=0$, so that $\ell=1$ and $m=1$.
\end{enumerate}
Note that in both these cases $-\lfloor -\frac{m}3\rfloor=1$ and 
\[
\int_\p \psi(\varpi^\ell b^{-1} x-3^{-3}\varpi^{3\ell}a^{-1}b^{-1}x^3)\\ dx=q^{-1}\Cu(\varpi^{\ell+1}b^{-1},-3^{-3}\varpi^{3(\ell+1)}a^{-1}b^{-1})=\begin{cases} q^{-1} &  \text{ in case }\eqref{case b=a} \\ 0 & \text{ in case }\eqref{case b<a}.\end{cases}
\]
The last equality is a consequence of Lemma \ref{lem cobic exp}. Indeed, in case \eqref{case b=a} both $\varpi^{\ell+1}b^{-1}$ and $-3^{-3}\varpi^{3(\ell+1)}a^{-1}b^{-1}$ are in $\O$ and in case \eqref{case b<a} we have $1<\abs{\varpi^{\ell+1}b^{-1}}=q^{-2}\abs{b}^{-1}$ while $\abs{-3^{-3}\varpi^{3(\ell+1)}a^{-1}b^{-1}}=q^{-4}\abs{b}^{-1}$.
We conclude that
\[
\G_0(a,b)=\begin{cases} q^{-2} & \val(a)=\val(b), \,\ell=j=\val(a)-1 \\ 0 & \text{otherwise.}\end{cases}
\]

Similarly,
\[
\int_{\p^{m-2n}}\psi(\varpi^\ell a^{-1}z)\ dz=\begin{cases} q^\ell\abs{a}& n=0\text{ and }\ell=j \\ 0 & \text{otherwise.} \end{cases}
\]
Also, applying the variable change $x\mapsto \varpi^{-\ell}x$ we have
\[
\int_{\O^*}\psi(\varpi^\ell (b^{-1}+\rho^k a^{-1}) x-3^{-3}\varpi^{3\ell}a^{-1}b^{-1}x^3)\ dx=q^\ell\Cu_\ell(b^{-1}+a^{-1}\rho^k,-3^{-3}a^{-1}b^{-1}).
\]
Therefore, 
\[
\G_{n,k}(a,b)=\begin{cases} q^{2\ell}\abs{a}\Cu_\ell(b^{-1}+a^{-1}\rho^k,-3^{-3}a^{-1}b^{-1}) & n=0\\ 0 & \text{otherwise.}\end{cases}
\]
The formulas for $G_{\ell+1}^\ell(a,b)$ and $G_\ell^\ell(a,b)$ follow. Applying \eqref{eq cu l} with $t=\rho^{-k}$ we observe that if $\abs{a}=\abs{b}$ then the sum $\sum_{k=0}^2 \Cu_\ell(b^{-1}+a^{-1}\rho^k,-3^{-3}a^{-1}b^{-1})$ is symmetric with respect to $a$ and $b$. The last equality and the lemma therefore follow from the formula for $G_\ell^\ell(a,b)$.
\end{proof}

\subsubsection{Computation of $\J_\ell(a,b)$-Case 1:} Here we apply Corollary \ref{cor di} in order to provide a formula for $\J_\ell(a,b)$ whenever 
either $\ell<\val(b)-1$ or $\ell=\val(b)-1$ and $3\nmid \val(b)$.
\begin{lemma}\label{lem ji}
If 
\begin{itemize}
\item either $\abs{a}=\abs{b}\le q^{-2}$ and $ \frac{\val(a)}2\le \ell\le \val(a)-2$ 
\item or  $\abs{a}=\abs{b}\le q^{-2}$, $\ell=\val(a)-1$ and $3\nmid \val(b)$ 
\item or $\abs{b}<\abs{a}<1$ and $\ell=\frac{\val(a)}2$ (here necessarily $2\mid\val(a)$)
\end{itemize}
then
\[
(a,b)_3\,\J_\ell(a,b)=
\delta_{\ell,\val(a)-1}\delta_{\val(a),\val(b)}\abs{a}^{-1}+\abs{ab}^{-1}\sum_{k=0}^2 \Cu_\ell(b^{-1}+a^{-1}\rho^k,-3^{-3}a^{-1}b^{-1}).
\]
If $\abs{b}<\abs{a}<1$ and $\frac{\val(a)}2< \ell \le \val(a)-1$ then $\J_\ell(a,b)=0$.
\end{lemma}
\begin{proof}
With the assumption that either $\ell<\val(b)-1$ or $\ell=\val(b)-1$ and $3\nmid \val(b)$, by Corollary \ref{cor di}, in the domain of integration we have
\[
\Kl(b;-ab^{-2}\varpi^\ell(1-y_2)^{-1},\varpi^{-\ell} y_2)=(b,\varpi^{2\ell}ay_2^{-1}(1-y_2)^{-1})_3 \Cu(\varpi^\ell b^{-1},3^{-3}\varpi^{3\ell} a^{-1}b^{-1}(y_2^{-1}-1))
\]
and
\[
\Kl(b^{-1}(y_2-1)^{-1};a^{-1}\varpi^\ell ,a^{-1}by_2\varpi^{-\ell} )=(b^{-1}(y_2-1)^{-1},\varpi^{2\ell}b^{-1}y_2^{-1})_3\Cu(a^{-1}\varpi^\ell,-3^{-3}\varpi^{3\ell}a^{-1}b^{-1}y_2^{-1}).
\]
Note further that
\[
(b,\varpi^{2\ell}ay_2^{-1}(1-y_2)^{-1})_3(b^{-1}(y_2-1)^{-1},\varpi^{2\ell}b^{-1}y_2^{-1})_3=(b,a)_3(y_2-1,b^{-1}\varpi^\ell)_3.
\]
Plugging this into the formula in Lemma \ref{lem j>0} defining $\J_\ell(a,b)$ we obtain that
\begin{multline*}
(a,b)_3\,\J_\ell(a,b)=\abs{a}^{-2}\int_{\abs{y_2}=q^{-2\ell}\abs{b}^{-1},\,\abs{1-y_2}=\abs{ab^{-1}}} \Cu(\varpi^\ell b^{-1},3^{-3}\varpi^{3\ell} a^{-1}b^{-1}(y_2^{-1}-1)) \\\Cu(a^{-1}\varpi^\ell,-3^{-3}\varpi^{3\ell}a^{-1}b^{-1}y_2^{-1})\ dy_2.
\end{multline*}
Note that if $\abs{b}<\abs{a}<1$ then the domain of integration is empty (so that $\J_\ell(a,b)=0$) unless $2\ell=\val(a)$. For the rest of the proof we assume that either $\abs{a}=\abs{b}$ or $2\ell=\val(a)$.

The domain of integration over $y_2$ is explicated further as follows. 
If either 
\begin{itemize}
\item $\abs{b}<\abs{a}$, $\val(a)$ is even and $\ell=\frac{\val(a)}2$ or 
\item $\abs{a}=\abs{b}$ and $\ell>\frac{\val(a)}2$ 
\end{itemize}
then the domain is characterized by $\abs{y_2}=q^{-2\ell}\abs{b}^{-1}$. If $\abs{b}=\abs{a}$, $\val(a)$ is even and $\ell=\frac{\val(a)}2$ then the domain is $\O^*\setminus (1+\p)$. 
In order to unify notation let
\[
\delta=\begin{cases} 1 & \abs{b}=\abs{a}, \val(a)\text{ is even and }\ell=\frac{\val(a)}2 \\ 0 & \text{otherwise.} \end{cases}
\]
Writing the cubic exponential integrals attached to the functions $\Cu$, we obtain
\begin{multline*}
(a,b)_3\,\J_\ell(a,b)=\abs{a}^{-2}\int_\O \int_\O \left[\int_{\abs{y_2}=q^{-2\ell}\abs{b}^{-1},\,\abs{1-y_2}=\abs{ab^{-1}}} \psi(3^{-3}\varpi^{3\ell}a^{-1}b^{-1}(x^3-y^3)y_2^{-1})\ dy_2 \right]\\ \psi(\varpi^\ell b^{-1}x-3^{-3}\varpi^{3\ell} a^{-1}b^{-1}x^3+a^{-1}\varpi^\ell y)\ dx\ dy.
\end{multline*}
The inner integral over $y_2$ may be evaluated using the straightforward formulas
\begin{equation}\label{eq psiy}
\int_{\abs{y_2}=q^{-2\ell}\abs{b}^{-1}}\psi(ty_2^{-1})\ dy_2=\begin{cases} q^{-2\ell}\abs{b}^{-1}(1-q^{-1}) & t\in b^{-1}\p^{2\ell} \\ -q^{-2\ell-1}\abs{b}^{-1}& \abs{t}=q^{2\ell+1}\abs{b} \\ 0 & \abs{t}>q^{2\ell+1}\abs{b} \end{cases} 
\end{equation}
and
\[
\int_{\O^*\setminus (1+\p)}\psi(ty_2^{-1})\ dy_2=\begin{cases} 1-2q^{-1} & t\in \O \\ -q^{-1}[1+\psi(t)] & \abs{t}=q \\ 0 & \abs{t}>q. \end{cases}
\]
In the notation of \S\ref{ss G int} we conclude that
\begin{multline*}
(a,b)_3\,\abs{a^2b}q^{2\ell}\J_\ell(a,b)
=(1-\delta)\left[ (1-q^{-1}) G_\ell^\ell(a,b)-q^{-1}(G_{\ell+1}^\ell(a,b)-G_\ell^\ell(a,b))\right]+\\ \delta\left[(1-2q^{-1})G_\ell^\ell(a,b)-q^{-1}(G_{\ell+1}^\ell(a,b)-G_\ell^\ell(a,b))-q^{-1}(G_{\ell+1}^\ell(b,a)-G_\ell^\ell(b,a))\right]=G_\ell^\ell(a,b).
\end{multline*}
The last equality and the lemma follow from Lemma \ref{lem G comp}.
\end{proof}
\subsubsection{Computation of $\J_{\val(a)-1}(a,b)$-Case 2:} 
\begin{lemma}\label{lem j1 case2}
Assume that $\abs{a}=\abs{b}<1$ and $3 |\val(b)$. 
Then
\[
\abs{a}(a,b)_3\,\J_{\val(a)-1}(a,b)=\begin{cases} q-(1+q^{-1}) & -(ab^{-1})^3\in1+\p \\ -(1+q^{-1}) & -(ab^{-1})^3\in \O^*\setminus (1+\p). \end{cases}
\]
\end{lemma}
\begin{proof}
Since $3|\val(b)=\val(a)>0$, for $\ell=\val(a)-1$ we have $q^{-2\ell}\abs{b}^{-1}=q^{2}\abs{a}<1$. It follows from Lemma \ref{lem j>0} that
\begin{multline*}
\J_{\val(a)-1}(a,b)=\abs{a}^{-2}\int_{\O^*}\int_{\O^*} \psi(-ab^{-2}\varpi^{\val(a)-1}u+a^{-1}\varpi^{\val(a)-1} v) \\
\int_{\abs{y_2}=q^2\abs{a}} \psi(\varpi^{1-\val(a)} (u^{-1}+a^{-1}bv^{-1})y_2)\ dy_2\ du\ dv.
\end{multline*}
Note that for $u,v\in\O^*$ we have
\[
u^{-1}  +a^{-1}bv^{-1}\in\begin{cases} \p & v\in -a^{-1}bu+\p \\ \O^* & \text{otherwise.}\end{cases}
\]
Applying \eqref{eq psiy} to integrate over $y_2$ we deduce that
\[
\int_{\abs{y_2}=q^2\abs{a}} \psi(\varpi^{1-\val(a)} (u^{-1}+a^{-1}bv^{-1})y_2)\ dy_2=\begin{cases} (q^2-q)\abs{a} & v\in -a^{-1}bu+\p \\ -q\abs{a}& \text{otherwise} \end{cases}
\]
and therefore
\begin{multline*}
\J_{\val(a)-1}(a,b)=\abs{a}^{-1} \left[(q-1)\int_{\O^*} \psi(-\varpi^{\val(a)-1}(ab^{-2}+a^{-2}b)u)\ du \right. \\ \left.
-q\int_{\O^*} \psi(-ab^{-2}\varpi^{\val(a)-1}u) \int_{\O^*\setminus (-a^{-1}bu+\p)} \psi(a^{-1}\varpi^{\val(a)-1} v) \ dv \ du \right].
\end{multline*}

We have
\[
-q \int_{\O^*\setminus (-a^{-1}bu+\p)} \psi(a^{-1}\varpi^{\val(a)-1} v) \ dv=1+\psi(-a^{-2}b\varpi^{\val(a)-1} u).
\]
Applying \eqref{eq psiy} again we obtain
\[
\J_{\val(a)-1}(a,b)=\abs{a}^{-1} \left[q\int_{\O^*} \psi(-\varpi^{\val(a)-1}(ab^{-2}+a^{-2}b)u)\ du-q^{-1}\right]
\]
and
\[
\int_{\O^*} \psi(-\varpi^{\val(a)-1}(ab^{-2}+a^{-2}b)u)\ du=\begin{cases} 1-q^{-1} & -(ab^{-1})^3\in1+\p \\ -q^{-1} & -(ab^{-1})^3\in \O^*\setminus (1+\p). \end{cases}
\]
Since $(a,b)_3=1$ in this case, the lemma readily follows.
\end{proof}

\subsubsection{Completion of the proof of Proposition \ref{prop J formula}.}
The functional equation \eqref{eq J bar fe} is proved in Corollary \ref{cor: Jfeq}. Assume that $\abs{b}\le \abs{a}$. The computation of $J(a,b)$ for the cases $\abs{a}>1$, $1=\abs{a}=\abs{b}$ and $1=\abs{a}>\abs{b}$ is respectively taken care of in Lemma \ref{lem J not in O}, Lemma \ref{lem J ab units} and Lemma \ref{lem J a unit}.
In the case $\abs{a}=q^{-1}$ it follows from \eqref{eq J3}, Lemma \ref{lem j0} and Lemma \ref{lem j>0} that $J_3(a,b)=0$ and the formula is now a consequence of \eqref{eq break j} and Lemmas \ref{lem J1} and \ref{lem J2}. 

Assume now that either $\abs{b}\le\abs{a}\le q^{-2}$.
It follows from Lemmas \ref{lem J1} and \ref{lem J2} that $J_1(a,b)=J_2(a,b)=0$ and therefore from \eqref{eq break j} and \eqref{eq J3} that 
\[
J(a,b)=J_3(a,b)=\J_0(a,b)+\J_{>0}(a,b).
\]
Combining Lemmas \ref{lem i=1 3val}, \ref{lem j0}, \ref{lem ji} for $\ell=\val(a)-1$ and \ref{lem j1 case2}  
we conclude that
\begin{multline}
(a,b)_3[\J_0(a,b)+\J_{\val(a)-1}(a,b)]=\\ \begin{cases} 3\abs{a}^{-1}+\abs{ab}^{-1}\sum_{k=0}^2\Cu_{\val(a)-1}(b^{-1}+a^{-1}\rho^k,-3^{-3}a^{-1}b^{-1}) & \abs{a}=\abs{b} \\ \delta_{\val(a),2} \abs{ab}^{-1}\sum_{k=0}^2\Cu_1(b^{-1}+a^{-1}\rho^k,-3^{-3}a^{-1}b^{-1}) & \abs{b}<\abs{a}.
\end{cases}
\end{multline}
Applying  Lemmas \ref{lem j>0} and \ref{lem ji} for $\ell<\val(a)-1$ the formula follows. This completes the proof of the proposition. \qed

\section{The integral \texorpdfstring{$I(a,b)$}~}{}\label{The-integral-I(a,b)}
In this section we provide the following formula for $I(a,b)$.
\begin{proposition}\label{prop I formula}
For $\abs{b}\le \abs{a}$ we have
\[
I(a,b)=\begin{cases}
1 & \abs{a}=\abs{b}=1 \\
2q & \abs{a}=\abs{b}=q^{-1} \\
3\abs{a}^{-1}+\abs{ab}^{-1}\sum_{\ell=\lfloor\frac{\val(a)+1}2\rfloor}^{\val(a)-1} \sum_{k=0}^2 \Cu_\ell(b^{-1}-\rho^k a^{-1},2a^{-1}b^{-1}) & \abs{a}=\abs{b} \le q^{-2} \\
\abs{b}^{-1}\Cu(b^{-1},2a^{-1}b^{-1}) & \abs{b}< \abs{a}=1 \\
\abs{ab}^{-1}\sum_{k=0}^2 \Cu_{\frac{\val(a)}2}(b^{-1}-\rho^k a^{-1},2a^{-1}b^{-1}) & \abs{b}<\abs{a}<1 \text{ and }2|\val(a) \\
0 &  \abs{b}<\abs{a}<1 \text{ and }2\nmid \val(a) \\
0 & \abs{a}>1.
\end{cases}
\]
Furthermore,
\begin{equation}\label{eq I bar fe}
I(b,a)=\overline{I(-a,-b)}, \ \ \ a,\, b\in F^*.
\end{equation}
\end{proposition}

\subsection{A first simplification and functional equation}

In this section we show that $I(a,b)$ satisfies a functional equation that allows us to reduce the computation to the case $\abs{b}\le \abs{a}$.
\subsubsection{} We begin by observing some basic properties of the integral $I(a,b)$.
\begin{lemma}\label{lem I props}
We have the following properties of $I(a,b)$.
\begin{enumerate}
\item\label{part I real} The integral $I(a,b)$ takes real values, that is, 
\[
\overline{I(a,b)}=I(a,b).
\]
\item\label{part zero not O} $I(a,b)=0$ unless $a,\,b\in \O$.
\item For $a,\,b\in \O$ we have
\begin{equation}\label{eq I as  sum Ij}
I(a,b)=\sum_{j=0}^{\min(\val(a),\val(b))}I(j;a,b)
\end{equation}
where
\[
I(j;a,b)=q^{-2j}\int \psi[x+y+2a\rho^2xy^2-(2\rho bx+2\rho^2ay)z]  \ ds\ dx\ dy\ dz
\]
and the integral is over $x,y,z,s\in F$ such that
$$
x+as,\,y-bs,\,(xb+y\rho^2a)s-(xy+z\rho)\in \p^{-j},\ xb+y\rho^2a\in \p^j.
$$ 
\end{enumerate}
\end{lemma}
\begin{proof}
Changing order of integration, it follows from the definition \eqref{eq I def} that 
\[
I(a,b)=\int_{\max(\abs{a},\abs{b})\le \abs{t}\le 1}\abs{t}^2 \left[\int \psi[x+y+2a(xy-z)y\rho^2-2bxz\rho]  \ ds\ dx\ dy\ dz\right]\ d^*t
\]
where the inner integral is over $x,y,z,s\in F$ such that
\[
t(x+as),t(bs-y),t^{-1}(xb+y\rho^2a),t[(xb+y\rho^2a)s-(xy+z\rho)])\in \O.
\]
The second part immediately follows.
Furthermore, the change of variables 
\[
(x,y,s)\mapsto -(x,y,s)
\] 
preserves the domain of integration and transforms the integrand into its conjugate. The first part follows. 

Assume that $a,b\in \O$. The inner integral over $x,\,y,\,z,\,s$ depends only on $\abs{t}$ and the third part of the lemma also follows, observing that $I(j;a,b)$ is the integral over the part of the domain where $\abs{t}=q^{-j}$.
\end{proof}
\subsubsection{} 
In the next lemma, a variable change in $z$ simplifies the domain of integration and allows us to integrate over $z$ and simplify the expression for $I(j;a,b)$. We then observe that the resulting expression satisfies a symmetry between $a$ and $b$.   
\begin{lemma}
For $a,\,b\in\O$ and $0\le j\le \min(\val(a),\val(b))$ we have
\begin{equation}\label{eq Ij} 
I(j;a,b)=q^{-j} \int \psi[x+y+2\left(bx^2y-axy^2+(abxy-b^2x^2-a^2y^2)s\right)]  \ ds\ dx\ dy
\end{equation}
where the integral is over $x,y,s\in F$ such that
\[
x+as,\,y-bs\in \p^{-j},\ bx,\,ay \in \p^j.
\]
Consequently
\begin{equation}\label{eq I fe} 
I(a,b)=I(-b,-a), \ \ \ a,\, b\in F^*.
\end{equation}
\end{lemma}
\begin{proof}
After the variable change 
\[
z\mapsto z+b\rho^2 xs+a\rho ys-\rho^2xy
\]
we have
\[
I(j;a,b)=q^{-2j}\int \psi[x+y+2a\rho^2xy^2-(2\rho bx+2\rho^2ay)(z+b\rho^2 xs+a\rho ys-\rho^2xy)]  \ ds\ dx\ dy\ dz
\]
where the integral is over $x,y,z,s\in F$ such that
\[
x+as,\,y-bs,\,z\in \p^{-j},\ xb+y\rho^2a\in \p^j.
\]
After integrating over $z$ this becomes
\[
q^{-j} \int \psi[x+y+2a\rho^2xy^2-(2\rho bx+2\rho^2ay)(b\rho^2 xs+a\rho ys-\rho^2xy)]  \ ds\ dx\ dy=
\]
\[
q^{-j} \int \psi[x+y+2\left(bx^2y-axy^2+(abxy-b^2x^2-a^2y^2)s\right)]  \ ds\ dx\ dy
\]
where the integral is now over $x,y,s\in F$ such that
\[
x+as,\,y-bs\in \p^{-j},\ xb+y\rho^2a,\,xb+y\rho a \in \p^j
\]
or equivalently
\[
x+as,\,y-bs\in \p^{-j},\ xb,\,ya \in \p^j.
\]
This completes the proof of \eqref{eq Ij}. 

Applying \eqref{eq Ij}, an interchange between $x$ and $y$ shows that $I(j;a,b)=I(j;-b,-a)$.
The functional equation \eqref{eq I fe} immediately follows from \eqref{eq I as  sum Ij} and Lemma \ref{lem I props} \eqref{part zero not O}.
\end{proof}

This reduces the computation of $I(a,b)$ to the case where $\abs{b}\le \abs{a}\le 1$.

\subsection{The case $\abs{b}\le \abs{a}\le 1$}
Our goal in this section is to compute $I(j;a,b)$ for $\abs{b}\le \abs{a}\le 1$ and $0\le j\le \val(a)$.

\subsubsection{} First, we further simplify the expression for $I(j;a,b)$ for the case $\abs{b}\le \abs{a}\le 1$. In the following lemma we first apply to \eqref{eq Ij} a variable change 
in $s$ that under the assumption $\abs{b}\le \abs{a}$ simplifies the domain of integration and allows us to integrate over $s$. Motivated by the polynomial decomposition
\[
X^2+Y^2-XY=(X+\rho Y)(X+\rho^2 Y)
\]
we follow up with a variable change in $y$ that further simplifies the domain of integration.
\begin{lemma}\label{lem s int}
For $\abs{b}\le \abs{a}\le 1$ and $0\le j\le \min(\val(a),\val(b))$ we have
\begin{equation}\label{eq I}
I(j;a,b)
=\abs{a}^{-1} \int \psi[(1-\rho a^{-1}b)x-(1+2\rho)y+2a^{-1}b^2x^3]  \ dx\ dy
\end{equation}
where the integral is over $x,y\in F$ such that
\begin{equation}\label{eq Ij domain}
bx-\rho ay\in a\p^{-j},\ bx,\,ay, y(bx+ay) \in \p^j.
\end{equation}
\end{lemma}
\begin{proof}
Apply to \eqref{eq Ij} the variable change 
\[
s\mapsto s-a^{-1}x
\] 
to obtain
\[
I(j;a,b)=q^{-j} \int \psi[x+y+2\left(bx^2y-axy^2+(abxy-b^2x^2-a^2y^2)(s-a^{-1}x)\right)]  \ ds\ dx\ dy
\]
\[
=q^{-j} \int \psi[x+y+2\left(a^{-1}b^2x^3+(abxy-b^2x^2-a^2y^2)s\right)]  \ ds\ dx\ dy
\]
where the integral is now over $x,y,s\in F$ such that
\[
as,\,bs-(a^{-1}bx+y)\in \p^{-j},\ xb,\,ya \in \p^j
\]
or, under our assumption that $\abs{b}\le \abs{a}$, equivalently such that 
\[
as,\,a^{-1}bx+y\in \p^{-j},\ xb,\,ya \in \p^j.
\]
Now integrating over $s$ we have
\[
I(j;a,b)
=\abs{a}^{-1} \int \psi[x+y+2a^{-1}b^2x^3]  \ dx\ dy
\]
where the integral is now over $x,y\in F$ such that
\[
a^{-1}bx+y\in \p^{-j},\ xb,\,ya \in \p^j,\, abxy-b^2x^2-a^2y^2\in a\p^j.
\]
After the variable change  
\[
y\mapsto -(\rho a^{-1}b x+(1+2\rho)y)
\] 
we have
\[
I(j;a,b)
=\abs{a}^{-1} \int \psi[(1-\rho a^{-1}b)x-(1+2\rho)y+2a^{-1}b^2x^3]  \ dx\ dy
\]
where the integral is now over $x,y\in F$ such that
\[
(1-\rho)bx-(1+2\rho)ay\in a\p^{-j},\ bx,\,ay, y(bx+ay) \in \p^j.
\]
Note that $1+2\rho=(1-\rho)\rho$ and $1-\rho\in\O^*$ so that the first condition is equivalent to $bx-\rho ay\in a\p^{-j}$ and the lemma follows.
\end{proof}

\subsubsection{}\label{ss split dom} By splitting the domain of integration three ways we express $I(j;a,b)$ as a sum of three integrals that we subsequently evaluate separately.

Let $\abs{b}\le \abs{a}\le 1$, fix an integer $j$ such that $0\le j\le \val(a)$ and set
\[
m=\min(j,\val(a)-j,\lfloor \frac{\val(a)-j}2 \rfloor).
\] 
By definition 
\[
\p^{-m}= \p^{-j}\cap \p^{j-\val(a)}\cap \p^{-\lfloor \frac{\val(a)-j}2 \rfloor}. 
\]
Note further that $y\in \p^{-\lfloor \frac{\val(a)-j}2 \rfloor}$ if and only if $ay^2\in\p^j$. Consequently, we have 
\begin{equation}\label{eq in m dom}
y\in \p^{-m}\ \ \ \text{if and only if}\ \ \ ay\in a\p^{-j}\cap \p^j\ \ \ \text{and} \ \ \ ay^2\in \p^j. 
\end{equation}
It follows that the condition 
\[
y\not\in \p^{-m}\ \ \  \text{and}\ \ \ ay\in \p^j
\] 
is equivalent to the condition
\[
(y\not\in \p^{-j}\ \ \ \text{or}\ \ \ ay^2\not\in \p^j)\ \ \ \text{and}\ \ \ ay\in \p^j.
\] 
Consequently, the set of $x,\,y\in F$ satisfying \eqref{eq Ij domain} partitions into three parts with the disjoint extra conditions:
\begin{enumerate}
\item\label{cond in pm} $y\in \p^{-m}$
\item\label{cond not in pj} $y\not\in \p^{-j}$
\item\label{cond the rest} $y\in \p^{-j}$ and $ay^2\not\in \p^j$.
\end{enumerate}

Write 
\[
I(j;a,b)=I_1(a,b)+I_2(a,b)+I_3(a,b)
\]
where $I_\ell(a,b)$ is defined by the integral \eqref{eq I} over $x,\,y\in F$ satisfying \eqref{eq Ij domain} as well as the condition ($\ell$) above, $\ell=1,2,3$. 

For future reference, we further point out that it is immediate from the definitions that
\begin{equation}\label{eq m prop}
m\ge 0\ \ \ \text{and}\ \ \ m=0\ \ \ \text{if and only if}\ \ \ j\in \{0,\val(a)-1,\val(a)\}.
\end{equation}

\subsubsection{} In the notation of \S\ref{ss split dom} we compute $I_1(a,b)$.

\begin{lemma}\label{lem I1}
We have
\[
I_1(a,b)=\begin{cases}
\abs{b}^{-1}\Cu(ab^{-1},2a^2b^{-1}) & j\in\{0,\val(a)\} \\ 
q\abs{b}^{-1}\Cu(\varpi^{-1}(\rho^2ab^{-1}-1),2\varpi^{-3}a^2b^{-1}) & 1\le j=\val(a)-1 \\
\abs{ab}^{-1}\Cu_j(b^{-1}-\rho a^{-1},2a^{-1}b^{-1}) & \lfloor\frac{\val(a)+1}2\rfloor\le j\le \val(a)-2 \\ 
0 & 1\le j\le \lfloor\frac{\val(a)-1}2\rfloor.
\end{cases}
\]
\end{lemma} 
\begin{proof}
It follows from \eqref{eq in m dom} that the conditions 
\[
x\ne 0,\,y\in \p^{-m},\,bx-\rho ay\in a\p^{-j},\ bx,\,ay, y(bx+ay) \in \p^j
\]
and
\[
x\ne 0,\,x\in b^{-1}(\p^j\cap a\p^{-j}),\,y\in\p^{-m}\cap (bx)^{-1}\p^j
\]
are equivalent. Excluding the measure zero set where $x=0$ we conclude that 
\[
I_1(a,b)=\abs{a}^{-1} \int_{b^{-1}(\p^j\cap a\p^{-j})} \psi[(1-\rho a^{-1}b)x+2a^{-1}b^2x^3]  \left\{\int_{\p^{-m}\cap(bx)^{-1}\p^j}\psi[-(1+2\rho)y]\ dy\right\}\ dx.
\]
We have
\[
\int_{\p^{-m}\cap(bx)^{-1}\p^j}\psi[-(1+2\rho)y]\ dy=\begin{cases} 1 & m=0\text{ or } \abs{bx}=q^{-j} \\ 0 & m>0\text{ and } \abs{bx}<q^{-j}.\end{cases}
\]
Note that the set of $x\in b^{-1}(\p^j\cap a\p^{-j})$ such that $\abs{bx}=q^{-j}$ is empty unless $\frac{\val(a)}2\le j$ and in this case it is precisey the set of $x$ such that $\abs{bx}=q^{-j}$.
 Taking \eqref{eq m prop} into consideration we have that
\[
I_1(a,b)=\begin{cases} 
\abs{a}^{-1} \int_{b^{-1}(\p^j\cap a\p^{-j})} \psi[(1-\rho a^{-1}b)x+2a^{-1}b^2x^3]   \ dx & j\in \{0,\val(a)-1,\val(a)\} \\ 
\abs{a}^{-1} \int_{\abs{bx}=q^{-j}} \psi[(1-\rho a^{-1}b)x+2a^{-1}b^2x^3] \ dx &\max(1,\,\frac{\val(a)}2)\le j\le \val(a)-2 \\ 0 & 1\le j <\frac{\val(a)}2 . \end{cases}
\]
We further observe that
\[
\p^j\cap a\p^{-j}=\begin{cases} a\O & j=0\text{ or }j=\val(a) \\ a\p^{-1} & j=\val(a)-1\ge 1, \end{cases}
\]
that after the variable change $x\mapsto b^{-1}a x$ we have that
\[
\abs{a}^{-1} \int_{b^{-1}a\O} \psi[(1-\rho a^{-1}b)x+2a^{-1}b^2x^3]   \ dx=\abs{b}^{-1} \Cu(ab^{-1},2a^2b^{-1}),
\]
that after the variable change $x\mapsto b^{-1}a\varpi^{-1} x$ we have that
\[
\abs{a}^{-1}\int_{b^{-1}a\p^{-1}} \psi[(1-\rho a^{-1}b)x+2a^{-1}b^2x^3] \ dx=q\abs{b}^{-1}\,\Cu(\varpi^{-1}( ab^{-1}-\rho),2\varpi^{-3}a^2b^{-1})
\]
and that after the variable change $x\mapsto b^{-1} x$ we have that 
\[
\abs{a}^{-1} \int_{\abs{bx}=q^{-j}} \psi[(1-\rho a^{-1}b)x+2a^{-1}b^2x^3] \ dx=\abs{ab}^{-1} \Cu_j(b^{-1}-\rho a^{-1},2a^{-1}b^{-1}).
\]
The lemma readily follows.
\end{proof}

\subsubsection{} In the notation of \S\ref{ss split dom} we compute $I_2(a,b)$.

\begin{lemma}\label{lem dom D}
Let
\[
D=\{(x,y)\in F^2:y\not\in \p^{-j},\,bx-\rho ay\in a\p^{-j},\ bx,\,ay, y(bx+ay) \in \p^j\}.
\]
and
\[
D'=\{(x,y)\in F^2:\abs{ab^{-1}}q^j<\abs{x}\le \abs{a}^{\frac12}\abs{b}^{-1}q^{-\frac{j}2},\,y\in \rho^2 a^{-1}bx+\p^{-j}\}.
\]
Then $D=D'$.
\end{lemma}
\begin{proof}
It is easy to see that 
\[
y\not\in \p^{-j},\,bx-\rho ay\in a\p^{-j}
\]
if and only if
\[
\abs{ab^{-1}}q^j<\abs{x},\,y\in \rho^2 a^{-1}bx+\p^{-j}
\]
and that these conditions imply 
\[
\abs{bx}=\abs{ay}>\abs{a}q^j.
\]
Assume now that $(x,y)\in D$. 
From the first two conditions defining $D$ we have that $\abs{bx-\rho ay}<\abs{bx}=\abs{ay}$ and therefore we must have $\abs{bx+ay}=\abs{bx}$. Consequently, the condition $y(bx+ay) \in \p^j$ implies that $\abs{a^{-1}b^2x^2}=\abs{y(bx+ay)}\le q^{-j}$ and it follows that $(x,y)\in D'$. 

Conversely, if $(x,y)\in D'$
then the first two conditions defining $D$ are satisfied and $\abs{ay}=\abs{bx}$. The weak inequality in the definition of $D'$ now implies $bx,\,ay,\,bxy,\,ay^2\in \p^j$ and therefore also $y(bx+ay) \in \p^j$ so that $(x,y)\in D$. Then lemma follows. 
\end{proof}
\begin{lemma}\label{lem I2}
We have
\[
I_2(a,b)=\begin{cases}
\abs{ab}^{-1}\sum_{\ell=\lfloor\frac{\val(a)+1}2\rfloor}^{\val(a)-1}\Cu_\ell(b^{-1}-a^{-1},2a^{-1}b^{-1}) & j=0 \\ 
0 & j\ge 1.
\end{cases}
\]
\end{lemma}
\begin{proof}
As a consequence of Lemma \ref{lem dom D} we have that $\abs{a}I_2(a,b)$ equals
\[
\sum_{\ell=\lfloor\frac{\val(a)+j+1}2 \rfloor-\val(b)}^{\val(a)-\val(b)-j-1} \int_{\val(x)=\ell} \psi[(1-\rho a^{-1}b)x+2a^{-1}b^2x^3]\left\{\int_{\rho^2 a^{-1}bx+\p^{-j}} \psi(-(1+2\rho)y)\ dy\right\}  \ dx.
\]
Since $-(1+2\rho)\rho^2=\rho-1$ we have that
\[
\int_{\rho^2 a^{-1}bx+\p^{-j}} \psi(-(1+2\rho)y)\ dy=\begin{cases}
\psi((\rho-1)a^{-1}bx) & j=0 \\
0 & j\ge 1.
\end{cases}
\]
We conclude that if $j\ge 1$ then $I_2(a,b)=0$ and if $j=0$ then 
\[
I_2(a,b)=\abs{a}^{-1}\sum_{\ell=\lfloor\frac{\val(a)+1}2 \rfloor-\val(b)}^{\val(a)-\val(b)-1} \int_{\val(x)=\ell} \psi[(1- a^{-1}b)x+2a^{-1}b^2x^3]\ dx
\]
The variable change $x\mapsto b^{-1}x$ shows that
\[
\int_{\val(x)=\ell} \psi[(1- a^{-1}b)x+2a^{-1}b^2x^3]\ dx=\abs{b}^{-1}\Cu_{\ell+\val(b)}(b^{-1}-a^{-1},2a^{-1}b^{-1}).
\]
After the change $\ell\mapsto \ell-\val(b)$ in the index of summation the lemma follows. 
\end{proof}

\subsubsection{} In the notation of \S\ref{ss split dom} we compute $I_3(a,b)$.
\begin{lemma}\label{lem dom D3}
Let
\[
D=\{(x,y)\in F^2:y\in \p^{-j},\,ay^2\not\in \p^j,\,bx-\rho ay\in a\p^{-j},\ bx,\,ay, y(bx+ay) \in \p^j\}
\]
and
\[
D'=\{(x,y)\in F^2:\abs{b}^{-1}\abs{a}^{\frac12}q^{-\frac{j}2}<\abs{x}\le \abs{b}^{-1}\min(\abs{a}q^j,q^{-j}),\,y\in - a^{-1}bx+(bx)^{-1}\p^{j}\}.
\]
Then $D=D'$ and $D$ is empty unless $j<\val(a)<3j$.
\end{lemma}
\begin{proof}
Note that if $ay^2\not\in \p^j$ and $y(bx+ay) \in \p^j$ then $\abs{bx}=\abs{ay}$. It follows that $bx\not\in y^{-1}\p^j=a(bx)^{-1}\p^j$, and therefore $x$ satisfies the strong inequality in the definition of $D'$.
Furthermore, $y\in \p^{-j}$ and $bx-\rho ay\in a\p^{-j}$ imply that $bx\in a\p^{-j}$ and together with $bx\in \p^j$ this gives that $x$ also satisfies the weak inequality.
Consequently, for $(x,y)\in D$ we have $y^{-1}\p^j=ab^{-1}x^{-1}\p^j$ and therefore it is easy to see that $D\subseteq D'$.

For $(x,y)\in D'$ the strong inequality in the definition of $D'$ implies that $(bx)^2\not\in a\p^j$, while the condition on $y$ implies that $bx(ay+bx)\in a\p^j$. Therefore $\abs{ay}=\abs{bx}$ and from the weak inequality it follows that $ay,\,bx\in a\p^{-j}\cap \p^j$ and from the strong inequality that $\abs{ay^2}=\abs{a^{-1}(bx)^2}>q^{-j}$ and therefore $D'\subseteq D$.

The condition on $j$ follows from the fact that if $D$ is not empty then $\abs{b}^{-1}\abs{a}^{\frac12}q^{-\frac{j}2}< \abs{b}^{-1}\min(\abs{a}q^j,q^{-j})$.
\end{proof}
\begin{lemma}\label{lem I3}
We have
\[
I_3(a,b)=\begin{cases}
\abs{ab}^{-1}\Cu_j(b^{-1}-\rho^2a^{-1},2a^{-1}b^{-1}) &  \lfloor \frac{\val(a)+1}2 \rfloor\le j\le \val(a)-1 \\
0 & j=\val(a)\text{ or } j\le \lfloor \frac{\val(a)-1}2 \rfloor.
\end{cases}
\]
\end{lemma}
\begin{proof}
It follows from Lemma \ref{lem dom D3} that $I_3(a,b)=0$ unless $\lfloor \frac{\val(a)}3 \rfloor+1\le j\le \val(a)-1$. For the rest of the proof assume that these inequalities hold. Note that
\[
\int_{- a^{-1}bx+(bx)^{-1}\p^{j}}\psi(-(1+2\rho)y)\ dy=\begin{cases}
\psi((1+2\rho)a^{-1}bx) & \abs{bx}=q^{-j} \\ 
0 & \abs{bx}<q^{-j}.
\end{cases}
\]
It therefore further follows from Lemma \ref{lem dom D3} that
\[
I_3(a,b)=\abs{a}^{-1}\int \psi[(1-\rho^2 a^{-1}b)x+2a^{-1}b^2 x^3]\ dx
\]
where the integral is over all $x\in F$ such that 
\[
\abs{bx}=q^{-j}\ \ \ \text{and} \ \ \ \abs{b}^{-1}\abs{a}^{\frac12}q^{-\frac{j}2}<\abs{x}\le \abs{b}^{-1}\min(\abs{a}q^j,q^{-j}).
\]
Note that this domain is empty and therefore $I_3(a,b)=0$ unless $j\ge \frac{\val(a)}2$ and when this is the case the domain is defined by $\abs{bx}=q^{-j}$. 
Applying the variable change $x\mapsto b^{-1}x$ we obtain in this case that
\[
I_3(a,b)=\abs{ab}^{-1}\Cu_j(b^{-1}-\rho^2a^{-1},2a^{-1}b^{-1}).
\]
The lemma follows.
\end{proof}

\subsection{Completion of proof of Proposition \ref{prop I formula}}
The functional equation  \eqref{eq I bar fe} is a consequence of part \eqref{part I real} of Lemma \ref{lem I props}  and equation \eqref{eq I fe}.
If $a\not\in\O$ then $I(a,b)=0$ by part \eqref{part zero not O} of Lemma \ref{lem I props}. 
 
Assume for the rest of this section that $\abs{b}\le \abs{a}\le 1$. 
Combining Lemmas \ref{lem I1}, \ref{lem I2}, \ref{lem I3} and \eqref{eq I as  sum Ij} we have that
\begin{multline*}
I(a,b)=(1+\delta_{\val(a)\ge 1})\abs{b}^{-1}\Cu(ab^{-1},2a^2b^{-1})+ \\
\delta_{\val(a)\ge 2}\left\{
q\abs{b}^{-1}\Cu(\varpi^{-1}(\rho^2ab^{-1}-1),2\varpi^{-3}a^2b^{-1})+
\abs{ab}^{-1}\Cu_{\val(a)-1}(b^{-1}-a^{-1},2a^{-1}b^{-1})+ \right. \\
\left. \abs{ab}^{-1}\Cu_{\val(a)-1}(b^{-1}-\rho^2a^{-1},2a^{-1}b^{-1})\right\}+ 
\abs{ab}^{-1}\sum_{\ell=\lfloor\frac{\val(a)+1}2\rfloor}^{\val(a)-2} \sum_{k=0}^2 \Cu_\ell(b^{-1}-\rho^k a^{-1},2a^{-1}b^{-1})
\end{multline*}
where for $i=1,2$
\[
\delta_{\val(a)\ge i}=\begin{cases}
1 & \val(a)\ge i \\
0 & \text{otherwise.}
\end{cases}
\]

The evaluation of $I(a,b)$ may now 
 be obtained from this expression by applying Lemma~\ref{lem cobic exp}, equations \eqref{eq cu l} and \eqref{eq cu l is zero}, and the identity
\begin{equation}\label{eq C unit}
\Cu(x,y)=\Cu(ux,u^3y),\ \ \ x,\,y\in F,\ u\in\O^*.
\end{equation}
In detail, Proposition \ref{prop I formula} is a straightforward consequence of the following seven identities:
\begin{enumerate}
\item $\Cu(ab^{-1},2a^2b^{-1})=\begin{cases}
1 &\abs{a}=\abs{b}\le 1 \\ 
\Cu(b^{-1},2a^{-1}b^{-1}) & 1=\abs{a}>\abs{b} \\
0 & \abs{b}<\abs{a}\le q^{-1}.
\end{cases}$

Indeed, for the case $1=\abs{a}>\abs{b}$ apply \eqref{eq C unit} with $u=a^{-1}$ and in the other cases apply Lemma \ref{lem cobic exp}.
\item 
If $\abs{b}\le \abs{a}=q^{-2}$ then 
\begin{equation}\label{id a=2}
\Cu(\varpi^{-1}(\rho^2ab^{-1}-1),2\varpi^{-3}a^2b^{-1})=\Cu(\varpi(b^{-1}-\rho a^{-1}),2\varpi^3a^{-1}b^{-1}).
\end{equation}

Indeed, in this case apply \eqref{eq C unit} with $u=\rho\varpi^2 a^{-1}\in\O^*$.

\item If $\abs{b}< \abs{a}=q^{-2}$ then $\Cu(\varpi^{-1}(\rho^2ab^{-1}-1),2\varpi^{-3}a^2b^{-1})=q\,\Cu_1(b^{-1}-\rho a^{-1},2a^{-1}b^{-1})$.

Indeed, apply the identity \eqref{id a=2} and note that 
\[
\abs{\varpi(b^{-1}-\rho a^{-1})}=\abs{2\varpi^3a^{-1}b^{-1}}=q^{-1}\abs{b}^{-1}\ge q^2
\] 
so that by Lemma \ref{lem cobic exp} we have
\[
\Cu(\varpi^{-1}(\rho^2ab^{-1}-1),2\varpi^{-3}a^2b^{-1})=\Cu_0(\varpi(b^{-1}-\rho a^{-1}),2\varpi^3a^{-1}b^{-1}).
\]
The identity now follows from \eqref{eq cu l} applied with $t=\varpi$.
\item If $\abs{b}=\abs{a}=q^{-2}$ then $\Cu(\varpi^{-1}(\rho^2ab^{-1}-1),2\varpi^{-3}a^2b^{-1})=q^{-1}+q\,\Cu_1(b^{-1}-\rho a^{-1},2a^{-1}b^{-1})$.

Indeed, apply the identity \eqref{id a=2} and note that
\[
\abs{\varpi(b^{-1}-\rho a^{-1})}\le \abs{2\varpi^3a^{-1}b^{-1}}=q
\] 
so that by Lemma \ref{lem cobic exp} we have
\[
\Cu(\varpi^{-1}(\rho^2ab^{-1}-1),2\varpi^{-3}a^2b^{-1})=q^{-1}+\Cu_0(\varpi(b^{-1}-\rho a^{-1}),2\varpi^3a^{-1}b^{-1}).
\]
The identity now follows from \eqref{eq cu l} applied with $t=\varpi$.
\item If $\abs{b}<\abs{a}\le q^{-3}$ then $\Cu(\varpi^{-1}(\rho^2ab^{-1}-1),2\varpi^{-3}a^2b^{-1})=0$.

Indeed, since $\max(\abs{2\varpi^{-3}a^2b^{-1}},q)<\abs{\varpi^{-1}(\rho^2ab^{-1}-1)}$ this is immediate from Lemma \ref{lem cobic exp}.

\item If $\abs{b}<\abs{a}\le q^{-3}$ then $\Cu_\ell(b^{-1}-\rho^k a^{-1},2a^{-1}b^{-1})=0$ whenever $\ell\le \val(a)-1$, $\val(a)\not=2\ell$ and $k\in\{0,1,2\}$. 

Indeed, since 
\[
\abs{b^{-1}-\rho^k a^{-1}}=\abs{b}^{-1}>\abs{a}^{-1}\ge q^{\ell+1}
\]
this is immediate from \eqref{eq cu l is zero}.
\item
If $\abs{b}=\abs{a}\le q^{-3}$ then 
\[
q\,\Cu(\varpi^{-1}(\rho^2ab^{-1}-1),2\varpi^{-3}a^2b^{-1})=1+\abs{a}^{-1}\Cu_{\val(a)-1}(b^{-1}-\rho a^{-1}, 2a^{-1}b^{-1}). 
\]

Indeed, in this case 
\[
\abs{\varpi^{-1}(\rho^2ab^{-1}-1)}\le q \ \ \ \text{while}\ \ \ 2\varpi^{-3}a^2b^{-1}\in \O
\] 
so that by Lemma \ref{lem cobic exp} we have
\[
q\,\Cu(\varpi^{-1}(\rho^2ab^{-1}-1),2\varpi^{-3}a^2b^{-1})=1+q\,\Cu_0(\varpi^{-1}(\rho^2ab^{-1}-1),2\varpi^{-3}a^2b^{-1}).
\]
The identity now follows from \eqref{eq cu l} applied with $t=\rho^2\varpi^{-1}a$.
\end{enumerate}

This completes the proof of Proposition \ref{prop I formula}. \qed

\section{The comparison of \texorpdfstring{$I(a,b)$}{}~and \texorpdfstring{$J(a,b)$}{}: Proof of Theorem~\ref{thm main}}\label{The-big-cell-comparison-is-done}
Let $a,\,b\in F^*$, $c=-54a$ and $d=54b$. It follows from the functional equations \eqref{eq J bar fe} and \eqref{eq I bar fe} that without loss of generality we may assume that $\abs{b}\le \abs{a}$. Applying \eqref{eq C unit} with $u=54$ we have 
\[
\Cu(d^{-1},-3^{-3}c^{-1}d^{-1})=\Cu(b^{-1},2a^{-1}b^{-1})
\]
and similarly applying \eqref{eq cu l} with  $t=54$ we have
\[
\Cu_\ell(d^{-1}+\rho^k c^{-1},-3^{-3}c^{-1}d^{-1})=\Cu_\ell(b^{-1}-\rho^k a^{-1},2a^{-1}b^{-1}).
\]
The proof of the theorem now follows from Propositions \ref{prop J formula} and \ref{prop I formula}. \qed

\section{The comparison of orbital integrals for the other relevant orbits}\label{Comparison for smaller cells}
In this Section we compute and match the local contributions to the non-big-cell relevant orbits for the unit elements in the appropriate Hecke algebras.
\subsection{The I integrals}
We begin with the $I$ integrals.  The relevant orbits are given in Theorem~\ref{I-relevant-orbits}.  

There are three relevant orbits indexed by a single power of $\rho$
(see Lemmas~\ref{lem sl2}, \ref{relevant-2}, and \ref{relevant-47}).  
It is straightforward to check that for each such orbit the local integral evaluates to $1$.  
For example, the integral arising from Lemma~\ref{lem sl2} is
\begin{multline*}
\int_F\int_{F^*}\int_{K_2} \mathbf{1}_{\O^8}[(0,0,0,t,0,\rho^2t,0,xt)\lambda(k)]\abs{t}^2\,dk\, d^*t\ \psi(x) \,dx\\
=\int_{K_2}\,dk\int_{|t|, |xt|\leq1} |t|^2\psi(x)\,d^*t\,dx=\int_{0<|t|\leq1}\Big(\int_{|x|\leq |t|^{-1}} \psi(x)\,dx\Big)\,|t|^2\,d^*t.
\end{multline*}
The inner integral in $x$ is zero unless $|t|=1$ and so the evaluation follows.

We turn to the local contribution to
Lemma~\ref{relevant-3} when the local test function is the
characteristic function of $\PGL_3(\O)$.  This is given, for $a\in F^*$ ($a$ replaces $\alpha$ in the notation of the Lemma), by
\[
I_1(a)=\int \psi(-\rho a^{-1}xy-\rho^2a^{-1}z+a^{-2}xz+x+y)\ dx\ dy \ dz \abs{t}^2 \ d^*t
\]
where the integral is over the domain
\[
t,t(a+x),a^{-2}t^{-1},ty,a^{-2}t^{-1}x,t(\rho ay-\rho^2 xy-z)\in \O.
\]
Note that the domain of integration is empty unless $\abs{a}\ge 1$. When $\abs{a}=1$ the domain is 
\begin{equation}\label{degenerate-one}
t\in \O^*,x,y,z\in \O,
\end{equation}
and we deduce that $I_1(a)=1$. 

Assume now that $\abs{a}=q^n>1$.
We can write 
\[
I_1(a)=\sum_{j=0}^{2n} I_{1,j}(a)
\]
where 
\[
I_{1,j}(a)=q^{-2j} \int \psi(-\rho a^{-1}xy-\rho^2a^{-1}z+a^{-2}xz+x+y)\ dx\ dy \ dz
\]
where the domain of integration is 
\[
\abs{y},\abs{a+x},\abs{\rho ay-\rho^2 xy-z}\le q^j,\ \abs{x}\le q^{2n-j}.
\]
After the variable change $z\mapsto z+\rho a y-\rho^2 xy$ we have
\[
I_{1,j}(a)=q^{-2j} \int_{\abs{x}\le q^{2n-j},\,\abs{a+x}\le q^j}\psi(x)\int_{\p^{-j}}\psi[\rho a^{-1}x(1-\rho a^{-1}x)y]\ dy \int_{\p^{-j}}\psi[-\rho^2a^{-1}(1-\rho a^{-1}x)z]\ dz \ dx
\]
\[
=\int\psi(x)\ dx
\]
where the last integral is over the domain
\[
\abs{x}\le q^{2n-j},\,\abs{a+x}\le q^j,\,\abs{a^{-1}x(1-\rho a^{-1}x)},\,\abs{a^{-1}(1-\rho a^{-1}x)}\le q^{-j}
\]
or equivalently
\[
\abs{x},\,\abs{x(a-\rho x)},\,\abs{a-\rho x} \le q^{2n-j},\,\abs{a+x}\le q^j.
\]
If $j<n$ then the condition $\abs{a+x}\le q^j<q^n=\abs{a}$ implies that $\abs{x}=\abs{a-\rho x}=q^n$ and therefore $\abs{x(a-\rho x)}=q^{2n}$. The domain of integration is therefore empty unless $j=0$ in which case the domain is $x\in -a+\O$. We conclude that 
\[
I_{1,j}(a)=\begin{cases} \psi(-a)& j=0 \\  0 & 0<j<n.\end{cases}
\]
If $j>n$ then the condition $\abs{a-\rho x} \le q^{2n-j}<q^n=\abs{a}$ implies that $\abs{x}=q^n>q^{2n-j}$ and therefore again the domain of integration is empty so that $I_{1,j}(a)=0$.

For $j=n$ note that the domain of integration becomes $\O\sqcup \rho^2 a+\O$ and therefore
\[
I_{1,n}(a)=1+\psi(\rho^2a).
\]
We conclude that the following holds.
\begin{lemma} If $a\in F^*$ then
\[
I_1(a)=\begin{cases} \psi(-a)+1+\psi(\rho^2a) & \abs{a}>1 \\ 1 & \abs{a}=1 \\ 0 & \abs{a}<1. \end{cases}
\] 
\end{lemma}
\qed

Next we turn to the local integral coming from the relevant orbits of Lemma~\ref{relevant-4} in the case that the local component of $\phi$ is the characteristic function of $\PGL_3(\O)$.  For $a\in F^*$, this is
$$
I_2(a)=\int \psi[a^{-2}(xy-z)y-\rho(xy\rho^2+z)a^{-1}+x+y] \ dx\ dy\ dz\abs{t}^2\ d^*t.
$$
where the integral is over the domain
\[
t,a^{-2}t^{-1},tx,t(a-y),a^{-2}t^{-1}y,t(xa-xy-z\rho)\in \O.
\]
The first two conditions imply that the domain is empty unless $|a|\geq1$.
When $\abs{a}=1$ the domain is once again \eqref{degenerate-one}
and it is easy to see that $I_2(a)=1$.  If $|a|=q^n>1$, then since $|t|\leq1$ and $|t|^{-1}\leq |a|^2=q^{2n}$ 
we may write 
\[
I_2(a)=\sum_{j=0}^{2n} I_{2,j}(a)
\]
where in $I_{2,j}(a)$ the domain of integration is limited to $|t|=q^{-j}$.  After the variable change $z\mapsto z+\rho^2 a x -\rho^2 x y$, we see that
\[
I_{2,j}(a)=q^{-2j} \int \psi(-a^{-2}y(\rho x y + z + \rho^2 a x)-\rho z a^{-1}+ y)\ dx\ dy \ dz
\]
where the domain of integration is 
\[
\abs{x},\abs{a-y},\abs{z}\le q^j,\ \abs{y}\le q^{2n-j}.
\]
The integration in $z$ gives zero unless $|1+\rho^2 a^{-1} y|\leq q^{-j} |a|$ and the integral in $x$ gives zero unless
$|a^{-1}y(a^{-1}y+\rho)|\leq q^{-j}$.  Thus we arrive at 
$$I_{2,j}(a)=\int \psi(y)\,dy$$
where the domain of integration is
$$|y|\leq q^{2n-j}, |a-y|\leq q^j, |1+\rho^2a^{-1}y|\leq q^{n-j}, |y(1+\rho^2a^{-1}y)|\leq q^{n-j}.$$

If $j<n$ then since $|a|=q^n$ the second inequality implies that $|y|=q^n$ and $|a^{-1}y-1|<1$.  But this
condition implies that $|1+\rho^2 a^{-1}y|=1$, so the last inequality is possible
only when $j=0$. In that case the domain reduces to $y\in a+\O$.  Thus
\[
I_{2,j}(a)=\begin{cases} \psi(a)& j=0 \\  0 & 0<j<n.\end{cases}
\]
If $j>n$ then in the domain of integration we have $|1+\rho^2 a^{-1}y|<1$,
 and this implies that $|y|=|a|=q^n$ which contradicts $|y|\leq q^{2n-j}$.  In this case the domain is thus empty.
Finally if $j=n$ then 
the inequality $|y|\leq q^n$ implies the second and third inequalities above, and in view of the inequality 
$\abs{y(1+\rho^2 a^{-1}y)}\leq1$ we see that the domain of integration is the union of the two $\O$-cosets $y\in \O$ and $y\in -\rho a+\O$.
It follows that $I_{2,n}(a)=1+\psi(-\rho a)$.  We conclude that the following holds.
\begin{lemma}
If $a\in F^*$ then
$$I_2(a)=\begin{cases} 1+\psi(a)+\psi(-\rho a) & \abs{a}>1 \\ 1 & \abs{a}=1 \\ 0 & \abs{a}<1. \end{cases}$$
\end{lemma}
\qed

\subsection{The J integrals}
The non-big cell relevant orbits are given by
\begin{enumerate}
\item $\zeta I_3, \zeta\in \mu_3$
\item $\begin{pmatrix} &   a^{-2}  \\ aI_2 &    \end{pmatrix}, a\in F^*$
\item $\begin{pmatrix}  & aI_2  \\ a^{-2}   & \end{pmatrix}, a\in F^*.$
\end{enumerate}

For the first family of three orbits, the local integral for the unit element is given by
$$\int_{N(F)} f_0(\zeta n)\psi(n)\ dn, \quad \zeta\in \mu_3.$$
Since
$$f_0(\zeta n)=\begin{cases} 1&n\in N\cap K\\0&\text{otherwise}\end{cases}$$
each integral is $1$.

We turn to the second family of orbits.  For $i=1,2$, write the $n_i$ that appears in the integration  as
$n_i=\left(\begin{smallmatrix}1&x_i&z_i\\&1&y_i\\&&1\end{smallmatrix}\right).$ Since we are modding out by the stabilizer we may take $y_1=0$. Multiplying, we see that the local integral is given by
$$J_1(a):=\int_{F^5} f_0\left(\begin{pmatrix}-ax_1&-ax_1x_2-az_1&a^{-2}-a(z_1y_2+x_1z_2)\\a&ax_2&az_2\\0&a&ay_2\end{pmatrix}\right)
\psi (-x_1+x_2+y_2)\,d(*).$$

Since $f_0$ vanishes away from $K\times \mu_3$ we see that $J_1(a)=0$ unless $|a|\leq1$.  In addition, in the support of the integrand,  the variables of integration must satisfy the inequalities
$$\abs{x_1},\abs{x_2},\abs{y_2},\abs{z_2}, \abs{x_1x_2+z_1}\leq |a|^{-1}, \abs{a^{-2}-a(z_1y_2+x_1z_2)}\leq1.$$

If $|a|=1$ the domain of integration is $\O^5$ and it follows directly from the factorization of the argument of $f_0$ that $\kappa=1$ in this domain.  Thus the integral is $1$.

Suppose from now on that $|a|<1$. Then in the support of the integrand the diagonal entries must be units.
Thus the domain of integration is given by
\begin{multline}\label{domain-degenerate1}
\abs{x_1}=\abs{x_2}=\abs{y_2}=\abs{a}^{-1},\\ z_1\in -x_1x_2+a^{-1}\O, z_2\in x_1^{-1}a^{-3}-x_1^{-1}z_1y_2+\O, \abs{a^{-3}+x_1x_2y_2}\leq \abs{a}^{-2}.
\end{multline}
Note that the last condition is required since $\abs{z_2}\leq \abs{a}^{-1}$.
We have the following evaluation.
\begin{lemma}\label{first-kappa} Suppose that $\abs{a}<1$ and that the conditions \eqref{domain-degenerate1} hold.  Then
$$\kappa\left(\begin{pmatrix}-ax_1&-ax_1x_2-az_1&a^{-2}-a(z_1y_2+x_1z_2)\\a&ax_2&az_2\\0&a&ay_2\end{pmatrix}\right)=(a,x_2y_2^2)_3.$$
\end{lemma}

\begin{proof}We describe the proof in brief. Since $ay_2\in\O^*$, left multiplication by an element of $N\cap K$ reduces the computation to the evaluation of $\kappa(g_1)$
where
$$g_1=\begin{pmatrix}-ax_1&-ax_1x_2-y_2^{-1}a^2+y_2^{-1}ax_1z_2&0\\a&ax_2-ay_2^{-1}z_2&0\\0&a&ay_2\end{pmatrix}.$$
We have $y_2^{-1}\in\O$. If
$$g_2=\begin{pmatrix}1&&\\0&1\\0&-y_2^{-1}&1\end{pmatrix},\quad\text{then}\quad
g_1g_2=\begin{pmatrix}-ax_1&-ax_1x_2-y_2^{-1}a^2+y_2^{-1}ax_1z_2&0\\a&ax_2-ay_2^{-1}z_2&0\\0&0&ay_2\end{pmatrix}.$$
Since $g_1g_2$ is block diagonal it is easy to evaluate 
$$
\kappa(g_1g_2)=(a,a(x_2-y_2^{-1}z_2)ay_2)_3=(a,x_2y_2-z_2)_3=(a,x_2y_2)_3.
$$
In the last equality above we have used that $\abs{z_2}<\abs{x_2y_2}$ in the domain \eqref{domain-degenerate1}. Furthermore, $\kappa(g_2)=1$. We conclude that the value we require is given by
$(a,x_2y_2)_3 \, \sigma(g_1,g_2)^{-1}.$
Applying the algorithm of Bump-Hoffstein to compute $\sigma$, one finds that $\sigma(g_1,g_2)=(y_2,a)_3$.  The result follows.
\end{proof}

We arrive at the expression
$$J_1(a)=\int (a,x_2^2 y_2)_3 \psi(-x_1+x_2+y_2)\, dz_2\,dz_1\,dx_1\,dx_2\,dy_2,$$
where the integration is over the domain \eqref{domain-degenerate1}.  We carry out the inner integrations in $z_2$ and then $z_1$ to obtain
$$\abs{a}^{-1}\int (a,x_2^2 y_2)_3 \psi(-x_1+x_2+y_2)\,dx_1\,dx_2\,dy_2,$$
where the integration is over 
$$\abs{x_1}=\abs{x_2}=\abs{y_2}=\abs{a}^{-1}, \abs{a^{-1}+a^2 x_1x_2y_2}\leq 1.$$

Now we change $x_1\to a^{-1}x_1$, $x_2\to a^{-1}x_2$, and $y_2\to a^{-1}y_2$ to obtain 
$$\abs{a}^{-4}\int (a,x_2^2 y_2)_3\, \psi(a^{-1}(-x_1+x_2+y_2))\,dx_1\,dx_2\,dy_2,$$
where the integration is over 
$$\abs{x_1}=\abs{x_2}=\abs{y_2}=1, \abs{1+x_1x_2y_2}\leq \abs{a}.$$
Changing $x_1\mapsto x_1/(x_2y_2)$ the last inequality becomes $\abs{1+x_1}\leq\abs{a}$.  This allows us to do the $x_1$ integral. We arrive at
$$J_1(a)=\abs{a}^{-3}\int_{\abs{x_2}=\abs{y_2}=1} (a,x_2^2y_2)_3\, \psi\left(a^{-1}\left(x_2+y_2+\frac{1}{x_2y_2}\right)\right)\,dx_2\,dy_2.$$

We now make use of Corollary~\ref{cor di} to replace the integral in $x_2$  with the integral of a cubic (so the parameters $c,d,t$ of the 
Corollary are given by $c=a^{-1}$, $d=(ay_2)^{-1}$, $t=a^2$; note that $(t,c^{-1}d)_3=(a^2,y_2)_3^{-1}$).  We see that
$$J_1(a)=\abs{a}^{-3}\int_{\abs{y_2}=1} \Cu(a^{-1},-3^{-3}y_2a^{-1})\psi(a^{-1}y_2)\,dy_2.$$
Note that in the integrand above, the symbol $(a^2,y_2)_3$ obtained from the Corollary cancels the symbol $(a,y_2)_3$ obtained from $\kappa$.

We evaluate the last integral by using the definition of $\Cu$.  Substituting this definition and changing $x\to 3 x$, we have
\begin{multline*}
J_1(a)=\abs{a}^{-3}\int_{\O^*}\int_{\O} \psi(a^{-1}x-3^{-3}y_2a^{-1}x^3+y_2a^{-1})\,dx\,dy_2\\
=\abs{a}^{-3}\int_{\O}\psi(3a^{-1}x)\,\int_{\O^*}\psi(y_2a^{-1}(1-x^3))\,dy_2\,dx.
\end{multline*}
Note that
\begin{itemize}
\item $\abs{a^{-1}(1-x^3)}\leq1$ if and only if $x\in \rho^k+a\O$, $k=0,1,2$
\item $\abs{a^{-1}(1-x^3)}=q$ if and only if $x\in \rho^k+\varpi^{-1}a\O^*$, $k=0,1,2$,
\end{itemize}
and the inner integral is zero if $\abs{a^{-1}(1-x^3)}>q$. 
If $x\in\rho^k+a\O$, $k=0,1,2$, then the inner integral is $1-q^{-1}$, and so this piece contributes
$$(1-q^{-1})\abs{a}^{-2}\sum_{k=0}^2 \psi(3\rho^k a^{-1}).$$
 If $x\in\rho^k+\varpi^{-1}a\O^*$ for some $k=0,1,2$ then the inner integral is $-q^{-1}$. If in addition $\abs{a}<q^{-1}$ then the sets $\rho^k+\varpi^{-1}a\O^*$ are disjoint for $k=0,1,2$, so this piece contributes
 $$-q^{-1}\abs{a}^{-3}\sum_{k=0}^2\psi(3\rho^ka^{-1})\int_{\varpi^{-1}a\O^*}\psi(3a^{-1}x)\,dx=q^{-1}\abs{a}^{-2}\sum_{k=0}^2\psi(3\rho^k a^{-1}) .$$
If $\abs{a}=q^{-1}$ then the sets $\rho^k+\varpi^{-1}a\O^*=\rho^k+\O^*$, $k=0,1,2$ are no longer disjoint, however, their union equals $\O\setminus (\sqcup_{k=0}^2 (\rho^k+\p))$. Consequently, in this case this piece contributes
\[
-q^2\int_{\O\setminus (\sqcup_{k=0}^2 (\rho^k+\p))}\psi(3a^{-1}x)\ dx=q^2\sum_{k=0}^2 \int_{\rho^k+\p} \psi(3a^{-1}x)\ dx=q\sum_{k=0}^2 \psi(3\rho^k a^{-1}).
\]
We arrive at the following evaluation.
\begin{lemma}  The orbital integral for the orbit $\begin{pmatrix} &   a^{-2}  \\ aI_2 &    \end{pmatrix}$, $a\in F^*$, is given by
$$J_1(a)=\begin{cases} \abs{a}^{-2}\sum_{k=0}^2 \psi(3\rho^ka^{-1})&\abs{a}<1\\
1&\abs{a}=1\\
0&\abs{a}>1.
\end{cases}$$
\end{lemma}

We turn to the evaluation of the orbital integral for the orbit $\left(\begin{smallmatrix}  & aI_2  \\ a^{-2}   & \end{smallmatrix}\right)$ with $a\in F^*.$  In this case modding out by the stabilizer we may take $x_1=0$.
Then the integral of concern is
$$J_2(a):=\int_{F^5}f_0\left(\begin{pmatrix} -a^{-2}z_1&a-a^{-2}x_2z_1&ay_2-a^{-2}z_1z_2\\-a^{-2}y_1&-a^{-2}x_2y_1&a-a^{-2}y_1z_2\\a^{-2}&a^{-2}x_2&a^{-2}z_2\end{pmatrix}\right)
\psi(-y_1+x_2+y_2)\,d(*).$$
Since $f_0$ is supported in $K$, we see that $J_2(a)=0$ unless $|a|\geq1$ and in this case the domain of integration is determined by the conditions 
$$\abs{x_2},\abs{y_1},\abs{z_2},\abs{z_1},\abs{x_2y_1}\leq\abs{a}^2,\abs{a-a^{-2}y_1z_2},\abs{a-a^{-2}x_2z_1},\abs{ay_2-a^{-2}z_1z_2}\leq1.$$

If $\abs{a}=1$ then the domain of integration is $\O^5$, $n_1,n_2\in N\cap K$, and the integrand is identically $1$, so the integral is $1$.

Suppose from now on that $\abs{a}>1$. Then since $\abs{a-a^{-2}y_1z_2}\leq1$ we must have $\abs{y_1z_2}=\abs{a}^3$ and since $\abs{a-a^{-2}x_2z_1}\leq1$
we must have $\abs{x_2z_1}=\abs{a}^3$. But then $\abs{x_2y_1z_1z_2}=\abs{a}^6$. Since $\abs{x_2y_1}\leq\abs{a}^2$ this gives $\abs{z_1z_2}\geq\abs{a}^4$. Since
$\abs{z_1},\abs{z_2}\leq\abs{a}^2$ we conclude that $\abs{z_1}=\abs{z_2}=\abs{a}^2$, $\abs{x_2}=\abs{y_1}=\abs{a}$.  
The last condition now also implies that $\abs{y_2}=\abs{a}$.

We conclude that the domain may be rewritten
\begin{equation}\label{domain-2345}
\abs{z_1}=\abs{z_2}=\abs{a}^2, y_1\in a^3z_2^{-1}+\O, x_2\in a^3z_1^{-1}+\O, y_2\in a^{-3}z_1z_2+a^{-1}\O,
\end{equation}
and on this domain we have
\begin{equation}\label{character-2345}
\psi(-y_1+x_2+y_2)=\psi(a^3z_1^{-1}-a^3z_2^{-1}+a^{-3}z_1z_2).
\end{equation}

The next step is the evaluation of $f_0$ in this domain.  
\begin{lemma}\label{second-kappa}  In domain \eqref{domain-2345} we have
$$\kappa\left(\begin{pmatrix} -a^{-2}z_1&a-a^{-2}x_2z_1&ay_2-a^{-2}z_1z_2\\-a^{-2}y_1&-a^{-2}x_2y_1&a-a^{-2}y_1z_2\\a^{-2}&a^{-2}x_2&a^{-2}z_2\end{pmatrix}\right)=
(a,z_1^{-1}z_2)_3(z_1^{-1},z_2)_3.$$
\end{lemma}

\begin{proof} The proof is similar to the proof of Lemma~\ref{first-kappa}, but in place of $g_1,g_2$ there we use
$$g_1=\begin{pmatrix} -ay_2z_2^{-1}&a-ax_2y_2z_2^{-1}&0\\-az_2^{-1}&-x_2z_2^{-1}a&0\\a^{-2}&a^{-2}x_2&a^{-2}z_2\end{pmatrix},\quad
g_2=\begin{pmatrix}1&&\\&1&\\-z_2^{-1}&-x_2z_2^{-1}&1\end{pmatrix}.$$
Since $g_1g_2$ is block diagonal, one sees that $\kappa(g_1g_2)=(a,x_2z_2^{-1})_3(x_2,z_2)_3$ while again $\kappa(g_2)=1$ and so the value we require is 
$(a,x_2z_2^{-1})_3(x_2,z_2)_3\sigma(g_1,g_2)^{-1}$. For this case, a computation using the algorithm of Bump-Hoffstein shows 
that $\sigma(g_1,g_2)=(a,z_2)_3$.  Using the conditions \eqref{domain-2345}, the Lemma follows.
\end{proof}

Using this expression and performing the integrations over $x_2$, $y_1$ and $y_2$ (after using \eqref{character-2345} the integrand is independent of these variables), we obtain
$$J_2(a)=\abs{a}^{-1}\int_{\abs{z_1}=\abs{z_2}=\abs{a}^2} (a,z_1z_2^{-1})_3(z_1,z_2)_3\,  \psi(a^3z_1^{-1}-a^3z_2^{-1}+a^{-3}z_1z_2)\,dz_1\,dz_2.$$
Changing $z_i$ to $a^2z_i$ for $i=1,2$, we get
$$J_2(a)=\abs{a}^{3}\int_{\O^*\times\O^*} (z_1,az_2)_3 (a,z_2)_3 \psi(a(z_1^{-1}-z_2^{-1}+z_1z_2))\,dz_1\,dz_2.$$

We again make use of Corollary~\ref{cor di}.  Now we replace the integral in $z_1$  with the integral of a cubic 
(the parameters $c,d,t$ of the Corollary are given by $c=a z_2$, $d=a$, $t=a^2z_2^2$; note that $(t,c^{-1}d)_3=(a,z_2)_3$).  
Just as in the prior case, the remaining cubic residue symbols cancel. We see that
\begin{multline*}
J_2(a)=\abs{a}^{3}\int_{\O^*} \Cu(a,-3^{-3}az_2^{-1})\, \psi(-az_2^{-1})dz_2\\
=\abs{a}^3\int_{\O^*}\int_{\O} \psi(ax-3^{-3}az_2^{-1}x^3-az_2^{-1})\,dx\,dz_2.
\end{multline*}
Sending $x\to 3x$, $z_2\to -z_2^{-1}$ and reordering the integral gives
$$J_2(a)=
\abs{a}^3\int_{\O}\psi(3ax)\int_{\O^*} \psi(az_2(1+x^3))\,dz_2\,dx.$$
The evaluation proceeds as in the prior case of computing $J_1$, with contributions whenever either $x\in -\rho^k+a^{-1}\O$ or $x\in -\rho^k+\varpi^{-1}a^{-1}\O^*$, $k=0,1,2$. We obtain the following evaluation.
\begin{lemma}  The orbital integral for the orbit $\begin{pmatrix}  & aI_2  \\ a^{-2}   & \end{pmatrix}$, $a\in F^*$, is given by
$$J_2(a)=\begin{cases} \abs{a}^{2}\sum_{k=0}^2 \psi(-3\rho^ka)&\abs{a}>1\\
1&\abs{a}=1\\
0&\abs{a}<1.
\end{cases}$$
\end{lemma}

\subsection{Comparison}
The comparison is given by the following Theorem

\begin{theorem}\label{degenerate equality} The relevant $I$ and $J$ orbitals integrals match up to transfer factors:
\begin{enumerate}
\item The orbital integrals attached to the 3 singleton relevant orbits for $I$ and for $J$ each have value $1$.
\item Let $a=3(\rho-\rho^2)c^{-1}.$ Then $I_1(a)=\abs{c}^{2}\psi(-3\rho c^{-1})J_1(c).$
\item Let $a=3(\rho-1)c$. Then
$I_2(a)=\abs{c}^{-2}\psi(3\rho c) J_2(c)$.
\end{enumerate}
\end{theorem}
The Theorem follows at once from the evaluations above.  We remark that there are other possible matchings available. First, for all $c\in F^*$ we have
$J_i(\rho^k c)=J_i(c)$ for $i=1,2$ and $k=1,2$. In addition, we have the following matchings.
\begin{enumerate}
\item 
If $a=3(\rho-1)c$ then $I_1(a)=\abs{c}^{-2}\psi(3c) J_2(c)$.
\item If $a=3(\rho^2-1)c^{-1}$ then $I_2(a)=\abs{c}^2 \psi(-3c^{-1}) J_1(c)$.
\end{enumerate}

This completes the proof of the Fundamental Lemma.

\def\cprime{$'$} \def\Dbar{\leavevmode\lower.6ex\hbox to 0pt{\hskip-.23ex
  \accent"16\hss}D} \def\cftil#1{\ifmmode\setbox7\hbox{$\accent"5E#1$}\else
  \setbox7\hbox{\accent"5E#1}\penalty 10000\relax\fi\raise 1\ht7
  \hbox{\lower1.15ex\hbox to 1\wd7{\hss\accent"7E\hss}}\penalty 10000
  \hskip-1\wd7\penalty 10000\box7}
  \def\polhk#1{\setbox0=\hbox{#1}{\ooalign{\hidewidth
  \lower1.5ex\hbox{`}\hidewidth\crcr\unhbox0}}} \def\dbar{\leavevmode\hbox to
  0pt{\hskip.2ex \accent"16\hss}d}
  \def\cfac#1{\ifmmode\setbox7\hbox{$\accent"5E#1$}\else
  \setbox7\hbox{\accent"5E#1}\penalty 10000\relax\fi\raise 1\ht7
  \hbox{\lower1.15ex\hbox to 1\wd7{\hss\accent"13\hss}}\penalty 10000
  \hskip-1\wd7\penalty 10000\box7}
  \def\ocirc#1{\ifmmode\setbox0=\hbox{$#1$}\dimen0=\ht0 \advance\dimen0
  by1pt\rlap{\hbox to\wd0{\hss\raise\dimen0
  \hbox{\hskip.2em$\scriptscriptstyle\circ$}\hss}}#1\else {\accent"17 #1}\fi}
  \def\bud{$''$} \def\cfudot#1{\ifmmode\setbox7\hbox{$\accent"5E#1$}\else
  \setbox7\hbox{\accent"5E#1}\penalty 10000\relax\fi\raise 1\ht7
  \hbox{\raise.1ex\hbox to 1\wd7{\hss.\hss}}\penalty 10000 \hskip-1\wd7\penalty
  10000\box7} \def\lfhook#1{\setbox0=\hbox{#1}{\ooalign{\hidewidth
  \lower1.5ex\hbox{'}\hidewidth\crcr\unhbox0}}}
\providecommand{\bysame}{\leavevmode\hbox to3em{\hrulefill}\thinspace}
\providecommand{\MR}{\relax\ifhmode\unskip\space\fi MR }
\providecommand{\MRhref}[2]{%
  \href{http://www.ams.org/mathscinet-getitem?mr=#1}{#2}
}
\providecommand{\href}[2]{#2}

\end{document}